\documentclass[11pt,a4paper,reqno]{amsart}
\usepackage[T1]{fontenc}
    \DeclareMathSizes{12}{12}{7}{6}

\usepackage{amssymb}
\usepackage{braket}
\usepackage{mathrsfs}
\usepackage{graphicx}
\usepackage{transparent}
\usepackage{amsthm}
\usepackage{geometry}
\usepackage{xcolor}
\usepackage[shortlabels]{enumitem}
\usepackage{mathtools}
\usepackage{setspace}
\usepackage{tikz}
\usepackage{tikz-cd}

\usepackage[hyperfootnotes=false
,hidelinks
]{hyperref}

\usepackage{wasysym} 
\usepackage{stmaryrd} 
\makeatletter

\@addtoreset{equation}{section}
\makeatother
\newtheorem{theorem}{Theorem}[section]
    \newtheorem*{theorem*}{Main Theorem}
    \newtheorem{corollary}[theorem]{Corollary}
    \newtheorem*{corollary*}{Corollary}
    \newtheorem{proposition}[theorem]{Proposition}
    \newtheorem{lemma}[theorem]{Lemma}

    \newtheorem{mainthm}{Theorem}

\theoremstyle{definition}
    \newtheorem{definition}[theorem]{Definition}
    \newtheorem{setup}[theorem]{Setup}
    
    \newtheorem{remark}[theorem]{Remark}
	\newtheorem{notation}[theorem]{Notation}
    \newtheorem*{setup*}{Setup}
    \newtheorem*{definition*}{Definition}
    \newtheorem*{remark*}{Remark}

    \newtheorem{example}[theorem]{Example}

\def\Ob{\operatorname{Ob}}
\def\Ext{\operatorname{Ext}}

\def\mod{\operatorname{\mathsf{mod}}}

\def\Inj{\operatorname{\mathsf{-Inj}}}

\def\Ab{\mathsf{Ab}}
\def\add{\operatorname{\mathsf{add}}}

\newcommand{\op}{\mathsf{op}}
\renewcommand{\sp}{\mathsf{sp}}
\newcommand{\A}{\mathscr{A}}
\newcommand{\B}{\mathscr{B}}
\newcommand{\C}{\mathscr{C}}
\newcommand{\D}{\mathscr{D}}
\newcommand{\E}{\mathcal{E}}

\newcommand{\I}{\mathcal{I}}
\newcommand{\K}{\mathcal{K}}
\renewcommand{\L}{\mathcal{L}}
\newcommand{\M}{\mathscr{M}}
\newcommand{\N}{\mathscr{N}}

\newcommand{\Q}{\mathcal{Q}}
\newcommand{\R}{\mathcal{R}}
\newcommand{\T}{\mathcal{T}}
\newcommand{\U}{\mathcal{U}}
\newcommand{\V}{\mathcal{V}}

\newcommand{\X}{\mathscr{X}}

\newcommand{\Ker}{\operatorname{Ker}}
\renewcommand{\Im}{\operatorname{Im}}
\newcommand{\Cok}{\operatorname{Cok}}

\newcommand{\id}{\mathsf{id}}

\newcommand{\Sn}{\mathscr{S}_\N}

\def\cone{\operatorname{\mathsf{Cone}}}
\def\cocone{\operatorname{\mathsf{CoCone}}}
\def\fib{\operatorname{\mathsf{F}}}
\def\cof{\operatorname{\mathsf{C}}}
\def\weq{\operatorname{\mathsf{W}}}
\def\veq{\operatorname{\mathsf{V}}}

\def\Iso{\operatorname{\mathsf{Iso}}}
\def\isoclass{\operatorname{\mathsf{isoclass}}}
\def\seq{\operatorname{\mathsf{Seq}}}
\def\ind{\operatorname{\mathsf{ind}}}
\def\Ho{\operatorname{\mathsf{Ho}}}

\newcommand\ET{\mathsf{Extri}}
\newcommand\TR{\mathsf{Tri}}
\newcommand\EX{\mathsf{Exa}}
\newcommand\AB{\mathsf{Abel}}
\newcommand\Wald{\mathsf{Wald}}
\newcommand\wWald{\mathsf{wWald}}

 \newcommand{\lra}{\longrightarrow}    
 \newcommand{\dra}{\dashrightarrow}    

\let\overlinedtocsubsection=\tocsubsection

\renewcommand{\tocsubsection}[2]{\hspace{21pt}\overlinedtocsubsection{#1}{#2}}
	\newcommand{\deff}{\coloneqq}
	\newcommand{\sse}{\subseteq}
	\newcommand{\rtail}{\rightarrowtail}
	\newcommand{\into}{\hookrightarrow}
	\newcommand{\onto}{\twoheadrightarrow}
        \newcommand{\iso}{\cong}
	
	\newcommand{\cok}{\operatorname{cok}\nolimits}
        \newcommand{\ac}{\mathsf{ac}}
        \renewcommand{\sec}{\mathsf{sec}}
        \newcommand{\ret}{\mathsf{ret}}
        \newcommand{\inc}{\mathsf{inc}}

    \newcommand{\fs}{\mathfrak{s}}
    \newcommand{\ft}{\mathfrak{t}}

	\newcommand{\BE}{\mathbb{E}}
	\newcommand{\BF}{\mathbb{F}}

        \newcommand{\BW}{\mathbb{W}}
        \newcommand{\BY}{\mathbb{Y}}
	\newcommand{\BZ}{\mathbb{Z}}

	\newcommand{\CE}{\mathcal{E}}

\newcommand{\bincoprod}{\mathbin\amalg}

	\usetikzlibrary{matrix,arrows,decorations.pathmorphing,positioning,decorations.pathreplacing}
	\tikzset{commutative diagrams/.cd, 
		mysymbol/.style = {start anchor=center, end anchor = center, draw = none}}
    \tikzset{
    labl/.style={anchor=north, rotate=90, inner sep=1mm}
    }
    \tikzcdset{scale cd/.style={every label/.append style={scale=#1},
    cells={nodes={scale=#1}}}}

	\let\amph=& 
	
	\tikzcdset{every label/.append style = {font = \footnotesize}}

\newcommand{\PO}[1]{\arrow[mysymbol]{#1}[description]{\mathrm{(PO)}}}
\newcommand{\wPO}[1]{\arrow[mysymbol]{#1}[description]{\mathrm{(wPO)}}}

\newcommand{\wPB}[1]{\arrow[mysymbol]{#1}[description]{\mathrm{(wPB)}}}

\makeatletter

 \def\@setaddresses{\par
     \nobreak \begingroup
     \setstretch{1} 
     \setlength{\parindent}{0cm}
     \footnotesize
     \def\author##1{\nobreak\addvspace\bigskipamount}%
     \def\\{\unskip, \ignorespaces}%
     \interlinepenalty\@M
     \def\address##1##2{\begingroup
         \par\addvspace\bigskipamount
     \@ifnotempty{##1}{(\ignorespaces##1\unskip) }%
     {\scshape\ignorespaces##2}\par\endgroup}%
     \def\curraddr##1##2{\begingroup
     \@ifnotempty{##2}{\nobreak\curraddrname
     \@ifnotempty{##1}{, \ignorespaces##1\unskip}\/:\space
     ##2\par}\endgroup}%
     \def\email##1##2{\begingroup
     \@ifnotempty{##2}{\nobreak\emailaddrname
     \@ifnotempty{##1}{, \ignorespaces##1\unskip}\/:\space
     \ttfamily##2\par}\endgroup}%
     \def\urladdr##1##2{\begingroup
     \def~{\char'\~}%
     \@ifnotempty{##2}{\nobreak\urladdrname
     \@ifnotempty{##1}{, \ignorespaces##1\unskip}\/:\space
     \ttfamily##2\par}\endgroup}%
     \addresses
     \endgroup
 }

\makeatother
\begin{document}
\title{Weak Waldhausen categories and a localization theorem}

\author[Ogawa]{Yasuaki Ogawa}
	\address{Faculty of Engineering Science, Kansai University, Suita Osaka 564-8680, Japan}
	\email{y\_ogawa@kansai-u.ac.jp} 

\author[Shah]{Amit Shah}
	\address{Department of Mathematics, City St George's, University of London, London EC1V 0HB, UK}
	\email{amit.shah.3@city.ac.uk}



\begin{abstract}
Waldhausen categories were introduced to extend algebraic $K$-theory beyond Quillen's exact categories. In this article, we modify Waldhausen's axioms so that it matches better with the theory of extriangulated categories, introducing a \emph{weak Waldhausen} category and defining its Grothendieck group. Examples of weak Waldhausen categories include any extriangulated category, hence any exact or triangulated category, and any Waldhausen category. A key feature of this structure is that it allows for ``one-sided'' extriangulated localization theory, and thus enables us to extract right exact sequences of Grothendieck groups that we cannot obtain from the theory currently available.

To demonstrate the utility of our Weak Waldhausen Localization Theorem, we give three applications. First, we give a new proof of the Extriangulated  Localization Theorem proven by Enomoto--Saito, which is a generalization at the level of $K_0$ of Quillen's classical Localization Theorem for exact categories. Second, we give a new proof that the index with respect to an $n$-cluster tilting subcategory $\mathscr{X}$ of a triangulated category $\mathscr{C}$ induces an isomorphism between $K_0^{\mathsf{sp}}(\mathscr{X})$ and the Grothendieck group of an extriangulated substructure of $\mathscr{C}$. Last, we produce a weak Waldhausen $K_0$-generalization of a localization construction due to Sarazola that involves cotorsion pairs but allows for non-Serre localizations. We show that the right exact sequences of Grothendieck groups obtained from our Sarazola construction and the Extriangulated Localization Theorem agree under a common setup.
\end{abstract}

\maketitle

{\footnotesize
\tableofcontents
}

\setlength{\baselineskip}{15pt}
\section{Introduction}
\label{sec:intro}

The idea behind $K$-theory is to assign invariants, known as $K$-groups, to an object $X$ of interest, e.g.\ a ring, topological space, category etc. 
Although this purpose of $K$-groups is modestly stated, they can (if well-defined) reflect key properties of $X$ and in doing so have been used to answer significant questions. Indeed, $K$-theory has a broad spectrum of connections to other fields, including algebraic geometry, topology, number theory, ring theory, and string theory (see e.g.\ \cite{Wei99}, \cite[Preface]{FG05}, \cite[p.~x]{Wei13}).

The $K$-theory of a ring $R$ begins with its Grothendieck group $K_{0}(R)$ (due to Grothendieck \cite{SGA6}), the Whitehead group $K_{1}(R)$ (due to Bass--Schanuel \cite{BS62}), and the group $K_{2}(R)$ (due to Milnor \cite{Mil71}), all defined using linear algebra and basic group theory. 
However, it turned out that producing a suitable notion of $K_{i}(R)$ for $i\geq 3$ using a similar approach was somewhat harder to see at that time.

This was impressively overcome by Quillen through his $+$-construction for rings \cite{Qui71a}. Firstly, his work used geometric and topological methods, and secondly he produced higher $K$-groups $K_{i}(R)$ ($i \geq 0$) not one at a time, but rather all at once as the homotopy groups of a classifying space. In further groundbreaking work \cite{Qui73}, Quillen laid out his $Q$-construction, which resulted in a unified and widely accepted higher algebraic $K$-theory for exact categories. 
This was the beginning of higher algebraic $K$-theory and Quillen's efforts in this direction earned him the Fields Medal in $1978$.

However, there was one drawback: computation in algebraic $K$-theory was generally challenging. 
To mitigate this, Quillen gave several reduction results \cite[Thms.~3, 4, 5]{Qui73}, allowing one to determine the $K$-theory of more complicated categories via simpler ones. 
The Localization Theorem \cite[Thm.~5]{Qui73} is one of the most powerful. It yields a long exact sequence 
\[
\begin{tikzcd}
\cdots
    \arrow{r}
& K_{1}(\A/\B)
    \arrow{r}
& K_{0}(\B)
    \arrow{r}
& K_{0}(\A)
    \arrow{r}
& K_{0}(\A/\B)
    \arrow{r}
& 0
\end{tikzcd}
\]
of $K$-groups from an exact sequence $\B \into \A \to \A/\B$ of abelian categories (see Definition~\ref{def:exact_seq_ET}), where $\B$ is a Serre subcategory of $\A$ and $\A/\B$ is the Serre quotient (i.e.\ a certain localization of $\A$). 
This extended the right exact sequence of only Grothendieck groups given by Bass \cite[Ch.~VIII, Cor.~5.5]{Bas68} (see also \cite{Hel65}) to the left using higher $K$-groups. 
A generalization of Quillen's result for exact categories was proven by C\'{a}rdenas-Escudero \cite[Thm.~7.0.63]{C-E98}.

Despite the wide scope of Quillen's $K$-theory, it was not sufficiently general for some applications $K$-theorists had in mind, e.g.\ where triangulated categories come into play. 
To tackle this problem, Waldhausen \cite{Wal85} introduced a generalization of exact categories, now called  Waldhausen categories (see Definition~\ref{def:Waldhausen_category}). Waldhausen even gave non-additive examples arising from topology \cite[Exam.~II.9.1.4]{Wei13}. 
More importantly, it was shown that an algebraic $K$-theory still exists for Waldhausen categories and a Localization Theorem \cite[Thm.~1.6.4]{Wal85} was also proved.

As a consequence of Waldhausen's theory, 
Thomason--Trobaugh \cite{TT90} established a $K$-theory for derived categories (of complicial biWaldhausen categories), as well as a Localization Theorem for the $K$-theory of schemes that was a substantial improvement on previous results in the literature; and later, Schlichting produced a $K$-theory for algebraic triangulated categories \cite{Sch06}.
In particular, in analogy with Quillen's \cite[Thm.~5]{Qui73} recalled above, the exact sequence of triangulated categories associated to the Verdier quotient of an algebraic triangulated category yields a long exact sequence of $K$-groups.

In this paper, we bring together several of the above ideas from the modern perspective of Nakaoka--Palu's extriangulated categories \cite{NP19} (see Section~\ref{sec:ET}), which give a simultaneous generalization of exact and triangulated categories. 
See Figure~\ref{fig:summary} below for a summary of the interactions we establish between these ideas. 
A localization theory for extriangulated categories was developed in \cite{NOS22} by Nakaoka--Ogawa--Sakai, which recovers classical localizations such as the Serre and Verdier quotients of abelian and triangulated categories, respectively (see Section~\ref{sec_localization}).
Our first main result (see Theorem~\ref{ThmA} below) is motivated by Schlichting's cylinder-free Localization Theorem \cite[Thm.~11]{Sch06} for Waldhausen categories. 
However, in order to draw from the theories of Waldhausen and extriangulated categories at the same time, we introduce the notion of a \emph{weak Waldhausen} (additive) category $(\C,\seq,\weq)$ and define its Grothendieck group $K_{0}(\C,\seq,\weq)$. 
Just like how a Waldhausen category is a category $\C$ equipped with classes $\cof$ and $\weq$ of cofibrations and weak equivalences, the data of a weak Waldhausen category includes classes $\seq$ and $\weq$ of distinguished $3$-term sequences and weak equivalences; see Definition~\ref{def:Waldhausen_category} and Definition~\ref{def:weak_Wald_cat}. 
An (additive) Waldhausen category $(\C,\cof,\weq)$ gives rise canonically to a weak Waldhausen category $(\C,\seq,\weq)$, where $\seq$ is the class of cofibration sequences of $(\C,\cof,\weq)$ (see Proposition~\ref{prop:Wald_vs_wWald}).
The following Localization Theorem for weak Waldhausen categories summarises our Theorem~\ref{thm:Quillen_localization1_new} and Corollary~\ref{cor:Quillen_localization2}.

\begin{mainthm}
\label{ThmA}
Suppose $(\C,\seq,\veq)$ and $(\C,\seq,\weq)$ are weak Waldhausen categories with $\veq\sse\weq$, and let $(\C^{\weq},\seq^{\weq},\veq^{\weq})$ be the weak Waldhausen subcategory of $\weq$-acyclic objects.
If 
\[
(\C^{\weq},\seq^{\weq},\veq^{\weq})\to (\C,\seq,\veq)\to (\C,\seq,\weq)
\]
is a localization sequence (see Definition~\ref{def:exact_sequence_wWald}) 
of weak (bi)Waldhausen categories, 
such that any weak equivalence in $\weq$ admits a suitable factorization, 
then we have an exact sequence of their Grothendieck groups as follows.
\begin{equation}\label{seq:ThmA}
\begin{tikzcd}
    K_{0}(\C^{\weq},\seq^{\weq},\veq^{\weq}) 
        \arrow{r}
    & K_{0}(\C,\seq,\veq)
        \arrow{r}
    & K_{0}(\C,\seq,\weq)
        \arrow{r}
    & 0
\end{tikzcd}
\end{equation}
\end{mainthm}

As a first application of Theorem~\ref{ThmA}, we prove that 
an extriangulated localization sequence (see Theorem~\ref{thm:NOS} and Remark~\ref{rem:on-theorem-4-14}\ref{item:extri_localization}) naturally gives rise to a localization sequence of weak Waldhausen categories which induces a right exact sequence \eqref{seq:ThmA}.
Consequently, we recover and, through the lens of weak Waldhausen categories, give a new proof of Enomoto--Saito's \cite[Cor.~4.32]{ES22}, which is a Localization Theorem for extriangulated categories; see Corollary~\ref{cor:ES_via_Waldhausen}.

An important aspect of the viewpoint we adopt here is that weak Waldhausen categories allow us to consider ``one-sided'' localizations that are beyond the scope of the extriangulated localization theory of \cite{NOS22}. This viewpoint was taken in \cite{Oga22a, Oga22b}, unifying several important constructions in the literature e.g.\ \cite{BBD,KZ08, KR07,Bel13,BM12,BM13}. In Section~\ref{sec:application-to-tri-cat}, 
we prove Theorem~\ref{thm:ex_seq_wWald_from_abel_loc} and exploit this one-sided localization perspective to understand the index in triangulated categories.

The index with respect to a $2$-cluster tilting subcategory of a triangulated category $\C$ was introduced by Palu \cite{Pal08}, and this definition was generalized to $n$-cluster tilting subcategories $\X$ in \cite{Jor21} for $n\geq 2$ an integer. The importance of the index is evidenced, for instance, by its use in the theory of 
cluster algebras \cite{Pal08,DehyKeller-On-the-combinatorics-of-rigid-objects-in-2-Calabi-Yau-categories,Pal09,Plam11} and of
friezes \cite{Guo-On-tropical-friezes-associated-with-Dynkin-diagrams,HolmJorgensen-generalized-friezes-and-a-modified-caldero-chapoton-map-depending-on-a-rigid-object-1,HolmJorgensen-generalized-friezes-and-a-modified-caldero-chapoton-map-depending-on-a-rigid-object-2,Jor21,JS21}. 
The index with respect to $\X$ induces an isomorphism between the split Grothendieck group
$K^{\sp}_{0}(\X)$ and the Grothendieck group 
of a certain extriangulated substructure of the triangulated structure on $\C$; see Corollary~\ref{cor:JS_index} or \cite[Cor.~5.5]{OS23}, and also \cite[Thm.~B]{JS23}, \cite[Prop.~4.11]{PPPP19}. 
Using Theorem~\ref{thm:ex_seq_wWald_from_abel_loc}, and hence as a second application of Theorem~\ref{ThmA}, we give a proof of this index isomorphism using some reduction-style techniques; see Section~\ref{sec:JS_index}. 
In fact, this process recovers the Extriangulated Resolution Theorem \cite[Thm.~4.5]{OS23}, demonstrating that there is a close connection between the Localization Theorem and the Resolution Theorem, two fundamental $K$-theoretic reduction theorems.

As our third and final application of Theorem~\ref{ThmA}, we generalize to the extriangulated setting a construction due to Sarazola \cite{Sar20}. 
Sarazola \cite[Thm.~3]{Sar20} gives a Localization Theorem for exact categories allowing one to ``localize'' at subcategories that are not necessarily Serre, and this is done using Waldhausen categories and cotorsion pairs. 
We produce a $K_{0}$-version of \cite[Thm.~3]{Sar20} for extriangulated categories and cotorsion pairs on them, using our weak Waldhausen categories; see Proposition~\ref{prop:Sarazola_homotopy_fiber}.

We will encounter exact sequences of extriangulated categories and Grothendieck groups, and localization sequences of (weak) Waldhausen categories in this article. We summarize 
some key ideas and links we make in this article in Figure~\ref{fig:summary}:

\begin{figure}[h]
\begin{tikzpicture}
    \node[rectangle, draw=black, text centered,align=center] at (-0.5,3.5) (ABTR) 
    {Serre and Verdier\\quotients};
    \node[rectangle, draw=black, text centered, minimum height=1.0cm, minimum width=1.0cm,align=center] at (-0.5,1) (ET) {
    Exact sequence\\of
    extriangulated\\categories};
  \node[rectangle, draw=black, text centered, minimum height=1.0cm, minimum width=1.0cm,align=center] at (-0.5,-1) (rET) {
    One-sided exact\\localization of\\
    extriangulated categories};
  \node[rectangle, draw=black, text centered, minimum height=1.0cm, minimum width=1cm,align=center] at (6,2.5) (Wald) {
    Localization sequence of\\
    Waldhausen categories};
  \node[rectangle, draw=black, text centered, minimum height=1.0cm, minimum width=1cm,align=center] at (6,0) (wWald) {
    Localization sequence of\\weak Waldhausen\\categories};
  \node[rectangle, draw=black, text centered, minimum height=1.0cm, minimum width=1cm,align=center] at (6,-2.5) (Gr) {
    Right exact\\sequence of\\
    Grothendieck groups};
    \node[rectangle, draw=black, text centered, minimum height=0.5cm, minimum width=1cm,align=center] at (11.5,1.25) (Co) {Cotorsion pairs};
  \draw[->] (ABTR) -- node[right] {\cite{NOS22}} (ET);
  \draw[->] (ET) -- node[above, yshift=0.1cm] {Cor.~\ref{cor:ES_via_Waldhausen}} (wWald);
  \draw[->] (Wald) -- node[right] {Prop.~\ref{prop:Wald_vs_wWald}} (wWald);
  \draw[->] (rET) -- node[below, yshift=-0.1cm] {Thm.~\ref{thm:ex_seq_wWald_from_abel_loc}} (wWald);
  \draw[->] (wWald) -- node[right] {Thm.~\ref{ThmA}} (Gr);
  \draw[->] (Co) -- node[right, yshift=0.3cm, xshift=-0.2cm] {\cite{Sar20}} (Wald);
  \draw[->] (Co) -- node[right, yshift=-0.4cm, xshift=-0.3cm] {Prop.~\ref{prop:Sarazola_homotopy_fiber}} (wWald);
\end{tikzpicture}
\caption{Connections between the key ideas in paper.}
\label{fig:summary}
\end{figure}
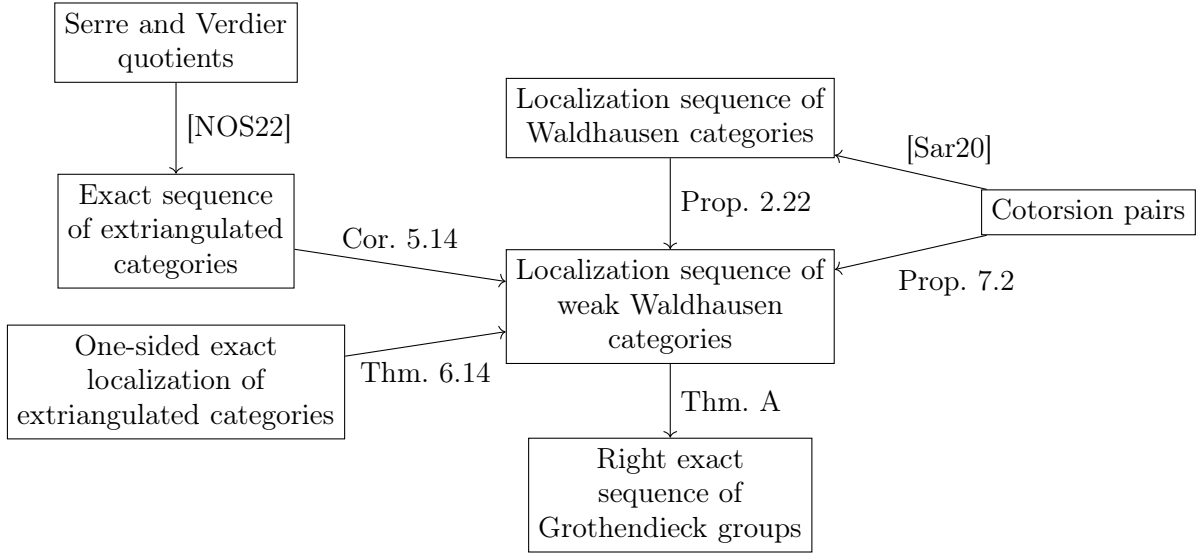

\subsection{Notation and conventions}
\label{sec:notation-convention}

Throughout this article we adopt the following conventions.
For any category $\C$, we use the following notation. 
\begin{itemize}
    \item 
    $\C(A,B)$ is the class of morphisms $A\to B$ for objects $A,B\in\C$.
    
    \item 
    $\C^{\to}$ is the class of all morphisms in $\C$.
    
    \item 
    $\C^{\to\to}$ is the subclass of $\C^{\to} \times \C^{\to}$ consisting of all pairs $(f,g)$ of composable morphisms, indicated by diagrams  
    $A \overset{f}{\to} B \overset{g}{\to} C$.
    
    \item 
    $\Iso\C \sse \C^{\to}$ is the class of all isomorphisms in $\C$.
    
    \item 
    $\isoclass(\C)$ is the class of all isomorphism classes of objects in $\C$. 
\end{itemize}
Subcategories are usually full but not always (e.g.\ see Remark~\ref{rem:Wald_cat_rem2}), so we explicitly indicate this. 
Functors are covariant unless stated otherwise. 

Suppose $\C$ is an additive category for the remainder of Section~\ref{sec:notation-convention}. 
If $f\colon A \to B$ and $g\colon C \to D$ are in $\C^{\to}$, 
then we denote the unique induced morphism 
$A \oplus C \to B\oplus D$ by 
$
f\oplus g 
    = \begin{psmallmatrix}
        f & 0 \\
        0 & g
    \end{psmallmatrix}
$.
A functor between additive categories is always additive.
If $F\colon \C\to \D$ is an additive functor, 
then its \emph{kernel} $\Ker F \sse \C$ 
and 
\emph{essential image} $\Im F \sse \D$ 
are the full subcategories defined as follows: 
\begin{align}
\Ker F &= \Set{C\in\C | FC\cong 0\textnormal{\ in\ }\D },\label{eqn:KerQ}\\
\Im F &= \Set{D\in\D | FC\cong D\textnormal{\ for some\ }C\in\C }.\nonumber
\end{align}

By an \emph{additive} subcategory of $\C$, we mean a full subcategory $\A\sse\C$ that is closed under finite direct sums. 
If $\A\sse\C$ is an additive subcategory, then $[\A]$ is the (two-sided) ideal of morphisms in $\C$ that factor through an object in $\A$, and the canonical additive quotient functor is  $\overline{(-)}\colon \C\to \C/[\A]$. In this case, 
the image in $\C/[\A]$ of a morphism $f\in\C(A,B)$ is denoted $\overline{f}$. 

We denote the abelian category of abelian groups by $\Ab$. 
Lastly, for reference, we list some categories of categories (and where they are introduced) that will often appear in the sequel:
\begin{itemize}[leftmargin=0.5cm]
    \item $\Wald$: the category of skeletally small Waldhausen (additive) categories (Section~\ref{sec:classical_Wald_cat}); 
    \item $\wWald$: the category of skeletally small weak Waldhausen (additive) categories (Section~\ref{sec:weak_Wald_cat}); and 
    \item $\ET$: the category of skeletally small extriangulated categories (Section~\ref{sec:ET}).
\end{itemize}

\section{Weak Waldhausen categories}
\label{sec:Waldhausen_categories}

A Waldhausen category is a category equipped with cofibrations and weak equivalences \cite{Wal85}. In this way, Waldhausen categories are a generalization of exact categories that still permit a higher $K$-theory. In Section~\ref{sec:classical_Wald_cat} we recall the definitions and basic properties of a Waldhausen category.
The axioms for Waldhausen categories rely on the notion of a pushout. Since pushouts do not necessarily exist in an arbitrary extriangulated category, we introduce \emph{weak Waldhausen additive} categories in Section~\ref{sec:weak_Wald_cat} and show that any extriangulated category is indeed an example. 
We introduce the Grothendieck group of a weak Waldhausen additive category in Section~\ref{sec:K0_of_wWald}, and this is used in Section~\ref{sec:localization-theorem-for-wWald} to produce a generalization of Quillen’s localization theorem.

\subsection{Classical Waldhausen categories}
\label{sec:classical_Wald_cat}

In this article we only consider a Waldhausen structure on an \emph{additive} category $\C$. Waldhausen's original definition in \cite{Wal85} only required a pointed category, but nothing is lost for our purposes. 
We refer to \cite[Ch.~II, Sec.~9]{Wei13} for a succinct treatment of the basics of Waldhausen categories.

\begin{definition}
\label{def:Waldhausen_category}
Let $\C$ be an additive category (with a distinguished zero object $0$) equipped with two classes 
$
\cof
$ and  
$\weq$
of morphisms in $\C$.
Morphisms in $\cof$ (resp.\ $\weq$) are called \emph{cofibrations} (resp.\ \emph{weak equivalences}) and indicated with feathered arrows 
$\begin{tikzcd}[column sep=0.5cm,cramped]
    {}\arrow[tail]{r}&{}
\end{tikzcd}$ 
(resp.\ decorated arrows 
$\begin{tikzcd}[column sep=0.5cm,cramped]
    {}\arrow[]{r}{\sim}&{}
\end{tikzcd}$
).

The triplet $(\C,\cof,\weq)$ is a \emph{Waldhausen (additive) category} if the following axioms are satisfied.
\begin{enumerate}[label=\textup{(C\arabic*)}]
\setcounter{enumi}{-1}
\item\label{C0} The class $\cof$ is closed under composition and contains each isomorphism in $\C$.
\item\label{C1} The unique morphism $0\to C$ to any object in $\C$ is a cofibration.
\item\label{C2} For any cofibration $A\overset{f}{\rightarrowtail} B$ and any morphism $A\overset{c}{\lra} C$, there exists a pushout diagram
    \[
    \begin{tikzcd}[row sep=0.5cm]
        A \arrow[tail]{r}{f}\arrow{d}[swap]{c}\PO{dr}& B\arrow{d}{b} \\
        C \arrow[tail]{r}[swap]{g}& B\bincoprod_{A} C
    \end{tikzcd}
    \]
with the morphism $C\overset{g}{\rightarrowtail} B\bincoprod_{A}C$ being a cofibration.
\end{enumerate}
\begin{enumerate}[label=\textup{(W\arabic*)}]
\setcounter{enumi}{-1}
\item\label{W0} The class $\weq$ is closed under composition and contains each isomorphism in $\C$.
\item\label{W1} (Gluing Axiom) Consider a commutative diagram of the form
\[
\begin{tikzcd}[row sep=0.5cm]
    C_{1} 
        \arrow{d}{\sim}[swap]{z}
    & A_{1} 
        \arrow{l}[swap]{c_{1}}
        \arrow[tail]{r}{f_{1}}
        \arrow{d}{\sim}[swap]{x}
    & B_{1}
        \arrow{d}{\sim}[swap]{y} \\
    C_{2} 
    & A_{2} 
        \arrow{l}[swap]{c_{2}}
        \arrow[tail]{r}{f_{2}}
    & B_{2}
\end{tikzcd}
\]
in which all vertical arrows are weak equivalences and the feathered arrows are cofibrations.
Then the induced map $B_{1}\bincoprod_{A_{1}} C_{1}\to B_{2}\bincoprod_{A_{2}}C_{2}$ is also a weak equivalence.
\end{enumerate}
A Waldhausen category $(\C,\cof,\weq)$ is simply denoted by $\C$ if there is no confusion.
\end{definition}

\begin{remark}
\label{rem:Wald_cat_rem1}
\begin{enumerate}[label=\textup{(\arabic*)}]

    \item\label{item:consequences-of-c0} Axiom \ref{C0} has two immediate consequences. First, it implies $\cof$ is closed under isomorphism. 
    Second, in the context of \ref{C2}, any pushout of $f$ along $c$ is a cofibration. That is, if 
    \[
    \begin{tikzcd}[row sep=0.5cm]
        A \arrow[tail]{r}{f}\arrow{d}[swap]{c}\PO{dr}& B\arrow{d} \\
        C \arrow{r}[swap]{g'}& P
    \end{tikzcd}
    \]
    is a pushout square in $\C$, then $g' \colon C \rtail P$ is also a cofibration.
    
    \item\label{item:consequences-of-w0} Similarly, axiom \ref{W0} also has some nice consequences. Indeed, $\weq$ is also closed under isomorphism; and, in the setup of \ref{W1}, the canonical morphism from any pushout of $f_{1}$ along $c_{1}$ to any pushout of $f_{2}$ along $c_{2}$ is a weak equivalence (see Lemma~\ref{lem:WW1'}).
	
    \item\label{item:cofib-sequences} Let $A\overset{f}{\rightarrowtail} B$ be a cofibration. Using \ref{C2}, we see that $f$ admits a cokernel in $\C$.     
    Note that any cokernel $g\colon B\to C$ of $f$ satisfies $C\cong0\bincoprod_{A} B$. 
    We call any cokernel sequence $A\overset{f}{\rightarrowtail} B\overset{g}{\twoheadrightarrow} C$ (i.e.\ $g=\cok f$) a \emph{cofibration sequence}, and any cokernel $g$ of $f$ is called a \emph{fibration}.
    Note that the class of cofibration sequences is closed under isomorphisms, but that the class $\fib = \fib_{\cof}$ of fibrations is not closed under composition in general. 
\end{enumerate}
\end{remark}

\begin{example}\label{exam:exact-is-waldhausen}
Any exact category $(\C,\CE)$ can be regarded as a Waldhausen category by declaring that all (admissible) inflations are cofibrations and all isomorphisms are weak equivalences. 
In this case, the collection of all cofibration sequences is $\CE$.
\end{example}

\begin{remark}
\label{rem:Wald_cat_rem2}
In the original language of \cite{Wal85}, $\cof$ and $\weq$ are considered as subcategories of $\C$ that are \emph{not necessarily full}. For instance, from the class $\cof$ of morphisms, we define a subcategory of $\C$, denoted also by $\cof$ by abuse of notation, with $\Ob(\cof) = \Ob(\C)$ and with $\cof(A,B) = \Set{f\in\C(A,B) | f\in\cof}$. 
Note that $\cof$ must contain all objects of $\C$ as $\cof$ contains all identity morphisms of $\C$. 
Similarly, we can think of $\weq$ as a subcategory of $\C$. 
However, $\fib$ cannot be thought of as subcategory of $\C$ since it is not necessarily closed under composition (see Remark~\ref{rem:Wald_cat_rem1}\ref{item:cofib-sequences}).
\end{remark}

We will need the notion of a morphism of Waldhausen categories. 

\begin{definition}\label{def:exact_functor_Wald-cats}
Suppose $(\C,\cof,\weq)$ and $(\C',\cof',\weq')$ are Waldhausen categories. 
\begin{enumerate}
    \item 
An additive functor $F\colon \C \to \C'$ is called \emph{exact} if 
$F(\cof)\sse\cof'$,
$F(\weq)\sse\weq'$,
and, for every cofibration $f\colon A \rtail B$ in $\cof$ and every morphism $c\colon A\to C$ in $\C$, 
the canonical morphism 
$FB \bincoprod_{FA} FC \to F(B\bincoprod_{A}C)$ 
is an isomorphism in $\C'$. 

    \item 
If $\C$ is an additive subcategory of $\C'$ and $\mathsf{inc}\colon \C \into \C'$ is the inclusion functor, then $(\C,\cof,\weq)$ is called a \emph{Waldhausen subcategory} of $(\C',\cof',\weq')$ if:
\begin{itemize}[itemsep=5pt]

	\item $\mathsf{inc}$ is an exact functor of Waldhausen categories;
	
	\item $\cof = \Set{ f\colon A\rtail B | 
            \begin{aligned}
                &A,B\in\C\text{ and }f\in\cof'\text{, and if }A\overset{f}{\rtail} B \onto C\text{ is a} \\
                &\text{cofibration sequence in }\C'\text{, then }C\in\C
            \end{aligned}
            }$; and 
	
	\item $\weq = \weq' \cap \C^{\to} = \Set{ f\colon \begin{tikzcd}[column sep=0.5cm,ampersand replacement=\&, cramped]A\arrow{r}{\sim}\& B\end{tikzcd} | A,B\in\C\text{ and }f\in\weq'}$.

\end{itemize}
\end{enumerate}
Note that the composition of exact functors is exact as well.
Thus, we denote by $\Wald$ the category of skeletally small Waldhausen categories and exact functors.
\end{definition}

The following additional axioms are standard in classical Waldhausen theory (see \cite[p.~327]{Wal85}) and will be used to define localization sequences in $\Wald$ (see Definition~\ref{def:exact_sequence_Wald}).

\begin{definition}
\label{def:saturation_extension}
Let $(\C,\cof,\weq)$ be a Waldhausen category.
\begin{enumerate}[label=\textup{(\arabic*)}]

\item\label{item_saturation}
    \emph{Saturation axiom}: $\weq$ satisfies the $2$-out-of-$3$ property for compositions of morphisms; namely, for any composite morphism $gf$, if any two of $f, g, gf$ belong to $\weq$, then so does the third.
    
\item\label{item_extension}
    \emph{Extension axiom}: For any commutative diagram
    \[
    \begin{tikzcd}[row sep=0.5cm]
    A_{1} \arrow{d}{a}[swap]{\sim}\arrow[tail]{r}& B_{1} \arrow[two heads]{r}\arrow{d}{b} & C_{1} \arrow{d}{c}[swap]{\sim}\\
    A_{2}\arrow[tail]{r} & B_{2} \arrow[two heads]{r} & C_{2}
    \end{tikzcd}
    \]
    of cofibration sequences in $\C$, if $a,c\in\weq$ then $b\in\weq$.
    
\end{enumerate}
\end{definition}

For later use, we prove the following easy lemma.

\begin{lemma}
\label{lem:Wald_contain_split}
Any split exact sequence is a cofibration sequence in a Waldhausen category $(\C,\cof,\weq)$.
\end{lemma}

\begin{proof}
First, for any $X,Y\in\C$, note that if we pushout the cofibration $0\rtail Y$ along $0\rtail X$, we see that the canonical inclusion $\begin{psmallmatrix}\id_X \\ 0\end{psmallmatrix}\colon X \to X\oplus Y(\cong Y\bincoprod_0 X)$ is a cofibration. 
Now suppose $\begin{tikzcd}[column sep=0.5cm, cramped]A \arrow{r}{f}& B\arrow{r}{g} & C\end{tikzcd}$ is a split exact sequence. 
Then $f$ is isomorphic to $\begin{psmallmatrix}\id_A \\ 0\end{psmallmatrix}$ by the Splitting Lemma \cite[Lem.~2.7]{Sha19}, so $f$ is also a cofibration by Remark~\ref{rem:Wald_cat_rem1}\ref{item:consequences-of-c0}. 
Since $g = \cok f$, we have that $\begin{tikzcd}[column sep=0.5cm, cramped]A \arrow{r}{f}& B\arrow{r}{g} & C\end{tikzcd}$ is a cofibration sequence by definition (see Remark~\ref{rem:Wald_cat_rem1}\ref{item:cofib-sequences}).
\end{proof}

\subsubsection{The Grothendieck group and some localization theorems}
\label{sec:K0_of_Wald_cat}

There is a natural definition of the Grothendieck group of a Waldhausen category, which extends to a functor $K_{0}\colon \Wald\to \Ab$. 
For a skeletally small additive category $\C$, we denote the \emph{split Grothendieck group of $\C$} by $K^{\sp}_{0}(\C)$. 
Recall that $K^{\sp}_{0}(\C)$ is the free abelian group generated by symbols $[A]$, one for each $A\in\C$, modulo the subgroup generated by 
elements $[A] - [B] + [C]$, one for each split exact sequence 
$\begin{tikzcd}[column sep=0.5cm, cramped]
    A \arrow{r}& B \arrow{r}& C.
\end{tikzcd}$

\begin{definition}\label{def:K0_of_Wald_cat}
Assume that $(\C,\cof,\weq)$ is a skeletally small Waldhausen category.
The \emph{Grothendieck group} $K_{0}(\C)\deff K_{0}(\C,\cof,\weq)$ is 
$K^{\sp}_{0}(\C)$ modulo the relations:
\begin{itemize}
\item $[C]=[C']$ for each weak equivalence 
$\begin{tikzcd}[column sep=0.5cm, cramped]C\arrow{r}{\sim}&C',\end{tikzcd}$ and 
\item $[B]=[A]+[C]$ for each cofibration sequence $A\rightarrowtail B\twoheadrightarrow C$.
\end{itemize}
\end{definition}

From Example~\ref{exam:exact-is-waldhausen}, it is clear that if $(\C,\E)$ is a skeletally small exact category and $(\C,\cof,\weq)$ the corresponding Waldhausen structure, then $K_{0}(\C,\cof,\weq)$ is the same as the Grothendieck group $K_{0}(\C) \deff K_{0}(\C,\E)$ of the exact category $(\C,\E)$.

In the rest of this subsection, we motivate our main results by recalling two key localization theorems that are classical. The first concerns the localization of an abelian category by a Serre subcategory. We also remark that Quillen generalized Theorem~\ref{thm:Quillen_localization_for_abel_cat} when introducing higher algebraic $K$-theory (see \cite[Sec.~5, Thm.~5]{Qui73}).

\begin{theorem}
\label{thm:Quillen_localization_for_abel_cat}
\cite[Ch.~VIII, Cor.~5.5]{Bas68}
Let $\N\to\C\to\C/\N$ be a Serre quotient of an abelian category $\C$ by a Serre subcategory $\N\subseteq \C$.
Then it induces an exact sequence $K_{0}(\N)\to K_{0}(\C)\to K_{0}(\C/\N)\to 0$ of their Grothendieck groups.
\end{theorem}

Theorem~\ref{thm:Quillen_localization_for_abel_cat} is a special case of a very recent result of Enomoto--Saito \cite[Cor.~4.32]{ES22} about extriangulated localizations. 
We establish a more general version of Theorem~\ref{thm:Quillen_localization_for_abel_cat} in Section~\ref{sec:localization-theorem-for-wWald} and consequently recover \cite[Cor.~4.32]{ES22}; see Corollary~\ref{cor:Quillen_localization2} and Corollary~\ref{cor:ES_via_Waldhausen}. 
Our results are framed in terms of weak Waldhausen additive categories that we introduce in Section~\ref{sec:weak_Wald_cat}. Furthermore, the results are motivated by a theorem very related to Theorem~\ref{thm:Quillen_localization_for_abel_cat} but for Waldhausen categories as we now begin to recall.

\begin{definition}
\label{def:w_acyclic}
\cite[p.~181]{Wei13} 
Let $(\C,\cof,\weq)$ be a Waldhausen category. 
An object $C\in \C$ is called \emph{$\weq$-acyclic} if the zero morphism $0\rightarrowtail C$ belongs to $\weq$.
We denote by $\C^{\weq}$ the full subcategory of $\C$ consisting of all $\weq$-acyclic objects.
\end{definition}

Suppose $\id_{\C}\colon (\C,\cof,\veq) \to (\C,\cof,\weq)$ is an inclusion of a Waldhausen subcategory, i.e.\ 
$(\C,\cof,\veq)$ and $(\C,\cof,\weq)$ are Waldhausen categories such that $\veq \sse \weq$. 
It is known that $\C^{\weq}$ inherits a natural Waldhausen structure from $(\C,\cof,\veq)$, giving rise to a Waldhausen subcategory $(\C^{\weq},\cof^{\weq},\veq^{\weq})$ of $(\C,\cof,\veq)$,
where 
$\cof^{\weq} = \cof \cap (\C^{\weq})^{\to}$, and 
$\veq^{\weq} = \veq \cap (\C^{\weq})^{\to}$; see \cite[p.~181]{Wei13}. 
Note that if $A\overset{f}{\rtail} B \onto C$ is a cofibration sequence in $\C$ with $A,B\in\C^{\weq}$ then it is automatic that $C\in\C^{\weq}$ 
(see Lemma~\ref{lem:localization_sequence_wWald} for details of a similar argument).
In particular, the inclusion functor $\mathsf{inc}\colon \C^{\weq} \into \C$ is an exact functor $(\C^{\weq},\cof^{\weq},\veq^{\weq}) \into (\C,\cof,\veq)$ and 
\begin{equation}
\label{seq:localization_sequence_Wald}
\begin{tikzcd}
(\C^{\weq},\cof^{\weq},\veq^{\weq}) 
    \arrow[hook]{r}{\mathsf{inc}} 
&(\C,\cof,\veq) 
    \arrow[hook]{r}{\id_{\C}} 
&(\C,\cof,\weq)
\end{tikzcd}
\end{equation}
is a sequence of Waldhausen subcategories.

The next definition is inspired by \cite[Thm.~6.3]{Sar20}.

\begin{definition}
\label{def:exact_sequence_Wald}
Suppose that $(\C,\cof,\veq)$ and $(\C,\cof,\weq)$ are Waldhausen categories with $\veq\sse\weq$. 
We call \eqref{seq:localization_sequence_Wald}
a \emph{localization sequence} in $\Wald$ 
if the functor 
$\id_{\C}$ satisfies the following property: 
for each Waldhausen category $(\C',\cof',\weq')$ satisfying the saturation and extension axioms and for each exact functor 
$F\colon (\C,\cof,\veq)\to (\C',\cof',\weq')$ with $F(\C^{\weq})\sse (\C')^{\weq'}$, 
we have that 
$F$ extends to an exact functor $F\colon (\C,\cof,\weq)\to (\C',\cof',\weq')$, i.e.\ $F(\weq)\sse \weq'$.
\[
\begin{tikzcd}[row sep=0.5cm]
(\C^{\weq},\cof^{\weq},\veq^{\weq}) 
    \arrow[hook]{r}{\mathsf{inc}} 
&(\C,\cof,\veq) 
    \arrow[hook]{r}{\id_{\C}} 
    \arrow{d}[swap]{F}
&(\C,\cof,\weq)
    \arrow[dotted]{dl}{F}\\
{}
& (\C',\cof',\weq')
&{}
\end{tikzcd}
\]
\end{definition}

It is natural to ask when the sequence \eqref{seq:localization_sequence_Wald} becomes a localization sequence in $\Wald$.
The following result, which can be regarded as the $K_{0}$-part of Schlichting's cylinder-free localization theorem \cite[Thm.~11]{Sch06}, gives us a partial answer. In this case, the functor $K_{0}\colon \Wald\to\Ab$ sends such a localization sequence to a right exact sequence of Grothendieck groups. 
A direct proof at the level of $K_{0}$ can be found in \cite{Wei13} and, on inspection, one can weaken the hypotheses of \cite[Ch.~II, Thm.~9.6]{Wei13} to obtain the same right exact sequence of Grothendieck groups. 
However, by one of our main results in Section~\ref{sec:localization-theorem-for-wWald}, we also see that there is an associated localization sequence of Waldhausen categories.
The proof of the following is given after Theorem~\ref{thm:Quillen_localization1_new}.

\begin{theorem}
\label{cor:Schlichting_localization}
(cf.\ \cite[Ch.~II, Thm.~9.6]{Wei13}) 
Consider the sequence \eqref{seq:localization_sequence_Wald}.
Assume $(\C,\cof, \weq)$ satisfies the saturation axiom, 
and that every morphism in $\weq$ is the composition of a morphism from $\cof$ followed by a morphism from $\veq$. 
Then \eqref{seq:localization_sequence_Wald} is a localization sequence in $\Wald$ and induces the following right exact sequence in $\Ab$.
\begin{equation*}
\begin{tikzcd}[column sep=1.5cm]
    K_{0}(\C^{\weq},\cof^{\weq},\veq^{\weq})
        \arrow{r}{K_{0}(\inc)}
    & K_{0}(\C,\cof,\veq)
        \arrow{r}{K_{0}(\id_{\C})}
    & K_{0}(\C,\cof,\weq)
        \arrow{r}
    & 0
\end{tikzcd}
\end{equation*}
\end{theorem}

\subsection{Weak Waldhausen categories}
\label{sec:weak_Wald_cat}

Since the conflations of an extriangulated category are not kernel-cokernel pairs (but only weak-kernel-weak-cokernel pairs) in general, but cofibration sequences in Waldhausen categories are cokernel sequences, 
we cannot expect to equip an extriangulated category with a Waldhausen structure in general.
We propose a more general framework in this subsection, which is fundamental in the remainder of the article.

Although we will compare the theory developed in this subsection to that of Section~\ref{sec:classical_Wald_cat} (see Proposition~\ref{prop:Wald_vs_wWald} and Remark~\ref{rem:wWH-WH-comparison}), 
we delay providing examples that are not Waldhausen until Section~\ref{sec:wWET-cats} in which extriangulated category theory will arm us with novel examples.

\begin{definition}
\label{def:weak_Wald_cat}
Let $\C$ be an additive category (with a distinguished zero object $0$) equipped with a class $\seq\sse\C^{\to\to}$ of \emph{distinguished sequences} of the form 
\begin{equation}
\label{eqn:dist-sequence}
\begin{tikzcd}
A \arrow{r}{f}& B \arrow{r}{g}& C
\end{tikzcd}
\end{equation}
in $\C$, and a class $\weq$ of morphisms in $\C$. 
Denote by $\cof = \cof_{\seq}$ (resp.\ $\fib = \fib_{\seq}$) the class consisting of morphisms $f$ (resp.\ $g$) such that there is a distinguished sequence of the form \eqref{eqn:dist-sequence}. 
The morphisms in $\cof$ (resp.\ $\fib)$ are called \emph{cofibrations} (resp.\ \emph{fibrations}) and denoted by 
$\begin{tikzcd}[column sep=0.5cm, cramped]{}\arrow[tail]{r}{}&{}\end{tikzcd}$ 
(resp.\ 
$\begin{tikzcd}[column sep=0.5cm, cramped]{}\arrow[two heads]{r}{}&{}\end{tikzcd}$). 
The morphisms in $\weq$ are called \emph{weak equivalences} and are denoted by 
$\begin{tikzcd}[column sep=0.5cm, cramped]{}\arrow{r}{\sim}&{}\end{tikzcd}$. 

\begin{enumerate}[label=\textup{(\arabic*)}]
\item The triplet $(\C,\seq,\weq)$ is called a \emph{weak Waldhausen (additive) category} if the following axioms are satisfied.

    \begin{enumerate}[label=\textup{(WC\arabic*)},leftmargin=41pt]
    \setcounter{enumii}{-1}
    \item\label{WC0} The class $\cof$ is closed under composition and contains each isomorphism in $\C$.
    \item\label{WC1} $\mathsf{Seq}$ contains all split exact sequences and is closed under isomorphism.
Any distinguished sequence \eqref{eqn:dist-sequence} is a weak cokernel sequence, in that $g$ is a weak cokernel of $f$.
For a cofibration $f$, the corresponding fibrations are also called \emph{distinguished weak cokernels}.
    \item\label{WC2} Any pair $(f,c)$ of a cofibration $A\overset{f}{\rightarrowtail} B$ and a morphism $A\overset{c}{\lra}C$ yields a cofibration 
    $\begin{tikzcd}[cramped]
    A \arrow{r}{\begin{psmallmatrix}
        f\\-c
    \end{psmallmatrix}}& B\oplus C.
    \end{tikzcd}$
    Furthermore, the associated distinguished sequences of the form 
    $\begin{tikzcd}[column sep=1cm, cramped]
     A \arrow{r}{\begin{psmallmatrix}
         f\\-c
        \end{psmallmatrix}}
    & B\oplus C
        \arrow{r}{\begin{psmallmatrix}
            b, \amph g
        \end{psmallmatrix}}
    & D
    \end{tikzcd}$
    satisfy that $g$ belongs to $\cof$.
    \end{enumerate}
    \begin{enumerate}[label=\textup{(WW\arabic*)},leftmargin=41pt]
    \setcounter{enumii}{-1}
    
    \item\label{WW0} The class $\weq$ is closed under composition and contains each isomorphism in $\C$.
    
	\item\label{WW1} (Gluing Axiom) Suppose that 
    \begin{equation}
    \label{eqn:gluing-axiom}
    \begin{tikzcd}[row sep=0.5cm]
    C_{1} \arrow{d}{\sim}[swap]{z}
    & A_{1} \arrow{l}[swap]{c_{1}}\arrow[tail]{r}{f_{1}}\arrow{d}{\sim}[swap]{x}& B_{1} \arrow{d}{\sim}[swap]{y}\\
    C_{2} & A_{2} \arrow{l}[swap]{c_{2}}\arrow[tail]{r}{f_{2}}& B_{2}
    \end{tikzcd}
    \end{equation}
    is a commutative diagram 
    in which all vertical arrows are weak equivalences and the feathered arrows are cofibrations, and that 
    $\begin{tikzcd}[column sep=1.2cm, cramped]
     A_{i} \arrow{r}{\begin{psmallmatrix}
         f_{i}\\-c_{i}
        \end{psmallmatrix}}
    & B_{i}\oplus C_{i}
        \arrow{r}{\begin{psmallmatrix}
            b_{i}, \amph g_{i}
        \end{psmallmatrix}}
    & D_{i}
    \end{tikzcd}$
    are distinguished sequences for $i=1,2$. 
    Then there is a weak equivalence  
    $w\colon D_{1}\overset{\sim}{\lra} D_{2}$ such that the following diagram commutes.
    \[
    \begin{tikzcd}[column sep=1.5cm,row sep=0.6cm]
     A_{1} 
        \arrow{r}{\begin{psmallmatrix}
         f_{1}\\-c_{1}
        \end{psmallmatrix}}
        \arrow{d}{x}
    & B_{1}\oplus C_{1}
        \arrow{r}{\begin{psmallmatrix}
            b_{1}, \amph g_{1}
        \end{psmallmatrix}}
        \arrow{d}{y \oplus z}
    & D_{1}
        \arrow{d}{w}
    \\
    A_{2} 
        \arrow{r}{\begin{psmallmatrix}
         f_{2}\\-c_{2}
        \end{psmallmatrix}}
    & B_{2}\oplus C_{2}
        \arrow{r}{\begin{psmallmatrix}
            b_{2}, \amph g_{2}
        \end{psmallmatrix}}
    & D_{2}
    \end{tikzcd}
    \]
    \end{enumerate}

\setcounter{enumi}{1}
\item The triplet $(\C,\seq,\weq)$ is called a \emph{weak coWaldhausen (additive) category} 
if the triplet $(\C^{\op},\seq^{\op},\weq^{\op})$ is a weak Waldhausen additive category.

\item The triplet $(\C,\seq,\weq)$ is called a \emph{weak biWaldhausen (additive) category} if $(\C,\seq,\weq)$ is both weak Waldhausen and weak coWaldhausen. 

\item\label{item:saturation-for-wWH} A weak Waldhausen category $(\C,\seq,\weq)$ is said to satisfy the \emph{saturation axiom} if $\weq$ satisfies Definition~\ref{def:saturation_extension}\ref{item_saturation}.

\item A weak Waldhausen category $(\C,\seq,\weq)$ is said to satisfy the \emph{extension axiom} if it satisfies the analogue of Definition~\ref{def:saturation_extension}\ref{item_extension} for distinguished sequences.
\end{enumerate}
\end{definition}

As the name suggests, the notion of a weak Waldhausen category is rather weaker than the classical notion since the fibrations are now only \emph{weak} cokernels. Thus, an increased level of care is needed in arguments.

\begin{remark}\label{rem:wPO}
Suppose $(\C,\seq,\weq)$ is a weak Waldhausen category.  
\begin{enumerate}[label=\textup{(\arabic*)}]

\item 
As for Waldhausen categories, note that the classes $\cof = \cof_{\seq}$ of cofibrations and $\weq$ of weak equivalences induce subcategories of $\C$; see Remark~\ref{rem:Wald_cat_rem2}.

\item 
Since 
$\begin{tikzcd}[column sep=0.7cm, cramped]0 \arrow{r}& C \arrow{r}{\id_{C}}& C\end{tikzcd}$ 
is split exact, axiom \ref{WC1} implies $0\rtail C$ is a cofibration for any $C\in\C$.

\item\label{item:weak-PO-defn} Suppose $f\colon A \rtail B$ is a cofibration and $c\colon A\to C$ is any morphism. Then, by \ref{WC2}, there is a distinguished sequence 
$\begin{tikzcd}[column sep=1cm, cramped]
    A \arrow{r}{\begin{psmallmatrix}
        f\\-c
    \end{psmallmatrix}}
    & B\oplus C
        \arrow{r}{\begin{psmallmatrix}
            b, \amph g
        \end{psmallmatrix}}
    & D
\end{tikzcd}$
with $g\in\cof$. Since this distinguished sequence is a weak cokernel sequence, it gives rise to a \emph{weak pushout square} as follows.
\begin{equation}\label{eqn:remark-about-weak-POs}
\begin{tikzcd}[row sep=0.5cm]
A \arrow[tail]{r}{f}\arrow{d}[swap]{c}\wPO{dr}& B\arrow{d}{b} \\
C \arrow[tail]{r}[swap]{g}& D.
\end{tikzcd}
\end{equation}
That is, given morphisms $d\in\C(C,E)$ and $e\in\C(B,E)$ with $dc=ef$, there exists (a not necessarily unique) $h\in\C(D,E)$ such that $hg=d$ and $hb=e$.
We note that the object $D$ is not uniquely determined, even up to isomorphism. 
Therefore, a given cofibration may possibly admit non-isomorphic distinguished weak cokernels. However, these weak cokernels are all \emph{weakly equivalent} (i.e.\ there is a weak equivalence between them) by \ref{WW1}; see also \ref{WW1'} in Lemma~\ref{lem:WW1'}. 

Notice that if $\begin{psmallmatrix}
    b, & g
\end{psmallmatrix}$ 
is a cokernel of 
$\begin{psmallmatrix}
        f\\-c
    \end{psmallmatrix}$ 
(i.e.\ if $\begin{psmallmatrix}
    b, & g
\end{psmallmatrix}$
is also epic), 
then \eqref{eqn:remark-about-weak-POs} is a pushout square, with the usual universal property including uniqueness.
\end{enumerate}
\end{remark}

\begin{remark}\label{rem:wPB}
If $(\C,\seq,\weq)$ is weak coWaldhausen, then the concept dual to a weak pushout exists.
For any pair $(f,c)$ of a fibration $B\overset{f}{\twoheadrightarrow}C$ and a morphism $D\overset{c}{\lra}C$, there exists a distinguished sequence 
$\begin{tikzcd}[column sep=1.2cm, cramped]
    E \arrow{r}{\begin{psmallmatrix}
        b \\ g
    \end{psmallmatrix}}
    & B\oplus D 
        \arrow{r}{\begin{psmallmatrix}
            f,\amph -c
        \end{psmallmatrix}}
    & C
\end{tikzcd}$
with $g\in\fib_{\seq}$.
The associated commutative diagram 
\[
\begin{tikzcd}[row sep=0.5cm]
E \arrow[two heads]{r}{g}\arrow{d}[swap]{b}\wPB{dr}& D \arrow{d}{c}\\
B \arrow[two heads]{r}[swap]{f}&C 
\end{tikzcd}
\]
will be referred to as a \emph{weak pullback of $f$ along $c$}. 
\end{remark}

\begin{remark}\label{rem:why-distinguished-seqs}
We note that specifying a class of distinguished sequences in Definition~\ref{def:weak_Wald_cat} is essential for our purposes. 
Indeed, if we were to specify only a collection of cofibrations and define weak Waldhausen structures by simply replacing pushouts in the definition of a Waldhausen category (i.e.\ in Definition~\ref{def:Waldhausen_category}) with weak pushouts,
then triangulated categories would yield undesirable phenomena.
For instance, given a distinguished triangle 
$A\overset{f}{\lra}B\overset{g}{\lra}C\lra A[1]$ in a triangulated category $\C$,
the sequence
$
\begin{tikzcd}[cramped, column sep=0.7cm]
A \arrow{r}{f}
    & B\arrow{r}{\begin{psmallmatrix}
g\\
0
\end{psmallmatrix}}
    & C\oplus D
\end{tikzcd}
$
is a weak cokernel sequence for any object $D\in\C$. 
Then \ref{WW1} would force $C$ to be weakly equivalent to  $C\oplus D$. Since we would want $\C$ equipped with $\cof = \C^{\to}$ and $\weq = \Iso\C$ to be weak Waldhausen, this would mean $C$ would be isomorphic to  $C\oplus D$, which is the problem.

This shows that the class of all weak cokernel sequences is too large to capture the intended structure. 
The introduction of the class $\seq$ of distinguished sequences solves this issue. 

Equally, one could consider ``distinguished weak pushouts'' in the non-additive setting. But for simplicity, and especially from the perspective of our applications, we focus on the additive setting, leading to the more natural formulation of the axioms in Definition~\ref{def:weak_Wald_cat}.
\end{remark}

Note that the first assertion in the following lemma uses neither \ref{WW0} nor \ref{WW1}.

\begin{lemma}
\label{lem:basics_of_cofibration_and_weq}
Suppose $(\C,\seq,\weq)$ is a weak Waldhausen category. 
Then: 
\begin{enumerate}[label=\textup{(\arabic*)}]
    \item\label{item:direct-sum-of-cofs-is-cof} the finite direct sum of cofibrations is again a cofibration, and
    \item\label{item:direct-sum-of-weqs-is-weq} the finite direct sum of weak equivalences is again a weak equivalence.
\end{enumerate}
\end{lemma}

\begin{proof}
\ref{item:direct-sum-of-cofs-is-cof}:\;\;
Suppose $A_i\overset{f_i}{\rightarrowtail}B_i$ lies in $\cof$ for $i=1,2$.
Their direct sum 
$\begin{tikzcd}[cramped, column sep=1.3cm]
    A_1\oplus A_2
    \arrow{r}{f_1\oplus f_2}
    &B_1\oplus B_2
\end{tikzcd}$
factors as
\[
\begin{tikzcd}[column sep=1.5cm]
    A_1\oplus A_2\arrow{r}{f_1\oplus \id_{A_2}}
    &B_1\oplus A_2\arrow{r}{\id_{B_1}\oplus f_2}
    &B_1\oplus B_2.
\end{tikzcd}
\]
We shall show $f_1\oplus \id_{A_2}$ belongs to $\cof$.
Since the sequence 
\[
\begin{tikzcd}[column sep=2cm]
A_{1} \arrow{r}{\begin{psmallmatrix}f_1 \\ -\id_{A_{1}}\\0 \end{psmallmatrix}}
	& B_1 \oplus A_1 \oplus A_2 
		\arrow{r}
		{
			\begin{psmallmatrix}
			\id_{B_{1}} \amph f_{1} \amph 0 \\ 
			0 \amph 0 \amph \id_{A_{2}}
			\end{psmallmatrix}
		}
	& B_1 \oplus A_2
\end{tikzcd}
\]
is split exact, it lies in $\seq$ by \ref{WC1}. 
Moreover, we have that $f_1\oplus \id_{A_{2}}\colon A_{1}\oplus A_{2} \to B_{1}\oplus A_{2}$ is a cofibration by \ref{WC2}. 
Closure under isomorphism of $\cof$ shows $\id_{B_{1}}\oplus f_2\in \cof$, and so $f_1\oplus f_2\in \cof$ by \ref{WC0}.

\ref{item:direct-sum-of-weqs-is-weq}:\;\; 
Suppose $f_{i}\colon A_i\overset{\sim}{\lra}B_i$ lies in $\weq$ for $i=1,2$. 
Notice that we have pushout squares
\[
\begin{tikzcd}[column sep=1.2cm]
0 
    \arrow{r}
    \arrow{d}
& A_{1}
    \arrow{d}{\begin{psmallmatrix}
        \id_{A_{1}} \\ 0
    \end{psmallmatrix}}
\\
A_{2}
    \arrow{r}{\begin{psmallmatrix}
        0 \\ \id_{A_{2}}
    \end{psmallmatrix}}
& A_{1}\oplus A_{2}
\end{tikzcd}
\hspace{0.5cm}
\text{and}
\hspace{0.5cm}
\begin{tikzcd}[column sep=1.2cm]
0 
    \arrow{r}
    \arrow{d}
& B_{1}
    \arrow{d}{\begin{psmallmatrix}
        \id_{B_{1}} \\ 0
    \end{psmallmatrix}}
\\
B_{2}
    \arrow{r}{\begin{psmallmatrix}
        0 \\ \id_{B_{2}}
    \end{psmallmatrix}}
& B_{1}\oplus B_{2}
\end{tikzcd}
\]
(see Remark~\ref{rem:wPO}\ref{item:weak-PO-defn}), and by \ref{WC1} 
the distinguished split exact sequences
\[
\begin{tikzcd}[column sep=1.8cm]
0
    \arrow{r}
&[-0.5cm]
A_{1} \oplus A_{2}
    \arrow{r}{\id_{A_{1}} \oplus \id_{A_{2}}}
& A_{1} \oplus A_{2}
\end{tikzcd}
\hspace{0.5cm}
\text{and}
\hspace{0.5cm}
\begin{tikzcd}[column sep=1.8cm]
0
    \arrow{r}
&[-0.5cm]
B_{1} \oplus B_{2}
    \arrow{r}{\id_{B_{1}} \oplus \id_{B_{2}}}
& B_{1} \oplus B_{2}.
\end{tikzcd}
\]
Thus, applying \ref{WW1} to 
\[
\begin{tikzcd}[row sep=0.5cm]
A_{2}
    \arrow{d}{\sim}[swap]{f_{2}}
& 0 
    \arrow{l}
    \arrow[tail]{r}
    \arrow{d}{\sim}
& A_{1}
    \arrow{d}{\sim}[swap]{f_{1}}\\
B_{2}
& 0
    \arrow{l}
    \arrow[tail]{r}
& B_{1}
\end{tikzcd}
\]
shows that the unique morphism 
$
f_{1}\oplus f_{2}
\colon A_{1}\oplus A_{2} \to B_{1}\oplus B_{2}
$
must be a weak equivalence.
\end{proof}

We present a useful alternative version of the Gluing Axiom \ref{WW1}.

\begin{lemma}\label{lem:WW1'}
Suppose $\C$ is an additive category equipped with a class of distinguished sequences $\seq$ and weak equivalences $\weq$. 
Then $(\C,\seq,\weq)$ satisfies $\ref{WW1}$ if and only if 
it satisfies:
\begin{enumerate}[label=\textup{(WW\arabic*')},leftmargin=50pt]
\item\label{WW1'}
For any pair of distinguished sequences 
$\begin{tikzcd}[column sep=0.6cm, cramped]
A_{i}
    \arrow[tail]{r}{f_{i}}
& B_{i}
    \arrow[two heads]{r}{g_{i}}
& C_{i}
\end{tikzcd}$
($i=1,2$),
and weak equivalences 
$x\colon A_{1} \overset{\sim}{\to} A_{2}$
and 
$y\colon B_{1} \overset{\sim}{\to} B_{2}$ 
with 
$yf_{1} = f_{2}x$, 
there exists a weak equivalence 
$z\colon C_{1} \overset{\sim}{\to} C_{2}$
such that $zg_{1} = g_{2}y$.
\begin{equation}
\label{eqn:WW1'}
\begin{tikzcd}[row sep=0.5cm]
A_{1}
    \arrow[tail]{r}{f_{1}}
    \arrow{d}{x}[swap]{\sim}
& B_{1}
    \arrow[two heads]{r}{g_{1}}
    \arrow{d}{y}[swap]{\sim}
& C_{1}
    \arrow[dotted]{d}{\exists z}[swap]{\sim}
\\
A_{2}
    \arrow[tail]{r}{f_{2}}
& B_{2}
    \arrow[two heads]{r}{g_{2}}
& C_{2}
\end{tikzcd}
\end{equation}
\end{enumerate}
\end{lemma}
\begin{proof}
Due to Lemma~\ref{lem:basics_of_cofibration_and_weq}\ref{item:direct-sum-of-weqs-is-weq}, it is clear that \ref{WW1'} implies \ref{WW1}.
For the converse, assume \ref{WW1} holds and that we have the hypotheses of \ref{WW1'}. 
Applying \ref{WW1} to
\[
\begin{tikzcd}[row sep=0.5cm]
0
    \arrow[equals]{d}
& A_{1}
    \arrow[tail]{r}{f_{1}}
    \arrow{d}{x}[swap]{\sim}
    \arrow{l}
& B_{1}
    \arrow{d}{y}[swap]{\sim}
\\
0
& A_{2}
    \arrow[tail]{r}{f_{2}}
    \arrow{l}
& B_{2}
\end{tikzcd}
\]
and the distinguished sequences 
$\begin{tikzcd}[column sep=0.6cm, cramped]
A_{i}
    \arrow[tail]{r}{f_{i}}
& B_{i}
    \arrow[two heads]{r}{g_{i}}
& C_{i}
\end{tikzcd}$ 
($i=1,2$)
shows that \ref{WW1'} holds.
\end{proof}

We introduce some concepts for weak Waldhausen categories by analogy to the classical theory 
(cf.\ Definitions~\ref{def:exact_functor_Wald-cats} and \ref{def:w_acyclic}).

\begin{definition}
\label{def:exact_func_wWald}
Let $(\C,\seq,\weq)$ and $(\C',\seq',\weq')$ be weak Waldhausen categories.
\begin{enumerate}

\item 
An additive functor $F\colon \C\to\C'$ is called an \emph{exact functor} if it preserves distinguished sequences and weak equivalences, namely,
$F(\seq) \subseteq \seq'$ 
and 
$F(\weq)\subseteq \weq'$ hold. 
Moreover, if a quasi-inverse $F^{-1}$ of $F$ exists and is exact, then $F$ is called an \emph{exact equivalence}.

\item 
Suppose $\C$ is an additive subcategory of $\C'$ and $\inc\colon \C\to\C'$ is the inclusion functor. 
The triplet $(\C,\seq,\weq)$ is called a \emph{weak Waldhausen (additive) subcategory} of $(\C',\seq',\weq')$ if 
\begin{enumerate}[label=\textup{(\roman*)}]
\item $\mathsf{inc}$ is an exact functor of weak Waldhausen categories,

\item $\seq = \seq' \cap \C^{\to\to}$,
and 

\item $\weq = \weq' \cap \C^{\to}$.
\end{enumerate}

\item
An object $C\in \C$ is \emph{$\weq$-acyclic} if 
the zero map $0\rightarrowtail C$ belongs to $\weq$. 
We denote by $\C^{\weq}$ the full subcategory of all $\weq$-acyclic objects in $\C$.

\end{enumerate}
Like the case of classical Waldhausen categories, the composition of exact functors is exact.
Thus we denote by $\wWald$ the category of skeletally small weak Waldhausen categories and exact functors.
\end{definition}

The definition of an exact functor between weak Waldhausen categories differs from the original one (e.g., \cite[Ch.~II, 9.1.8]{Wei13}), which is attributable to specifying a class of distinguished sequences rather than cofibrations. 
However, see Remark~\ref{rem:wWH-WH-comparison}\ref{item:wWH-functor}.

The next result allows us to define a localization sequence for weak Waldhausen categories.

\begin{lemma}
\label{lem:localization_sequence_wWald}
Suppose $(\C,\seq,\veq)$ and $(\C,\seq,\weq)$ are weak Waldhausen categories with $\veq\sse\weq$. 
Consider the subcategory $\C^{\weq}\sse\C$ of $\weq$-acyclic objects. 
There is a weak Waldhausen subcategory $(\C^{\weq},\seq^{\weq},\veq^{\weq})$ of $(\C,\seq,\veq)$, and a sequence in $\wWald$
\begin{equation}
\label{seq:localization_sequence_wWald}
\begin{tikzcd}
    (\C^{\weq},\seq^{\weq},\veq^{\weq})
        \arrow[hook]{r}{\inc}
    &(\C,\seq,\veq)
        \arrow[hook]{r}{\id_{\C}} 
    &(\C,\seq,\weq)
\end{tikzcd}
\end{equation}
of weak Waldhausen subcategories where the first arrow is the canonical inclusion.
\end{lemma}

\begin{proof}
First, the subcategory $\C^{\weq}$ is full by definition, 
additive by Lemma~\ref{lem:basics_of_cofibration_and_weq}\ref{item:direct-sum-of-weqs-is-weq}, and closed under isomorphisms by \ref{WW0}. 

Second, we claim that if 
$A\overset{f}{\rightarrowtail} B\overset{g}{\twoheadrightarrow} C$ lies in $\seq$ with $A,B\in\C^{\weq}$, then $C\in\C^{\weq}$. 
Indeed, applying \ref{WW1'} (see Lemma~\ref{lem:WW1'}) for $(\C,\seq,\weq)$ to 
\[
\begin{tikzcd}[row sep=0.5cm]
0
    \arrow[tail]{r}
    \arrow{d}[swap]{\sim}
& 0
    \arrow[two heads]{r}{}
    \arrow{d}{}[swap]{\sim}
& 0
    \arrow[dotted]{d}
\\
A
    \arrow[tail]{r}{f}
& B
    \arrow[two heads]{r}{g}
& C
\end{tikzcd}
\]
shows that the (unique) morphism $0\to C$ lies in $\weq$ and so $C$ is $\weq$-acyclic. 

Thus, define 
\[
\seq^{\weq} \deff \seq \cap (\C^{\weq})^{\to\to}
\hspace{0.5cm}\text{and} \hspace{0.5cm}
\veq^{\weq} \deff \veq \cap (\C^{\weq})^{\to}.
\]
Notice that 
$
    \cof_{\seq^{\weq}} 
    = \cof_{\seq} \cap (\C^{\weq})^{\to}$, 
where $\cof_{\seq}$ is the class of cofibrations of $\seq$ (see Definition~\ref{def:weak_Wald_cat}). 
Axioms \ref{WC0}, \ref{WC1} and \ref{WW0} are all clear. 
Axioms \ref{WC2} and \ref{WW1} follow from those for $(\C,\seq,\veq)$ using the claim above. 
Hence, $(\C^{\weq},\seq^{\weq},\veq^{\weq})$ is a weak Waldhausen category. 
The remaining assertions are clear.
\end{proof}

\begin{remark}\label{rem:acyclics-inherit-coWH}
Note that in Lemma~\ref{lem:localization_sequence_wWald}, if $(\C,\seq,\weq)$ is weak coWaldhausen, then so too is $(\C^{\weq},\seq^{\weq},\veq^{\weq})$. 
This is because, in this case, if $A \rtail B \twoheadrightarrow C$ is in $\seq$ with $B,C\in\C^{\weq}$, then it follows that $A\in\C^{\weq}$.
\end{remark}

We make the following definition as a generalization of Definition~\ref{def:exact_sequence_Wald}.

\begin{definition}
\label{def:exact_sequence_wWald} 
We call the sequence \eqref{seq:localization_sequence_wWald} a \emph{localization sequence} in $\wWald$ if it satisfies the analogue of Definition~\ref{def:exact_sequence_Wald} for weak Waldhausen categories.
\end{definition}

We conclude Section~\ref{sec:weak_Wald_cat} by comparing the definitions introduced here with the corresponding notions for classical Waldhausen categories from Section~\ref{sec:classical_Wald_cat}. 
First, we characterise those weak Waldhausen categories that are Waldhausen, and show that the notion of a weak Waldhausen category indeed generalizes the classical notion.

\begin{proposition}
\label{prop:Wald_vs_wWald}
The following statements hold.
\begin{enumerate}[label=\textup{(\arabic*)}]
    \item\label{item:weakWHisWH} If $(\C,\seq,\weq)$ is a weak Waldhausen category where each distinguished sequence is a cokernel sequence, then $(\C,\cof_{\seq},\weq)$ is a Waldhausen category. 
    
    \item\label{item:WisweakWH} Suppose $(\C,\cof,\weq)$ is a Waldhausen category, and define $\seq$ to be the collection of all cofibration sequences. Then $(\C,\seq,\weq)$ is a weak Waldhausen category in which all the distinguished sequences are cokernel sequences.
\end{enumerate}
\end{proposition}

\begin{proof}
\ref{item:weakWHisWH}:\;\;
Note that \ref{C0}, \ref{C1} and \ref{W0} follow from \ref{WC0}, \ref{WC1} and \ref{WW0}, respectively. 
Since each distinguished sequence is a cokernel sequence and $\seq$ is closed under isomorphism, we see that $\seq$ contains \emph{all} cokernel sequences 
$A \overset{f}{\rightarrowtail} B \to C$
for each cofibration $f\in\cof_{\seq}$. 
Combining this observation with \ref{WC2} (resp.\ \ref{WW1}) yields \ref{C2} (resp.\ \ref{W1}).

\ref{item:WisweakWH}:\;\; 
First, note that \ref{WC0}, \ref{WW0} and \ref{WW1} follow from \ref{C0}, \ref{W0} and \ref{W1}, respectively. 

For \ref{WC1}, we have that $\seq$ contains all split exact sequences by Lemma~\ref{lem:Wald_contain_split} and it is clear that $\seq$ consists of weak cokernel sequences. It is closed under isomorphism by Remark~\ref{rem:Wald_cat_rem1}\ref{item:cofib-sequences}.

For \ref{WC2}, suppose $A\overset{f}{\rightarrowtail} B$ is a cofibration (i.e.\ belongs to $\cof = \cof_{\seq}$) and $A\overset{c}{\lra} C$ is any morphism. Then the morphism $\binom{f}{-c}$ is equal to the composite 
$(f\oplus \id_{C}) \circ \binom{\id_{A}}{-c}$. 
By Lemma~\ref{lem:basics_of_cofibration_and_weq}\ref{item:direct-sum-of-cofs-is-cof}, we have $\cof$ is closed under direct sums so $f\oplus \id_{C}\in\cof$. 
By Lemma~\ref{lem:Wald_contain_split}, the morphism $\binom{\id_{A}}{-c}$ lies in $\cof$ since 
$\begin{tikzcd}[column sep=1.2cm, cramped]
A
    \arrow{r}{\begin{psmallmatrix}
        \id_{A}\\-c
    \end{psmallmatrix}}
& A\oplus C
    \arrow{r}{\begin{psmallmatrix}
        c, \amph \id_{C}
    \end{psmallmatrix}}
& C
\end{tikzcd}$
is split exact. 
Hence, the composite is also a cofibration by \ref{C0}. 
Then $\ref{C2}$, and that $\seq$ and $\cof$ are closed under isomorphisms, imply there is a cofibration (i.e.\ distinguished) sequence of the form
$\begin{tikzcd}[column sep=1.1cm, cramped]
     A \arrow{r}{\begin{psmallmatrix}
         f\\-c
        \end{psmallmatrix}}
    & B\oplus C
        \arrow{r}{\begin{psmallmatrix}
            b, \amph g
        \end{psmallmatrix}}
    & B \bincoprod_{A} C
    \end{tikzcd}$
    with $g\in \cof$.
\end{proof}

In the setup of Proposition~\ref{prop:Wald_vs_wWald}\ref{item:WisweakWH}, we say that $(\C,\seq,\weq)$ is the \emph{corresponding weak Waldhausen category} of the Waldhausen category $(\C,\cof,\weq)$.

\begin{remark}
\label{rem:wWH-WH-comparison}
Suppose 
    $(\C,\cof,\veq)$, 
    $(\C,\cof,\weq)$ 
and 
    $(\C',\cof',\weq')$ 
are Waldhausen additive categories, 
with corresponding weak Waldhausen categories 
    $(\C,\seq,\veq)$,
    $(\C,\seq,\weq)$ 
and 
    $(\C',\seq',\weq')$, 
respectively.
The following observations are straightforward to verify. 
\begin{enumerate}[label=\textup{(\arabic*)}]

\item \label{item:wWH-functor}
An additive functor $F\colon \C\to\C'$ is an exact functor 
$F\colon (\C,\seq,\weq) \to (\C',\seq',\weq')$ of weak Waldhausen categories (as defined in Definition~\ref{def:exact_func_wWald}) 
if and only if it is an exact functor 
$F\colon (\C,\cof,\weq) \to (\C',\cof',\weq')$ 
of Waldhausen categories (in the sense of Definition~\ref{def:exact_functor_Wald-cats}).

\item \label{item:wWH-subcategory}
Suppose $\C$ is an additive subcategory of $\C'$ and $\mathsf{inc}\colon \C \into \C'$ is the inclusion functor. 
Then 
$\inc \colon (\C,\cof,\weq) \into (\C',\cof',\weq')$ is an inclusion of a Waldhausen subcategory if and only if 
$\inc \colon (\C,\seq,\weq) \into (\C',\seq',\weq')$ is an inclusion of a weak Waldhausen subcategory.

\item \label{item:wWH-acyclics}
Consider the Waldhausen subcategory $(\C^{\weq},\cof^{\weq},\veq^{\weq})\sse (\C,\cof,\veq)$ 
of $\weq$-acyclic objects. 
The corresponding weak Waldhausen category of $(\C^{\weq},\cof^{\weq},\veq^{\weq})$ is the weak Waldhausen subcategory $(\C^{\weq},\seq^{\weq},\veq^{\weq})$ of $(\C,\seq,\veq)$ as produced in Lemma~\ref{lem:localization_sequence_wWald}. 

\item \label{item:wWH-localization-seq}
Lastly, we note that if \eqref{seq:localization_sequence_wWald} a localization sequence of weak Waldhausen categories, then \eqref{seq:localization_sequence_Wald} is a localization sequence of Waldhausen categories (since we are assuming the categories in this remark are Waldhausen). However, the converse does not necessarily follow. 

\end{enumerate}
\end{remark}

\subsection{The Grothendieck group of a weak Waldhausen category}
\label{sec:K0_of_wWald}

The \emph{Grothendieck group} $K_{0}(\C,\seq,\weq)$ of a skeletally small weak Waldhausen category $(\C,\seq,\weq)$ is given by a straightforward adjustment of Definition~\ref{def:K0_of_Wald_cat}, using distinguished sequences in place of cofibration sequences. 
It is also clear that an exact functor 
$F\colon (\C,\seq,\weq) \to (\C',\seq',\weq')$ 
of weak Waldhausen categories 
induces a group homomorphism
$K_{0}(F) \colon K_{0}(\C,\seq,\weq) \to K_{0}(\C',\seq',\weq')$, which extends to a functor 
$K_{0}\colon \wWald\to \Ab$.

We end this subsection with some simple examples.

\begin{example}\label{exam:K0-for-WH-is-K0-for-wWH}
Suppose $(\C,\cof,\weq)$ is a skeletally small Waldhausen category with corresponding weak Waldhausen category $(\C,\seq,\weq)$. Then it is clear that 
$K_{0}(\C,\cof,\weq) = K_{0}(\C,\seq,\weq)$.
\end{example}

\begin{example}\label{ex:contractible}
Let $\A$ be a skeletally small abelian category.
Let $\seq$ denote the class of all right exact sequences
$A\overset{f}{\to} B\overset{g}{\twoheadrightarrow} C$, 
namely, 
sequences where $g$ is a cokernel of $f$; 
and put $\weq \deff\Iso\C$, the class of all isomorphisms in $\C$.
Then $(\A,\seq,\weq)$ is a weak Waldhausen additive category. 
Indeed, since $\cof_{\seq}=\A^{\to}$, the conditions \ref{WC0} and \ref{WW0} are clear.
The other conditions also follow from this and the fact that all distinguished sequences are cokernel sequences. Furthermore, Proposition~\ref{prop:Wald_vs_wWald}\ref{item:weakWHisWH} tells us that $(\A,\cof_{\seq},\weq)$ is a Waldhausen category.

Lastly, in this example, we have $K_{0}(\A,\seq,\weq)\cong0$. Indeed, as any object $C\in\A$ fits into a distinguished sequence $C\lra 0\lra 0$,  we have $[C]=0$ in $K_{0}(\A,\seq,\weq)$.
\end{example}

\section{Localization theorems for weak Waldhausen categories}
\label{sec:localization-theorem-for-wWald}

Having established the framework of weak Waldhausen categories, we can state and prove our two main localization results, namely, Theorem~\ref{thm:Quillen_localization1_new} and Corollary~\ref{cor:Quillen_localization2}, in this section. 
These show that, under certain assumptions on $\weq$, the sequence \eqref{seq:localization_sequence_wWald3} below is a localization sequence in $\wWald$, which in turn induces a right exact sequence of their Grothendieck groups.
Our localization results are based on Corollary~\ref{cor:Schlichting_localization}. 
To accurately handle the Serre and Verdier quotients in terms of weak Waldhausen categories, we introduce the \emph{homotopy category} of a weak Waldhausen category in Section~\ref{subsec:homotopy-category}. 
Under some conditions, we show that the Grothendieck group of a weak Waldhausen category is isomorphic to that of its homotopy category (when it exists). 
This will be used to prove Corollary~\ref{cor:ES_via_Waldhausen}.

We work under the following setup throughout Section~\ref{sec:localization-theorem-for-wWald}.

\begin{setup}
\label{setup:Quillen_loc1}
We assume there is a sequence 
\begin{equation}
\label{seq:localization_sequence_wWald3}
\begin{tikzcd}
    (\C^{\weq},\seq^{\weq},\veq^{\weq})
        \arrow[hook]{r}{\inc}
    &(\C,\seq,\veq)
        \arrow[hook]{r}{\id_{\C}} 
    &(\C,\seq,\weq)
\end{tikzcd}
\end{equation}
of skeletally small weak Waldhausen subcategories, where $\C^{\weq}\sse\C$ is the subcategory of $\weq$-acyclic objects in $\C$. 
We put $\cof\deff\cof_{\seq}$ and $\fib\deff\fib_{\seq}$. 
\end{setup}

To compare the Grothendieck groups of $(\C,\seq,\veq)$ and $(\C,\seq,\weq)$, we introduce certain classes of morphisms (cf.\ Definition~\ref{def:Sn_from_thick}).

\begin{definition}
\label{def:mors_assoc_to_wWH_cat}
We define the following four classes of morphisms in $\C$.
\begin{enumerate}[label=\textup{(\roman*)}]

    \item $\L^{\ac}\deff \cof \cap \weq$

    \item $\L^{\ac}_{\sec}\deff\Set{ f\in\L^{\ac} | f \text{ is a section} } \sse\L^{\ac}$

    \item $\R^{\ac}\deff \fib \cap \weq$

    \item $\T\deff\set{ g\in\R^{\ac} | \text{there is a sequence  $A\overset{f}{\rightarrowtail}B\overset{g}{\twoheadrightarrow}C$ in $\seq$ with $A\in\C^{\weq}$}} \sse \R^{\ac}$ 
\end{enumerate}
\end{definition}

For the proof of Theorem~\ref{thm:Quillen_localization1_new}, we need the following lemma.

\begin{lemma}
\label{lem:L_sec}
If $f\in\L^{\ac}$ and 
$A\overset{\substack{f\\ \sim}}{\rightarrowtail}B\overset{g}{\twoheadrightarrow}C$ is a distinguished sequence in $\seq$, 
then the morphisms $0\to C$ and $C \to 0$ both lie in $\weq$. In particular, $C\in\C^{\weq}$.
\end{lemma}

\begin{proof}
Apply \ref{WW1'} (see Lemma~\ref{lem:WW1'}) to
\[
\begin{tikzcd}[row sep=0.6cm]
A
    \arrow[equals]{r}
    \arrow[equals]{d}[swap]{\sim}
& A
    \arrow[two heads]{r}
    \arrow{d}{f}[swap]{\sim}
& 0
    \arrow{d}
\\ 
A 
    \arrow[tail]{r}{f}
& B 
    \arrow[two heads]{r}{g}
& C
\end{tikzcd}
\hspace{1cm}
\text{and}
\hspace{1cm}
\begin{tikzcd}[row sep=0.6cm]
A 
    \arrow[tail]{r}{f}
    \arrow{d}{f}[swap]{\sim}
& B 
    \arrow[two heads]{r}{g}
    \arrow[equals]{d}[swap]{\sim}
& C
    \arrow{d}
\\
B
    \arrow[equals]{r}
& B
    \arrow[two heads]{r}
& 0.
\end{tikzcd}
\]
\end{proof}

\begin{remark}
\label{rem:fail_converse_L_sec}
For an arbitrary weak Waldhausen category $(\C,\seq,\weq)$, the converse of the first assertion of Lemma~\ref{lem:L_sec} may fail. 
For instance, let $(\A,\seq,\weq)$ denote the weak Waldhausen category, where 
$\A$ is an abelian category with at least one proper epimorphism $f\colon A\to B$ (i.e.\ $f$ is epic but not monic), 
$\seq$ is the collection of all right exact sequences in $\A$, and $\weq = \Iso\C$. 
Then $f$ is a cofibration in $(\A,\seq,\weq)$; see Example~\ref{ex:contractible}.
However, even though the distinguished weak cokernel of $f$ is (up to isomorphism) $B \to 0$ and $0\in\C^{\weq}$, the morphism $f$ does not lie in $\weq=\Iso\A$ and hence $f\notin\L^{\ac}$.
\end{remark}

\begin{lemma}\label{lem:weak-biWald-implies-R-is-T}
Suppose $(\C,\seq,\weq)$ is weak biWaldhausen. 
If $g\in\R^{\ac}$ and 
$A\overset{f}{\rightarrowtail}B\overset{\substack{g\\ \sim}}{\twoheadrightarrow}C$ is in $\seq$, 
then $0\to A$ and $A \to 0$ both belong to $\weq$. 
In particular, $\R^{\ac} = \T$ in this case.
\end{lemma}

\begin{proof}
Apply Lemma~\ref{lem:L_sec} to the distinguished sequence
$C\overset{g^{\op}}{\rightarrowtail}B\overset{f^{\op}}{\twoheadrightarrow}A$
in the weak Waldhausen category $(\C^{\op},\seq^{\op},\weq^{\op})$, where $g^{\op}\in\weq^{\op}$.
\end{proof}

We can now prove our main results of Section~\ref{sec:localization-theorem-for-wWald}.
These show that the functor $K_{0}\colon \wWald\to\Ab$ sends certain localization sequences to right exact sequences. Recall we are assuming Setup~\ref{setup:Quillen_loc1}.

\begin{theorem}[Weak Waldhausen Localization Theorem]
\label{thm:Quillen_localization1_new}
Assume that $\weq$ consists of finite compositions of morphisms from $\L^{\ac} \cup \T \cup \veq$. 
Then \eqref{seq:localization_sequence_wWald3} is a localization sequence in $\wWald$ that induces a right exact sequence in $\Ab$ as follows.
\begin{equation}
\label{seq:homotopy_fiber1}
\begin{tikzcd}[column sep=2.1cm]
K_{0}(\C^{\weq},\seq^{\weq},\veq^{\weq}) 
    \arrow{r}{F \deff K_{0}(\inc)}
& K_{0}(\C,\seq,\veq)
    \arrow{r}{G \deff K_{0}(\id_{\C})}
& K_{0}(\C,\seq,\weq)
    \arrow{r}
&[-1.5cm] 0
\end{tikzcd}
\end{equation}
\end{theorem}
\begin{proof}
First, we confirm that \eqref{seq:localization_sequence_wWald3} is a localization sequence in $\wWald$. 
Thus, suppose $(\C',\seq',\weq')$ is a weak Waldhausen category satisfying the saturation and extension axioms, and also that $H\colon (\C,\seq,\veq)\to (\C',\seq',\weq')$ is an exact functor with $H(\C^{\weq})\sse (\C')^{\weq'}$. 
Since $(\C,\seq,\veq)$ and $(\C,\seq,\weq)$ have the same underlying category and the same distinguished sequences, we need only show $H(\weq)\sse\weq'$.
Since $H$ is exact, we know $H(\veq)\sse\weq'$.

To show $H(\L^{\ac})\sse\weq'$, 
consider $f\in\L^{\ac}$ together with a distinguished sequence $A\overset{f}{\rtail}B\onto C$, where $C\in\C^{\weq}$ by Lemma~\ref{lem:L_sec}.
There is a commutative diagram
\[
\begin{tikzcd}
HA \arrow[equals]{d}[swap]{\sim}\arrow[equals]{r}& HA \arrow[two heads]{r}\arrow[tail]{d}{Hf} & 0\arrow{d}{\sim}
\\
HA\arrow[tail]{r}{Hf} & HB \arrow[two heads]{r} & HC
\end{tikzcd}
\]
of distinguished sequences in $\C'$. 
By the assumption, we know $HC\in(\C')^{\weq'}$ which shows $0\to HC$ is in $\weq'$.
Thus, by the extension axiom, the morphism $HA\overset{Hf}{\rightarrowtail} HB$ also lies in $\weq'$.

Lastly, consider a morphism $g'\in\T$ which admits 
a distinguished sequence
$A'\overset{f'}{\rightarrowtail} B'\overset{g'}{\twoheadrightarrow} C'$ 
with $A'\in\C^{\weq}$. There is a commutative diagram
\begin{equation}
\label{eqn:g-primt-case}
\begin{tikzcd}
0
    \arrow{r}{}
    \arrow{d}[swap]{\sim}
& HB' 
    \arrow[equals]{r}
    \arrow[equals]{d}[swap]{\sim}
& HB' 
    \arrow{d}{Hg'}
\\
HA' 
    \arrow[tail]{r}{Hf'}
& HB' 
    \arrow[two heads]{r}{Hg'}
& HC' 
\end{tikzcd}
\end{equation}
in $\C'$ of sequences in $\seq'$, where $HA'\in (\C')^{\weq'}$. 
Then \ref{WW1'} (see Lemma~\ref{lem:WW1'}) implies $Hg'$ belongs to $\weq'$, noting that an induced morphism $HB' \to HC'$ must be equal to $Hg'$ in order to make \eqref{eqn:g-primt-case} commute.
We have thus verified $H(\weq)\sse\weq'$, and hence  \eqref{seq:localization_sequence_wWald3} is indeed a localization sequence.

Next we verify the last assertion.
Since \eqref{seq:localization_sequence_wWald3} sits in $\wWald$, we have the induced group homomorphisms $F$ and $G$ in \eqref{seq:homotopy_fiber1}.
Since $0\to A$ belongs to $\weq$ for any object $A\in\C^{\weq}$, the composition $F\circ G$ vanishes.
Thus, $G$ factors uniquely through the cokernel $\Cok F$ of $F$ as follows.
\begin{equation}
\label{diag:homotopy_fiber1}
\begin{tikzcd}
K_{0}(\C^{\weq},\seq^{\weq},\veq^{\weq}) 
    \arrow{r}{F}
& K_{0}(\C,\seq,\veq)
    \arrow[two heads]{r}{P}
    \arrow[two heads]{d}{G}
& \Cok F
    \arrow{r}{}
    \arrow[dotted]{dl}{^{\exists !}G'}
& 0\\
\isoclass(\C) 
    \arrow[two heads]{ur}{\phi}
    \arrow[two heads]{r}[swap]{\psi}
&K_{0}(\C,\seq,\weq)&&
\end{tikzcd}
\end{equation}
The homomorphisms $\phi$ and $\psi$ are the canonical surjective maps, 
and $\psi$ factors uniquely as $G\phi$ because $\veq\sse\weq$. 
We shall show that there exists an inverse of $G'$ by showing that $P\phi$ must factor uniquely through $\psi$.

First note that for any distinguished sequence 
$\begin{tikzcd}[column sep=0.5cm, cramped]
A \arrow[tail]{r}& B \arrow[two heads]{r}{}& C
\end{tikzcd}$
in $\seq$, we have 
$P[A] - P[B] + P[C] = 0$ in $\Cok F$, because already 
$\phi[A] - \phi[B] + \phi[C] = 0$.

Second, let $f\colon A\to B$ be a morphism in $\weq$. To see that 
$P[A] = P[B]$, it suffices to consider the three cases since $\weq$ consists of compositions from $\L^{\ac}\cup \T \cup \veq$. 
\begin{enumerate}
    \item If $f\in\veq$, then we see that $[A] = [B]$ already in $K_{0}(\C,\seq,\veq)$, and hence $P[A]=P[B]$ in $\Cok F$. 

    \item If $f\in\T$, then there is a distinguished sequence 
    $\begin{tikzcd}[column sep=0.5cm, cramped]
    K \arrow[tail]{r}& A \arrow[two heads]{r}{f}& B
    \end{tikzcd}$ in $\seq$ 
    with $K\in \C^{\weq}$. 
    In particular, $P[K] = 0$ in $\Cok F$ and hence the equality 
    $[A] = [K] + [B]$ in $K_{0}(\C,\seq,\veq)$ 
    implies 
    $P[A]=P[B]$ in $\Cok F$.  

    \item If $f\in\L^{\ac}$, then we have a distinguished sequence $A\overset{f}{\rightarrowtail} B\twoheadrightarrow C$ with $C\in\C^{\weq}$ by Lemma~\ref{lem:L_sec}. As above we deduce $P[A]=P[B]$ in $\Cok F$. 
\end{enumerate}

Hence, there exists a unique homomorphism $G''\colon K_{0}(\C,\seq,\weq)\to\Cok F$ with $P=G''\circ G$ by the universal property of $K_{0}(\C,\seq,\weq)$ using that $\phi$ is epic. 
This proves that $G'$ is an isomorphism and that \eqref{seq:homotopy_fiber1} is exact.
\end{proof}

Now we can see that Corollary~\ref{cor:Schlichting_localization}, a localization theorem for classical Waldhausen categories, is a consequence of Theorem~\ref{thm:Quillen_localization1_new}.

\begin{proof}[Proof of Corollary~\ref{cor:Schlichting_localization}]
Due to Proposition~\ref{prop:Wald_vs_wWald}\ref{item:WisweakWH}, we know that the Waldhausen category $(\C,\cof,\weq)$ corresponds to a weak Waldhausen category $(\C,\seq,\weq)$. Since $(\C,\cof,\weq)$ satisfies the saturation axiom, if $w = vf$ where $w\in\weq$, $v\in\veq\sse\weq$ and $f\in\cof$, we see that $f\in\cof\cap\weq = \L^{\ac}$. The result then follows from an application of Theorem~\ref{thm:Quillen_localization1_new}, noting Remark~\ref{rem:wWH-WH-comparison}\ref{item:wWH-acyclics} and Example~\ref{exam:K0-for-WH-is-K0-for-wWH}.
\end{proof}

In the case of biWaldhausen structures, we can slightly relax the assumption on $\weq$.

\begin{corollary}[Weak biWaldhausen Localization Theorem]
\label{cor:Quillen_localization2}
Assume that \eqref{seq:localization_sequence_wWald3} is a sequence of weak biWaldhausen categories. 
If $\weq$ consists of finite compositions of morphisms from $\L^{\ac}\cup\R^{\ac}\cup \veq$, 
then \eqref{seq:localization_sequence_wWald3} is a localization sequence in $\wWald$ and \eqref{seq:homotopy_fiber1} is exact.
\end{corollary}

\begin{proof}
By Lemma~\ref{lem:weak-biWald-implies-R-is-T}, this follows immediately from Theorem~\ref{thm:Quillen_localization1_new}.
\end{proof}

\begin{remark}
Note that in Theorem~\ref{thm:Quillen_localization1_new}, and hence also in Corollary~\ref{cor:Quillen_localization2}, we did not need the saturation axiom for $(\C',\seq',\weq')$.
\end{remark}

\subsection{The homotopy category}
\label{subsec:homotopy-category}
 
This subsection contains a preparatory result relating the Grothen\-dieck group of a weak Waldhausen category to that of its homotopy category (when it exists). This will be used to prove Corollary~\ref{cor:ES_via_Waldhausen}.

Recall that Setup~\ref{setup:Quillen_loc1} is in play. 
Let $\overline{(-)}\colon \C\to \overline{\C}\deff \C/[\C^{\weq}]$ denote 
the additive quotient functor and 
define $\overline{\weq}\deff \set{ \overline{w}\in \overline{\C}^{\to} | w\in\weq }$.
Consider the localization functors $Q\colon \C \to \C[\weq^{-1}]$ 
and 
$\overline{Q}\colon \overline{\C} \to \overline{\C}[\overline{\weq}^{-1}]$. 
Note that we cannot deduce these localizations are additive without further assumptions (see Setup~\ref{setup:section-3-1}). However, we do have the following.

\begin{lemma}
\label{lem:homotopy_cat}
There is an isomorphism $\C[\weq^{-1}]\cong \overline{\C}[\overline{\weq}^{-1}]$ of categories.
\end{lemma}
\begin{proof}
We consider the \emph{saturation} of $\weq$, namely,
\[
\weq_0\deff \set{f\in\C^{\to}|Q(f)\text{\ is an isomorphism in\ }\C[\weq^{-1}]} \supseteq \weq,
\]
and establish the associated localization $Q_0\colon \C\to\C[\weq_0^{-1}]$.
It is clear that the class $\weq_0$ is \emph{saturated}, that is, 
for each $f\in\C^{\to}$ we have  
$Q(f)$ is an isomorphism in $\C[\weq_0^{-1}]$ if and only if $f\in\weq_0$.
Then, by the universality of $\C[\weq^{-1}]$ and $\C[\weq_0^{-1}]$, we have an isomorphism $\C[\weq^{-1}]\cong\C[\weq_0^{-1}]$.

The ideal quotient $\C\to\overline{\C}$ is the localization $\C[\mathsf{S}^{-1}]$ with respect to the class $\mathsf{S}$ of retractions in $\C^{\to}$ that admit a kernel in $\C^{\weq}$; see \cite[the dual of Exam.~2.6]{Oga22a}.
We claim that $\mathsf{S}\sse\weq_0$. 
Indeed, by e.g\ \cite[Prop.~2.7]{Sha19}, 
any morphism in $\mathsf{S}$ is isomorphic to one of the form 
$\begin{tikzcd}[column sep=1.2cm, cramped]
N \oplus A
    \arrow{r}{\begin{psmallmatrix}
        0, \amph \id_{A}
    \end{psmallmatrix}}
& A,
\end{tikzcd}
$
where $N\in\C^{\weq}$. In particular, $0\colon 0 \to N$ is in $\weq$, and hence  
$
0 \oplus \id_{A}
\colon 0 \oplus A \to N \oplus A
$ 
also lies in $\weq \sse \weq_{0}$ by Lemma~\ref{lem:basics_of_cofibration_and_weq}\ref{item:direct-sum-of-weqs-is-weq}. 
The composition
$
\begin{tikzcd}[column sep=1.3cm, cramped]
0\oplus A
    \arrow{r}{0\oplus \id_{A}}
& N \oplus A
    \arrow{r}{\begin{psmallmatrix}
        0, \amph \id_{A}
    \end{psmallmatrix}}
& A
\end{tikzcd}
$
is an isomorphism in $\C$, so 
we have 
$\begin{psmallmatrix}
         0, & \id_{A}
    \end{psmallmatrix}
\in\weq_0$
as $\weq_0$ is saturated.

Since $Q$ inverts all morphisms in $\mathsf{S}\sse\weq_{0}$, 
there exists a unique functor 
$Q_{1}\colon \overline{\C} \iso \C[\mathsf{S}^{-1}] \to \C[\weq_0^{-1}]$ 
satisfying $Q = Q_{1} \circ \overline{(-)}$.
Also, by the universality of $\overline{\C}[\overline{\weq}^{-1}]$, there is a unique functor $Q_2\colon \overline{\C}[\overline{\weq}^{-1}]\to \C[\weq_0^{-1}]$ with $Q_1=Q_2\overline{Q}$.
Lastly, the universality of $\C[\weq^{-1}]$ guarantees a unique functor $Q_{3}\colon \C[\weq^{-1}]\to\overline{\C}[\overline{\weq}^{-1}]$ with $Q_{3} Q = \overline{Q} \circ \overline{(-)}$. 
The diagram below summarizes the situation.
\[
\begin{tikzcd}[column sep=2.5cm, row sep=0.7cm]
\C 
    \arrow{r}{Q\iso Q_{0}}
    \arrow{d}[swap]{\overline{(-)}}
& \C[\weq^{-1}]\iso \C[\weq_{0}^{-1}] 
    \arrow[dotted, bend left]{d}[yshift=-0.4cm]{^{\exists !}Q_{3}}
\\
\C[\mathsf{S}^{-1}] \iso\overline{\C}  
    \arrow[dotted]{ur}{^{\exists !}Q_{1}}
    \arrow{r}{\overline{Q}}
& \overline{\C}[\overline{\weq}^{-1}]
    \arrow[dotted, bend left]{u}[yshift=-0.4cm]{^{\exists !}Q_{2}}
\end{tikzcd}
\]

One can then check that $Q_{2}$ and $Q_{3}$ are mutually inverse  functors, yielding the isomorphism $\C[\weq^{-1}]\cong \overline{\C}[\overline{\weq}^{-1}]$ as desired.
\end{proof}

Given Lemma~\ref{lem:homotopy_cat}, we typically identify the categories $\C[\weq^{-1}]$  and $\overline{\C}[\overline{\weq}^{-1}]$. Note also that nothing is lost by assuming $\weq$ is saturated; see the proof of Lemma~\ref{lem:homotopy_cat}.
Thus, in addition to Setup~\ref{setup:Quillen_loc1}, we assume the following for the rest of Section~\ref{sec:localization-theorem-for-wWald}. 

\begin{setup}\label{setup:section-3-1}
Suppose $\overline{\weq}$ is a multiplicative system in $\overline{\C}$ and the class $\weq$ is saturated. 
\end{setup}

Under Setup~\ref{setup:section-3-1}, it follows from \cite[I.3.3]{GZ67} that $\overline{Q}\colon \overline{\C} \to \overline{\C}[\overline{\weq}^{-1}]$ is an additive functor of additive categories. 
Thus, in turn, the localization $Q\colon \C\to\C[\weq^{-1}]$ is also additive, which allows us to consider weak Waldhausen structures on $\C[\weq^{-1}]$.

\begin{definition}
\label{def:homotopy_cat}
We call a triplet 
$(\C[\weq^{-1}], \seq',\weq')$ the \emph{weak Waldhausen (additive) homotopy category of $(\C,\seq,\weq)$} if: 
\begin{enumerate}[label=\textup{(\roman*)}]
    \item $(\C[\weq^{-1}], \seq',\weq')$ is a weak Waldhausen category; 
    \item $\seq'$ coincides with the isomorphism closure of $Q(\seq)$; and 
    \item $\weq' = \Iso (\C[\weq^{-1}])$.
\end{enumerate}
In this case, we use the notation $\Ho(\C,\seq,\weq)$ to denote the triplet $(\C[\weq^{-1}], \seq',\weq')$. 
Note also that $Q\colon (\C,\seq,\weq) \to \Ho(\C,\seq,\weq)$ is an exact functor in $\wWald$.
\end{definition}

We may immediately prove the following.

\begin{proposition}
\label{prop:K0_of_homotopy_cat} 
If $\Ho(\C,\seq,\weq)$ exists, then 
$K_{0}(\C,\seq,\weq)\cong K_{0}(\Ho(\C,\seq,\weq))$.
\end{proposition}

\begin{proof}
Since the localization functor 
$Q\colon (\C,\seq,\weq) \to \Ho(\C,\seq,\weq)$ 
is exact, 
it induces a group homomorphism 
$K_{0}(Q)\colon K_{0}(\C,\seq,\weq)\to K_{0}(\Ho(\C,\seq,\weq))$.

We shall construct an inverse of $K_{0}(Q)$. 
Since  
$\Ob(\C,\seq,\weq) = \Ob(\Ho(\C,\seq,\weq))$, 
there is a canonical map 
$\isoclass(\C)\to K_{0}(\C,\seq,\weq)$ 
given by $[X]\mapsto [X]$, 
which we claim respects the generating relations of 
$K_{0}(\Ho(\C,\seq,\weq))$.

First, suppose $\alpha \colon A\to B$ is in $\weq'= \Iso (\C[\weq^{-1}])$, hence also an isomorphism in $\overline{\C}[\overline{\weq}^{-1}]$. 
As $\overline{\weq}$ is a multiplicative system (see Setup~\ref{setup:section-3-1}), 
there is a roof diagram 
$\begin{tikzcd}[column sep=0.5cm, cramped]
    A \arrow{r}{\overline{t}}& B' & B\arrow{l}[swap]{\overline{u}}
\end{tikzcd}$
in $\overline{\C}$ representing $\alpha$ with 
$u\in \weq$. 
Since $\alpha$ is an isomorphism and $\weq$ is saturated (see Setup~\ref{setup:section-3-1}), 
we actually have $t\in\weq$ also, 
so $[A] = [B'] = [B]$ in 
$K_{0}(\C,\seq,\weq)$.

Second, suppose  
$A\rightarrowtail B\twoheadrightarrow C$ 
is a distinguished sequence
in $\Ho(\C,\seq,\weq)$, which by definition is thus 
isomorphic to the image of a 
distinguished sequence 
$A'\rightarrowtail B'\twoheadrightarrow C'$ 
in $\seq$. 
For $X=A,B,C$, as argued above for $\alpha$, 
there is a roof diagram from $X$ to $X'$ 
consisting of morphisms in $\weq$. 
In particular, we have
$
[A]-[B]+[C]=[A']-[B']+[C']=0
$.
Hence, there is an induced group homomorphism 
$\psi\colon K_{0}(\Ho(\C,\seq,\weq)) \to K_{0}(\C,\seq,\weq)$ that is clearly inverse to $K_{0}(Q)$.
\end{proof}

\section{Extriangulated categories}
\label{sec:ET}

We recall some details on extriangulated categories in this section; see \cite{NP19} for more details. 
An \emph{extriangulated category} is a triplet $(\C,\BE,\fs)$ consisting of: 
\begin{enumerate}[label=\textup{(\arabic*)}]
    \item an additive category $\C$, 
    
    \item a biadditive functor $\BE\colon\C^{\op}\times\C\to \Ab$, and
    
    \item\label{item:realization} an assignment $\fs$ that, 
    for all $A,C\in\C$ and each element $\delta\in\BE(C,A)$, 
    associates an equivalence class $\fs(\delta) = [A\overset{f}{\lra} B\overset{g}{\lra} C]$ of a pair $(f,g) \in \C^{\to\to}$.
\end{enumerate}
(The equivalence mentioned in \ref{item:realization} above is the usual Yoneda equivalence for $3$-term sequences; see \cite[Def.~2.7]{NP19}.) 
Moreover, this triplet must satisfy the axioms 
(ET$1$)--\ref{ET4}, (ET$3)^{\op}$ and (ET$4)^{\op}$ as 
detailed in \cite[Def.~2.12]{NP19}. 
We recall the \ref{ET4} axiom after Definition~\ref{def:extri-terminology} since it used several times later.

For the remainder of Section~\ref{sec:ET}, let $(\C,\BE,\fs)$ be an extriangulated category.

\begin{definition}\label{def:extri-terminology}
The following terminology is now standard in extriangulated category theory; it is motivated by that used for exact categories (see \cite[App.~A]{Kel90}) and for triangulated categories.
\begin{enumerate}
\item
Let $A,C\in\C$ be objects and suppose $\delta\in\BE(C,A)$. 
Then $\delta\in\BE(C,A)$ is known as an \emph{$\BE$-extension}. 
For morphisms $a\colon A\to A'$ and $c\colon C' \to C$ in $\C$, we obtain two new $\BE$-extensions 
$a_{*}\delta\deff\BE(C,a)(\delta)\in\BE(C,A')$ and $c^{*}\delta\deff\BE(c,A)(\delta)\in\BE(C',A)$.

\item
Suppose $\fs(\delta) = [A\overset{f}{\lra}B\overset{g}{\lra}C]$ for some $\BE$-extension $\delta\in\BE(C,A)$. 
Then $A\overset{f}{\lra}B\overset{g}{\lra}C$ is called an \emph{$\fs$-conflation}, and 
the morphisms $f$ and $g$ are called an \emph{$\fs$-inflation} and an \emph{$\fs$-deflation}, respectively.
This situation is usually summarised by the diagram 
$A\overset{f}{\lra}B\overset{g}{\lra}C\overset{\delta}{\dashrightarrow}$, which is known as an \emph{$\fs$-triangle}.

\item
A \emph{morphism} of $\fs$-triangles 
from 
$A\overset{f}{\lra}B\overset{g}{\lra}C\overset{\delta}{\dashrightarrow}$ 
to 
$A'\overset{f'}{\lra}B'\overset{g'}{\lra}C'\overset{\delta'}{\dashrightarrow}$
is a triplet $(a,b,c)$ of morphisms 
$a\in\C(A,A')$, 
$b\in\C(B,B')$ and 
$c\in\C(C,C')$
such that $a_{*}\delta=c^{*}\delta'$ and so that the following diagram commutes. 
\[
\begin{tikzcd}[row sep=0.6cm]
    A 
        \arrow{r}{f}
        \arrow{d}{a}
    & B 
        \arrow{r}{g}
        \arrow{d}{b}
    & C 
        \arrow[dashed]{r}{\delta}
        \arrow{d}{c}
    & {}\\
    A' 
        \arrow{r}{f'}
    & B'
        \arrow{r}{g'}
    & C' 
        \arrow[dashed]{r}{\delta'}
    & {}
\end{tikzcd}
\]
\end{enumerate}
\end{definition}

In contrast to an exact category, an $\fs$-conflation is not a kernel-cokernel pair in general, but only a weak-kernel-weak-cokernel pair (see \cite[Prop.~3.3]{NP19}).
As expected, $\fs$-conflations play the role of short exact sequences or of distinguished triangles; see Examples~\ref{ex:exact_is_extri} and \ref{exam:triangulated-is-extri}.

With this terminology in place, let us recall \ref{ET4} which holds for $(\C,\BE,\fs)$; (ET$4)^{\op}$ is dual. 
\begin{enumerate}[label=\textup{(ET\arabic*)}]
\setcounter{enumi}{3}
\item\label{ET4}
Let $A\overset{f}{\lra}B\overset{f'}{\lra}D\overset{\delta}{\dra}$ and $B\overset{g}{\lra}C\overset{g'}{\lra}F\overset{\rho}{\dra}$ be $\fs$-triangles in $(\C,\BE,\fs)$. 
Then there is a commutative diagram 
\[
\begin{tikzcd}[column sep=1.7cm, row sep=0.6cm]
    A \arrow{r}{f}\arrow[equals]{d}{}& B \arrow{r}{f'}\arrow{d}{g}& D \arrow[dashed]{r}{\delta}\arrow{d}{d}& {} \\
    A \arrow{r}[swap]{h\deff g\circ f}& C\arrow{r}[swap]{h'}\arrow{d}{g'} & E \arrow[dashed]{r}{\delta'}\arrow{d}{e}& {} \\
      & F \arrow[equals]{r}{}\arrow[dashed]{d}{\rho}& F,\arrow[dashed]{d}{f'_{*}\rho} & {} \\
      & {}  &{}   &
\end{tikzcd}
\]
in $\C$, where 
$A\overset{h}{\lra}C\overset{h'}{\lra}E\overset{\delta'}{\dra}$ and $D\overset{d}{\lra}E\overset{e}{\lra}F\overset{f'_{*}\rho}{\dra}$ are $\fs$-triangles, 
and
$d^{*}\delta' = \delta$ and $e^{*}\rho=f_{*}\delta'$. 
\end{enumerate}

An extriangulated category $(\C,\BE,\fs)$ is simply denoted by $\C$ if there is no confusion.

The next lemma will be used in many places, and it follows from \cite[Prop.~1.20]{LN19} (see also \cite[Cor.~3.16]{NP19}, \cite[Prop.~2.22]{HS20}). 
It says that, like in the triangulated case, extriangulated categories admit weak pushouts and weak pullbacks.

\begin{lemma}
\label{lem:wPO_wPB}
The following properties hold.
\begin{enumerate}[label=\textup{(\arabic*)}]
\item\label{item:wPO}
For any $\fs$-triangle $A\overset{f}\lra B\overset{g}{\lra} C\overset{\delta}{\dra} $ together with a morphism $a\in\C(A, A')$ and an associated $\fs$-triangle $A'\overset{f'}\lra B'\overset{g'}{\lra} C\overset{a_{*}\delta}{\dra} $,
there exists $b\in\C(B,B')$ which gives a morphism of $\fs$-triangles 
\[
\begin{tikzcd}[row sep=0.6cm]
    A 
        \arrow{r}{f}
        \arrow{d}[swap]{a}
        \wPO{dr}
    & B 
        \arrow{r}{g}
        \arrow{d}{b}
    & C 
        \arrow[dashed]{r}{\delta}
        \arrow[equals]{d}{}
    & {}\\
    A' 
        \arrow{r}{f'}
    & B'
        \arrow{r}{g'}
    & C 
        \arrow[dashed]{r}{a_{*}\delta}
    & {}
\end{tikzcd}
\]
and makes 
$
\begin{tikzcd}[column sep=1.3cm, cramped]
A
    \arrow{r}{\begin{psmallmatrix}
        f\\a
    \end{psmallmatrix}}
& B\oplus A'
    \arrow{r}{\begin{psmallmatrix}
        b, \amph -f'
    \end{psmallmatrix}}
& B'
    \arrow[dashed]{r}{(g')^{*}\delta}
&{}
\end{tikzcd}
$
an $\fs$-triangle.
Furthermore, the commutative square ${\rm (wPO)}$ is a weak pushout of $f$ along $a$ (in the sense of Remark~\ref{rem:wPO}\ref{item:weak-PO-defn}).

\item\label{item:wPB}
Dually, for any $\fs$-triangle $A\overset{f}\lra B\overset{g}{\lra} C\overset{\delta}{\dra} $ together with a morphism $c\in\C(C',C)$ and an associated $\fs$-triangle $A\overset{f'}\lra B'\overset{g'}{\lra} C'\overset{c^{*}\delta}{\dra} $,
there exists $b\in\C(B',B)$ which gives a morphism of $\fs$-triangles
\[
\begin{tikzcd}[row sep=0.6cm]
    A 
        \arrow{r}{f'}
        \arrow[equals]{d}{}
    & B' 
        \arrow{r}{g'}
        \arrow{d}[swap]{b}
        \wPB{dr}
    & C' 
        \arrow[dashed]{r}{c^{*}\delta}
        \arrow{d}{c}
    & {}\\
    A 
        \arrow{r}{f}
    & B
        \arrow{r}{g}
    & C 
        \arrow[dashed]{r}{\delta}
    & {}
\end{tikzcd}
\]
and makes 
$
\begin{tikzcd}[column sep=1.3cm, cramped]
B'
    \arrow{r}{\begin{psmallmatrix}
        -g'\\b
    \end{psmallmatrix}}
& C'\oplus B
    \arrow{r}{\begin{psmallmatrix}
        c, \amph g
    \end{psmallmatrix}}
& C
    \arrow[dashed]{r}{f'_{*}\delta}
&{}
\end{tikzcd}
$
an $\fs$-triangle.
The commutative square ${\rm (wPB)}$ is a weak pullback of $g$ along $c$
(in the sense of Remark~\ref{rem:wPB}).
\end{enumerate}
\end{lemma}

We recall how exact categories and triangulated categories naturally give rise to extriangulated categories.

\begin{example}
\label{ex:exact_is_extri}
\cite[Exam.~2.13]{NP19}
Let $(\C,\E)$ be an exact category, such that 
the collection 
\[
\BE(C,A)\deff 
    \Ext^{1}_{\C}(C,A) 
    = \Set{[A\overset{f}{\lra} B\overset{g}{\lra} C]|A\overset{f}{\lra} B\overset{g}{\lra} C \text{ lies in }\E}
\]
is a set for all $A,C\in\C$. 
(This set-theoretic requirement is met for example if $\C$ is skeletally small, or if $\C$ has enough projectives or enough injectives.) 
We can extend this definition of $\BE$ on objects in $\C^{\op}\times\C$ to a biadditive functor $\C^{\op}\times\C\to\Ab$ using pushouts and pullbacks of (admissible) conflations. 
Declare that $A\overset{f}{\lra}B\overset{g}{\lra}C\overset{}{\dashrightarrow}$ 
is an $\fs$-triangle 
if and only if 
$A\overset{f}{\lra}B\overset{g}{\lra}C$
is a kernel-cokernel pair in $\E$.
Then $(\C,\BE,\fs)$ is an extriangulated category and we say it \emph{corresponds to} or \emph{is an exact category}. 
As a special case, if $\C$ is an abelian category and $\E$ consists of all the short exact sequences in $\C$, then we say $(\C,\BE,\fs)$ \emph{corresponds to} or \emph{is an abelian category}. 
\end{example}

\begin{example}\label{exam:triangulated-is-extri}
\cite[Prop.~3.22]{NP19} 
Let $(\C,[1],\triangle)$ be a triangulated category, where $[1]$ is the suspension automorphism and $\triangle$ is the triangulation consisting of (distinguished) triangles. 
Define $\BE(C,A)\deff\C(C,A[1])$ for $A,C\in\C$. Thus, the $\BE$-extensions in this case are just morphisms. 
If $a\colon A\to A'$ and $c\colon C' \to C$ are morphisms, and $h\in\BE(C,A)$, then 
$a_{*}h = \BE(C,a)(h) = a[1] \circ h $
and 
$c^{*}h = \BE(c,A)(h) = hc $. 
This yields a biadditive functor $\C^{\op}\times\C\to\Ab$. 
Lastly, declare that 
$\begin{tikzcd}[column sep=0.6cm, cramped]
    A
        \arrow{r}{f}
    & B 
        \arrow{r}{g}
    & C
        \arrow[dashed]{r}{h}
    &{}
\end{tikzcd}$
is an $\fs$-triangle 
if and only if 
$A\overset{f}{\lra}B\overset{g}{\lra}C\overset{h}{\lra}A[1]$
is a triangle in $\C$.
Then $(\C,\BE,\fs)$ is an extriangulated category and we say it \emph{corresponds to} or \emph{is a triangulated category}.
\end{example}

\subsection{The category \texorpdfstring{$\ET$}{ET} of extriangulated categories}

Just as for exact and triangulated categories, there is a notion of an exact functor between extriangulated categories.

\begin{definition}
\label{def:exact_functor}
Suppose $(\C,\BE,\fs)$, $(\C',\BE',\fs')$ and $(\C'',\BE'',\fs'')$ are extriangulated categories.
\begin{enumerate}[label=\textup{(\arabic*)}]

\item\label{item:exact_functor}
\cite[Def.~2.32]{B-TS21} 
An \emph{exact} (or \emph{extriangulated}) \emph{functor} from 
$(\C,\BE,\fs)$ to $(\C',\BE',\fs')$ 
is a pair $(F,\phi)$ of an additive functor $F\colon \C\to\C'$ and a natural transformation $\phi\colon \BE\Rightarrow\BE' \circ (F^{\op}\times F)$, 
such that if 
$A\overset{f}{\lra}B\overset{g}{\lra}C\overset{\delta}{\dashrightarrow}$ 
is an $\fs$-triangle, then 
$\begin{tikzcd}[column sep=1cm, cramped]
FA
    \arrow{r}{Ff}
& FB
    \arrow{r}{Fg}
& FC
    \arrow[dashed]{r}{\phi_{C,A}(\delta)}
&{}
\end{tikzcd}$
is an $\fs'$-triangle.
 
\item
\cite[Def.~2.11(2)]{NOS22} 
Suppose $(F,\phi)\colon (\C,\BE,\fs)\to (\C',\BE',\fs')$ and $(F',\phi')\colon (\C',\BE',\fs')\to (\C'',\BE'',\fs'')$ are exact functors. 
Then the composition 
$(F',\phi')\circ(F,\phi)\colon (\C,\BE,\fs)\to (\C'',\BE'',\fs'')$ 
is the exact functor $(F'',\phi'')$ given by 
$F''\deff F'\circ F$ 
and 
$\phi''\deff (\phi'\cdot(F^{\op}\times F))\circ\phi$, 
where $\phi'\cdot(F^{\op}\times F)$ denotes the whiskering of $\phi'$ with $F^{\op}\times F$.

\end{enumerate}
\end{definition}

If $(\C,\BE,\fs)$ and $(\C',\BE',\fs')$ are exact (resp.\ triangulated) categories, then 
an exact functor $(\C,\BE,\fs) \to (\C',\BE',\fs')$ 
coincides with the classical notion of an exact functor between exact categories (resp.\ an exact or triangulated functor between triangulated categories); see \cite[Thms.~2.34 and 2.33]{B-TS21}.

Exact functors between extriangulated categories allow us to consider the category of extriangulated categories.

\begin{definition}
We denote by $\ET$ the category of skeletally small extriangulated categories and exact functors between them. 
In addition, we let $\TR$, $\EX$ and $\AB$, respectively, denote the full subcategories of $\ET$ consisting of triangulated, exact and abelian categories, respectively.
\end{definition}

We restrict to skeletally small categories because we will usually be considering localizations or Grothendieck groups. 
If one further restricts to small extriangulated categories, then we obtain the $2$-category of small extriangulated categories, exact functors, and morphisms of extriangulated functors as introduced in \cite[Def.~2.11(3)]{NOS22}. 
This $2$-category-theoretic viewpoint has been developed in \cite{BTHSS23} for $n$-exangulated categories, which are a generalisation of extriangulated categories to higher homological algebra.

It is well known that an extension-closed subcategory of an exact category inherits an exact structure from the ambient category, whereas the analogue is not necessarily true for an extension-closed subcategory of a triangulated category. That is, $\EX$ is closed under taking extension-closed subcategories but $\TR$ is not. 
A similar issue arises when taking relative theories. 
Moreover, both $\EX$ and $\TR$ are not closed under general localizations in a useful sense. 
Extriangulated category theory has the flexibility to be closed under all these operations.

In the next two subsections, we recall the needed details about extension-closure and the localization theory of extriangulated categories. 
We recall the theory about relative structures later in Section~\ref{sec:application-to-tri-cat} where it will be used extensively.

\subsection{Extension-closed subcategories}
\label{sec:ext_closed_extri_cats}

Let $\N$ be a subcategory of the extriangulated category $(\C,\BE,\fs)$. Then $\N$ is said to be \emph{extension-closed} if it satisfies the following:
\begin{enumerate}[label=\textup{(\roman*)}]
    \item $\N$ is full, additive and closed under isomorphisms in $\C$, and 
    \item if $A\lra B\lra C$ is an $\fs$-conflation with $A,C\in\N$, then $B\in\N$. 
\end{enumerate}

The following is then easily verified.

\begin{proposition}
\label{prop:ext_closed_ET}
\cite[Rem.~2.18]{NP19}
Let $\N$ be an extension-closed subcategory of $(\C,\BE,\fs)$. Then there is an extriangulated category $(\N,\BE|_\N,\fs|_\N)$, where 
$\BE|_\N$ is the restriction of $\BE$ to $\N^{\op}\times \N$, and $\fs|_\N$ is the restriction of $\fs$ to $\BE|_\N$. 
Furthermore, the inclusion functor $\mathsf{inc}\colon\N\to\C$ is part of an exact functor 
$(\mathsf{inc},\iota)\colon (\N,\BE|_\N,\fs|_\N)\to(\C,\BE,\fs)$, where 
$\iota\colon \BE|_\N \Rightarrow \BE$ 
is the canonical inclusion. 
\end{proposition}

\subsection{Localization theory}
\label{sec_localization}

The localization theory of extriangulated categories was initiated in \cite{NOS22} in order to unify Verdier \cite{Ver96} and Serre quotients \cite{Gab62} (see Example~\ref{ex:list_mult_loc}). In this subsection we remain general, whereas in Section~\ref{sec:application-to-tri-cat} we will specialize to the localization of triangulated categories.

We begin by recalling the definition of a thick subcategory of an extriangulated category and then introducing exact sequences in $\ET$.

\begin{definition}\label{def:thick}
\cite[Def.~4.1]{NOS22} Let $\N$ be a subcategory of the extriangulated category $(\C,\BE,\fs)$. 
Then $\N$ is called \emph{thick} if:
\begin{enumerate}[label=\textup{(\roman*)}]
    \item $\N$ is full, additive and closed under direct summands in $\C$, and 
    \item $\N$ satisfies the \emph{$2$-out-of-$3$ property} for $\fs$-conflations, i.e\ if any two of $A, B$ or $C$ belong to $\N$ in an $\fs$-conflation $A\lra B\lra C$, then so does the third.
\end{enumerate}
\end{definition}

It follows that a thick subcategory $\N\sse\C$ is closed under isomorphisms and is extension-closed. Hence, $\N$ inherits an extriangulated structure by Proposition~\ref{prop:ext_closed_ET} and, in this case, we will sometimes write  $(\N,\BE|_\N,\fs|_\N)\sse(\C,\BE,\fs)$ is a thick subcategory. 
Furthermore, if $(\C,\BE,\fs)$ is a triangulated category, 
then a subcategory is thick in the sense of Definition~\ref{def:thick} if and only if it is thick in usual sense for triangulated categories (i.e.\ a triangulated subcategory that is closed under direct summands); see \cite[Exam.~4.8]{NOS22}.

\begin{definition}\label{def:Serre}
\cite[Def.~1.17]{Oga22b}
A thick subcategory $\N\sse\C$ is called a \emph{Serre} subcategory if any $\fs$-conflation $A\lra B\lra C$ with $B\in\N$ implies $A,C\in\N$.
\end{definition}

If $(\C,\BE,\fs)$ is an exact category, then the notion of a Serre subcategory as in Definition~\ref{def:Serre} is the same as the classical definition for exact categories given in \cite[Def.~4.0.35]{C-E98}.

We now introduce exact sequences of extriangulated categories in $\ET$, motivated by the notion of exact sequences in $\TR$ and $\AB$; see e.g.\ \cite[Sec.~2.3]{PV18}. The justification of our terminology is given by Proposition~\ref{prop:justification-for-exact-sequence-terminology} after we have recalled extriangulated localizations.

\begin{definition}
\label{def:exact_seq_ET}
Suppose $(\N,\BE|_\N,\fs|_\N)\sse(\C,\BE,\fs)$ is thick. 
We say that the sequence
\begin{equation}\label{eqn:exact-seq-in-ET}
\begin{tikzcd}[column sep=1.2cm]
(\N,\BE|_\N,\fs|_\N)
	\arrow{r}{(\inc,\iota)}
& (\C,\BE,\fs)
	\arrow{r}{(Q,\mu)}
& (\D,\BF,\ft)
\end{tikzcd}
\end{equation}
in $\ET$ is an \emph{exact sequence} of extriangulated categories, 
if the following conditions are fulfilled.
\begin{enumerate}[label=\textup{(\arabic*)}]
    \item\label{item:image-equals-kernel}
    $\N = \Im(\inc) = \Ker Q$ holds.

    \item\label{item:universality}
    For any exact functor 
    $(G,\psi)\colon (\C,\BE,\fs)\to (\C',\BE',\fs')$ 
    in $\ET$ with $G|_\N = G \circ \inc = 0$, 
    there is a unique exact functor
    $(G',\psi')\colon (\D,\BF,\ft) \to (\C',\BE',\fs')$ 
    such that $(G,\psi)=(G',\psi')\circ (Q,\mu)$.
	
\end{enumerate}

Exact sequences in $\TR$, $\EX$ and $\AB$ are defined similarly. 
\end{definition}

We will work under the following setup for the rest of Section~\ref{sec:ET}.

\begin{setup}
We suppose that $(\N,\BE|_\N,\fs|_\N)$ is a thick subcategory of a skeletally small extriangulated category $(\C,\BE,\fs)$.
\end{setup}

In the remainder of Subsection~\ref{sec_localization}, we review some sufficient conditions on the pair $(\C,\N)$ to guarantee the existence of an exact sequence in $\ET$ with end-term a certain localization of $\C$ dependent on $\N$.

\begin{definition}
Suppose $A\overset{f}{\lra} B\overset{g}{\lra} C$ is an $\fs$-conflation. 
The object $A$ is known as a \emph{cocone} of the $\fs$-deflation $g$ and $C$ a \emph{cone} of the $\fs$-inflation $f$. 
By \cite[Rem.~3.10]{NP19}, these objects are uniquely determined up to isomorphism by $g$ and $f$, respectively. 
Therefore, we write $\cocone(g)\deff A$ and $\cone(f)\deff C$.
\end{definition}

Associated to $\N$ we define several classes of morphisms in $\C$; 
we will compare Definitions~\ref{def:Sn_from_thick} and \ref{def:mors_assoc_to_wWH_cat} in Section~\ref{sec:wWET-cats} once we have seen that extriangulated categories are examples of weak Waldhausen categories (see Lemma~\ref{lem:compare-Lac-and-L}).

\begin{definition}
\label{def:Sn_from_thick}
\cite[Def.~4.3]{NOS22}, \cite[Lem.~2.12]{Oga22b}
Define subsets of $\C^{\to}$ as follows.
\begin{enumerate}[label=\textup{(\roman*)}]

\item
$\L \deff \Set{f\in\C^{\to} | f \ \textnormal{is an $\fs$-inflation with}\ \cone(f)\in\N}$

\item
$\L_{\sec}\deff\Set{ f\in\L | f\textnormal{ is a section}}$ 

\item
$\R \deff \Set{g\in\C^{\to} | g \ \textnormal{is an $\fs$-deflation with}\ \cocone(g)\in\N}$

\item
$\R_{\ret}\deff\Set{ g\in\R | g\textnormal{ is a retraction}}$

\item $\Sn$ is the smallest subclass of $\C^{\to}$ closed under composition containing both $\L$ and $\R$

\end{enumerate}
\end{definition}

A morphism $f\in\C^{\to}$ lies in $\Sn$ if and only if $f$ is a finite composition of morphisms from $\L \cup \R$; see \cite[p.~374]{NOS22}. 
It is clear that $\Sn$ contains all isomorphisms in $\C$ (as each isomorphism is e.g.\ an $\fs$-inflation with cone $0$), and $\Sn$ is closed under composition by definition. Furthermore, one can mimic the argument in Lemma~\ref{lem:basics_of_cofibration_and_weq}\ref{item:direct-sum-of-cofs-is-cof} to see that $\Sn$ is closed under finite direct sums.
Thus, $\Sn$ satisfies: 
\begin{enumerate}[label=\textup{(M\arabic*)}]
\setcounter{enumi}{-1}
    \item\label{MS0}
    $\Sn$ contains $\Iso\C$, and is closed under composition and taking finite direct sums.
\end{enumerate}
Consider the quotient functor 
$\overline{(-)}\colon \C\to \overline{\C} \deff \C/[\N]$, and define 
$\overline{\Sn} \deff \Set{ \overline{f} | f \in \Sn }\sse \overline{\C}^{\to}$. 
We denote by 
\[
\C/\N\deff \C[\Sn^{-1}] \cong\overline{\C}[\overline{\Sn}^{-1}]
\]
the localization of $\C$ at $\Sn$, where the isomorphism follows from the same argument as in Lemma~\ref{lem:homotopy_cat}. 
In addition, we denote the localization functor by $Q\colon \C\to\C/\N$. 
Under certain conditions, it indeed gives rise to an exact functor between extriangulated categories.

\begin{theorem}
\label{thm:NOS}
If $\Sn$ is saturated and $\overline{\Sn}$ satisfies conditions 
\textup{(MR1)}--\textup{(MR4)} of \cite[p.~343]{NOS22},
then there is an extriangulated category $(\C/\N,\widetilde{\BE},\widetilde{\fs})$ together with an 
exact functor $(Q,\mu)\colon (\C,\BE,\fs)\to (\C/\N,\widetilde{\BE},\widetilde{\fs})$, such that the following sequence is exact in $\ET$.
\begin{equation}\label{seq_exact_squence_in_ET}
\begin{tikzcd}[column sep=1.2cm]
(\N,\BE|_\N,\fs|_\N) 
    \arrow{r}{(\mathsf{inc},\iota)} 
&(\C,\BE,\fs) 
    \arrow{r}{(Q,\mu)} 
&(\C/\N,\widetilde{\BE},\widetilde{\fs})
\end{tikzcd}
\end{equation}
\end{theorem}

\begin{proof}
This follows from \cite[Thm.~3.8]{OS23} and \cite[Prop.~4.3]{ES22}.
\end{proof}

\begin{remark}\label{rem:on-theorem-4-14}
We make some remarks on Theorem~\ref{thm:NOS}.
\begin{enumerate}[label=\textup{(\arabic*)}]

\item\label{item:extri_localization}
We call the exact functor 
$(Q,\mu)\colon (\C,\BE,\fs)
	\to (\C/\N,\widetilde{\BE},\widetilde{\fs})$
the \emph{extriangulated localization of $\C$ with respect to $\N$}.
It is denoted more simply by $Q\colon \C\to\C/\N$ if the extriangulated structures are understood. 
We sometimes refer to \eqref{seq_exact_squence_in_ET} as an \emph{extriangulated localization sequence}. 

\item\label{item:saturation}
Although $\Sn$ being saturated is not necessary to prove \cite[Thm.~3.5]{NOS22}, it is essential to prove our main results and hence we assume this condition from the outset. Consequently, assuming (MR1) is redundant since the saturation of $\Sn$ implies (MR1). 

\item\label{item:Sn-bar}
The definition of $\overline{\Sn}$ we have used is not the original one given in \cite[Def.~3.1]{NOS22}. Indeed, let $\overline{\Sn}^{*}$ denote the closure of $\overline{\Sn}$ with respect to composition with isomorphisms. The condition $\overline{\Sn}=\overline{\Sn}^{*}$ is required in \cite[Thm.~3.5]{NOS22}, but this is guaranteed in our setup by \cite[Lem.~3.2(3)]{NOS22} when (MR1) is satisfied.

\item\label{item:MR2-for-ol-Sn}
The condition (MR2) says that $\overline{\Sn}$ is a multiplicative system in $\overline{\C}$. 

\item\label{item:conflations-in-localization}
The biadditive functor $\widetilde{\BE}$ and realization $\widetilde{\fs}$ are defined in \cite[Secs.~3.2--3.3]{NOS22}. 
A description of the $\widetilde{\fs}$-inflations and $\widetilde{\fs}$-deflations is given in \cite[Lem.~3.32]{NOS22}.

\end{enumerate}
\end{remark}

As mentioned earlier, Verdier and Serre quotients are motivating examples of extriangulated localizations; see \cite[Sec.~4]{NOS22}.

\begin{example}
\label{ex:list_mult_loc}
\cite[Exams.~4.8, 4.9]{NOS22}
\begin{enumerate}[label=\textup{(\arabic*)}]

\item\label{item:Verdier} (Verdier quotient.) 
Suppose $\N$ is thick in a (skeletally small) triangulated category $(\C,\BE,\fs)$. 
Then we have $\Sn=\L=\R$, 
$Q\colon \C \to \C/\N$ is the localization functor for the Verdier quotient of $\C$ by $\N$, 
and 
$\N \to \C \to \C/\N$ 
is an exact sequence in $\ET$.

\item (Serre quotient.) 
Suppose $\N$ is a Serre subcategory of a (skeletally small) abelian category $(\C,\BE,\fs)$. 
Then we have $\Sn=\L\circ\R$, 
$Q\colon \C \to \C/\N$ is the localization functor for the Serre quotient of $\C$ by $\N$, 
and 
$\N \to \C \to \C/\N$ 
is an exact sequence in $\ET$.

\end{enumerate}
\end{example}

In light of Example~\ref{ex:list_mult_loc}, we are in position to justify the choice of terminology in Definition~\ref{def:exact_seq_ET}.

\begin{proposition}\label{prop:justification-for-exact-sequence-terminology}
If \eqref{eqn:exact-seq-in-ET} is an exact sequence in $\ET$ that also lies in $\TR$ (resp.\ $\AB$), then it is the usual Verdier (resp.\ Serre) quotient. 
Conversely, a Verdier (resp.\ Serre) quotient $\N \into \C \to \D$ gives rise to an exact sequence \eqref{eqn:exact-seq-in-ET} in $\ET$.
\end{proposition}
\begin{proof}
First, assume \eqref{eqn:exact-seq-in-ET} is an exact sequence in $\ET$ in the sense of Definition~\ref{def:exact_seq_ET}, and that 
$(\N,\BE|_\N,\fs|_\N)$, 
$(\C,\BE,\fs)$
and 
$(\D,\BF,\ft)$
are triangulated categories. 
By definition, the subcategory $\N=\Ker Q$ is thick.
Thus, we can construct the Verdier quotient $\N\into \C\overset{Q'}{\lra} \C/\N$.
By universality of the Verdier quotient and \ref{item:universality} above, we have a triangle isomorphism $F\colon \D \cong \C/\N$ with $Q'=FQ$, and thus $Q \colon \C \to \D$ is a Verdier quotient.

Second, assume \eqref{eqn:exact-seq-in-ET} is an exact sequence in $\ET$ that lies in $\AB$. We must show that $\N=\Ker Q$ is a Serre subcategory of $\C$. 
Thus, let $A\lra N\lra C$ be a short exact sequence in $\C$ with $N\in\N$.
Applying $Q$, we get a $\ft$-conflation (i.e.\ short exact sequence as $\D$ is abelian) $QA\lra 0\lra QC$.
Thus, $QA\cong 0 \cong QC$ and hence  $A,B\in\Ker Q = \N$. Proceeding as in the triangulated case, we see that $Q \colon \C \to \D$ is (isomorphic to) the Serre quotient functor of $\C$ by $\N$. 

The converses follow from Example~\ref{ex:list_mult_loc}.
\end{proof}

\subsection{The Grothendieck group}
\label{sec:grothendieck-for-extri-cats}

Lastly in this section, we recall the definition of the Grothendieck group associated to an extriangulated category.

\begin{definition}\label{def:K0_of_extri-cat}
\cite[Sec.~4]{ZZ21} 
The \emph{Grothendieck group} of $(\C,\BE,\fs)$ is the abelian group  
\[
K_{0}(\C,\BE,\fs) 
	\deff 
        K^{\sp}_{0}(\C)  /  \braket{\, [A] - [B] + [C] \, |\begin{tikzcd}[column sep=0.5cm, ampersand replacement=\&]
			A \arrow{r}\& B \arrow{r}\& C \arrow[dashed]{r}\& {}
			\end{tikzcd} 
			\text{is an $\fs$-triangle}\, } .
\]
\end{definition}

Note that if $(\C,\BE,\fs)$ is the extriangulated category described in Example~\ref{ex:exact_is_extri} corresponding to a skeletally small exact category $(\C,\E)$, then the group $K_{0}(\C,\BE,\fs)$ coincides with the classical definition of the Grothendieck group $K_{0}(\C,\E)$ defined for an exact category. Similarly for a triangulated category.

Taking the Grothendieck groups of extriangulated categories provides a functor $K_{0}\colon \ET\to \Ab$.
To help understand the localization process, it is useful to investigate how exact sequences in $\ET$ behave under the functor $K_{0}$.
In this direction, Enomoto--Saito \cite[Cor.~4.32]{ES22} proved that a certain exact sequence in $\ET$ induces an exact sequence of Grothendieck groups, generalizing the results for Serre and Verdier quotients (see \cite[Ch.~VIII, Cor.~5.5]{Bas68} and \cite[VIII, Prop.~3.1]{SGA5}, respectively). 
We will apply Corollary~\ref{cor:Quillen_localization2} to an extriangulated localization sequence and recover \cite[Cor.~4.32]{ES22} in Corollary~\ref{cor:ES_via_Waldhausen}.

\section{Weak Waldhausen extriangulated categories}
\label{sec:wWET-cats}

In this section, we combine the frameworks of weak Waldhausen and of extriangulated categories. We will see 
that extriangulated categories give rise to a canonical weak biWaldhausen structure, and 
that an extriangulated localization sequence induces a localization sequence of weak Waldhausen categories which yields a right exact sequence of Grothendieck groups.

\subsection{Definitions and examples}

We begin by showing each extriangulated category has a canonical weak biWaldhausen structure.

\begin{example}
\label{ex:WET1}
Let $(\C,\BE,\fs)$ be an extriangulated category.
Define 
$\seq_{\fs}$ to be the class of all $\fs$-conflations, and $\veq_{\fs} \deff \Iso\C$. 
Then $(\C,\seq_{\fs}, \veq_{\fs})$ is a weak biWaldhausen category, where the cofibrations are all $\fs$-inflations and the fibrations are all $\fs$-deflations. 
Since the definition of an extriangulated category is self-dual (see \cite[Caution 2.20]{NP19}), it suffices to show that $(\C,\seq_{\fs}, \veq_{\fs})$ is weak Waldhausen.

Let us check the axioms laid out in Definition~\ref{def:weak_Wald_cat}.
First, note that $\cof = \cof_{\seq_{\fs}}$ is closed under compositions by \ref{ET4} and contains the isomorphisms in $\C$ as $\fs$ is an additive realization (see \cite[Def.~2.10]{NP19}). Thus, \ref{WC0} holds.
Axiom \ref{WC1} follows from $\fs$ being an additive realization, \cite[Prop.~3.7]{NP19} and \cite[Cor.~3.12]{NP19}.
Axiom \ref{WC2} follows from Lemma~\ref{lem:wPO_wPB}\ref{item:wPO}. 
It is clear that \ref{WW0} is satisfied.
Lastly, the Gluing Axiom \ref{WW1} is straightforward to verify using \cite[Cor.~3.6]{NP19}. 

In addition, we observe that $K_{0}(\C,\seq_{\fs}, \veq_{\fs}) = K_{0}(\C,\BE,\fs)$.
\end{example}

It is classical that the full subcategory of cofibrant objects in a closed model category forms a Waldhausen category; see 
e.g.\ \cite[Prop.~II.8.1, Lem.~II.8.8]{GJ99}.
Example~\ref{ex:WET1} immediately implies the following analogue.

\begin{example}
Let $(\C,\BE,\fs)$ be an extriangulated category and suppose $\M = (\emph{Fib}, \emph{Cof}, \BW)$ is an admissible model structure on $\C$ in the sense of \cite[Def.~5.5]{NP19}. Let $\U$ denote the full subcategory of \emph{cofibrant} objects in $\C$, i.e.\ $X\in\U$ if and only if the morphism $0\to X$ is in $\emph{Cof}$. 
Then, by \cite[Prop.~5.6]{NP19}, the subcategory $\U$ is part of a cotorsion pair $(\U,\U^{\perp_{1}})$ (see \cite[Def.~4.1]{NP19}).
In particular, $\U$ is extension-closed in $(\C,\BE,\fs)$ by \cite[Rem.~4.6]{NP19}. Hence, $\U$ inherits extriangulated, and thus also a weak biWaldhausen, structure from $\C$.
\end{example}

Example~\ref{ex:WET1} motivates the following terminology.

\begin{definition}\label{def:weak_Waldhausen_ET}
Let $(\C,\BE,\fs)$ be an extriangulated category, and suppose $(\C,\seq,\weq)$ is a weak Waldhausen category (not necessarily the canonical one).
If any distinguished sequence in $\seq$ is an $\fs$-conflation, then we call $(\C,\seq,\weq)$ a \emph{weak Waldhausen extriangulated category with respect to $(\C,\BE,\fs)$} (a \emph{wWET category}, for short). 
We say $(\C,\seq,\weq)$ a \emph{weak coWaldhausen extriangulated category with respect to $(\C,\BE,\fs)$} (or a \emph{wcoWET category}) if 
$(\C^{\op},\seq^{\op},\weq^{\op})$ is a wWET category with respect to $(\C^{\op},\BE^{\op},\fs^{\op})$. 
A \emph{weak biWaldhausen extriangulated category with respect to $(\C,\BE,\fs)$} (or a \emph{wbiWET category}) is a triplet $(\C,\seq,\weq)$ that is a wWET and a wcoWET category with respect to $(\C,\BE,\fs)$.
\end{definition}

\begin{example}\label{ex:extria_cat_is_wbiWET}
By Example~\ref{ex:WET1}, there is a canonical wbiWET category $(\C,\seq_{\fs},\veq_{\fs})$ associated to an extriangulated category $(\C,\BE,\fs)$. 
\end{example}

The next example is of a Waldhausen category that is not wWET (with respect to an extriangulated structure arising naturally in its context).

\begin{example}
\label{ex:contractible2}
Let $\A$ be a skeletally small abelian category with at least one non-monomorphism (e.g.\ $\A$ has a non-zero object).
As in Example~\ref{ex:contractible}, there is a weak Waldhausen category 
$(\A,\seq,\weq)$ where $\seq$ is the class of all right exact sequences in $\A$ and $\weq \deff\Iso\C$.

Let $(\A,\BE,\fs)$ denote the extriangulated structure on $\A$ induced by its abelian structure; see Example~\ref{ex:exact_is_extri}. 
Since an arbitrary right exact sequence is not a short exact sequence in general, the weak Waldhausen category $(\A,\seq,\weq)$ is not a wWET category with respect to $(\A,\BE,\fs)$. 
\end{example}

In the above example, if we had assumed $\A$ had a morphism that was not a weak kernel of any morphism, then it would follow that $(\A,\seq,\weq)$ is not wWET with respect to any extriangulated structure on $\A$.

The last example we give is weak Waldhausen but neither Waldhausen nor extension-closed in its ambient triangulated category.

\begin{example}
Consider the path algebra of the quiver $1\to 2\to 3$ over a field, its module category and the bounded derived category of that. Let $(\C,\BE,\fs)$ denote this triangulated category and let $(\C,\seq,\weq)$ be the corresponding weak Waldhausen structure (see Example~\ref{ex:WET1}). 
Consider the full, additive subcategory $\D\deff\add\{2,3,\substack{2\\3},\substack{1\\2}\}$.
This is not extension-closed in $\C$ as, e.g., $\substack{1\\2\\3}$ is not in $\D$. 
Define $\seq_{\D} \deff \seq\cap\D^{\to\to}$ (i.e.\ all $\fs$-conflations with all three terms lying in $\D$). 
Define $\weq_{\D}\deff\weq\cap\D^{\to}$ (i.e.\ all the isomorphisms of $\D$).
Notice that \ref{WC1} holds.

Note that this does not imply $\cof_{\D}$ coincides with $\cof\cap\D^{\to}$. In fact, $\cof_{\D}$ is much smaller: 
up to isomorphism, the only indecomposable cofibrations are $a\colon 3\to\substack{2\\3}$ and split monomorphisms with indecomposable domain, and well-defined compositions of these.
Notice that \ref{WC0} holds.

In particular, $\seq_{\D}$ consists only of the following $\fs$-conflations:
$
\begin{tikzcd}[column sep=1.3cm, cramped]
    3 \oplus 0
        \arrow{r}{a\oplus 0}
    & \substack{2\\3} \oplus X
        \arrow{r}{b\oplus \id_{X}}
    & 2 \oplus X,
\end{tikzcd}
$
$
\begin{tikzcd}[column sep=1.3cm, cramped]
    3 \oplus X
        \arrow{r}{a\oplus \id_{X}}
    & \substack{2\\3} \oplus X
        \arrow{r}{b\oplus 0}
    & 2 \oplus 0,
\end{tikzcd}
$
and split sequences with terms in $\D$. 
Now, if $A\to B \to D$ is in $\seq_{\D}$ and $A\overset{c}{\lra} C$ is a morphism in $\D$, then we can take a weak pushout in $\C$ to obtain an $\fs$-conflation of the form $C\to E \to D$, as well as another $\fs$-conflation $A \to B \oplus C \to E$. 
In order to show \ref{WC2}, we need to show that $E\in\D$. 
\begin{itemize}
    \item If $A\to B \to D$ is split, then $C\to E\to D$ splits and $E\cong C\oplus D\in\D$.

    \item Suppose $A\to B \to D$ is 
    $
\begin{tikzcd}[column sep=0.5cm, cramped]
    3
        \arrow{r}{a}
    & \substack{2\\3}
        \arrow{r}{b}
    & 2.
\end{tikzcd}
$
We have the following cases for $c$:
    \begin{itemize}
        \item $c$ is a section; 
        \item $c$ is a retraction, so either $c$ is either $\id_{3}\colon 3\to 3$ or the zero morphism $3 \to 0$; or
        \item $c=a$.
    \end{itemize}
In all cases, $C\to E\to D$ splits and $E\in\D$. 

    \item More general situations will lead to direct sums of the above and so $E\in\D$ again.
\end{itemize}

Lastly, \ref{WW0} and \ref{WW1} hold because they hold for $(\C,\seq,\weq)$. 
Thus, the triplet $(\D,\seq_{\D},\weq_{\D})$ is weak Waldhausen. Moreover, the weak Waldhausen structure on $\D$ is non-trivial, i.e.\ not just the split exact structure. The Grothendieck group $K_{0}(\D,\seq_{\D},\weq_{\D})$ is isomorphic to $\BZ^3$.
\end{example}

\subsection{\texorpdfstring{Theorem~\ref{thm:NOS}}{Theorem~4.14} induces a localization sequence in \texorpdfstring{$\wWald$}{wWald}}

For the rest of this section, we work under the following setup so Theorem~\ref{thm:NOS} applies.

\begin{setup}\label{setup:5}
Let $(\C,\BE,\fs)$ be a skeletally small extriangulated category with corresponding weak biWaldhausen category $(\C,\seq,\veq) = (\C,\seq_{\fs},\veq_{\fs})$.
Suppose $(\N,\BE|_\N,\fs|_\N)$ is a thick subcategory of $\C$, and that $\weq \deff \Sn$ is saturated and $\overline{\Sn}\sse(\C/[\N])^{\to}$ satisfies conditions 
\textup{(MR1)}--\textup{(MR4)} of \cite[p.~343]{NOS22}. 
\end{setup}

In particular, Setup~\ref{setup:5} implies there is an exact sequence \eqref{seq_exact_squence_in_ET} in $\ET$ by Theorem~\ref{thm:NOS}.

\begin{lemma}\label{lem:CSeqW-is-wbiWET}
The triplet $(\C,\seq,\weq)$ is a wbiWET category with respect to $(\C,\BE,\fs)$.
\end{lemma}

\begin{proof}
Note that \ref{WC0}--\ref{WC2} hold for $(\C,\seq,\weq)$ since they already hold for $(\C,\seq,\veq)$. 
Condition \ref{WW0} is clear.
To show the Gluing Axiom \ref{WW1}, consider the diagram \eqref{eqn:gluing-axiom}, where $x,y,z$ all lie in $\weq=\Sn$. Using Lemma~\ref{lem:wPO_wPB} and (ET$3$), we obtain a morphism
\begin{equation}
\label{diag:ex_seq_wWald_from_NOS}
    \begin{tikzcd}[column sep=1.5cm, row sep=0.6cm]
     A_{1} 
        \arrow{r}{\begin{psmallmatrix}
         f_{1}\\-c_{1}
        \end{psmallmatrix}}
        \arrow{d}{x}
    & B_{1}\oplus C_{1}
        \arrow{r}{\begin{psmallmatrix}
            b_{1}, \amph g_{1}
        \end{psmallmatrix}}
        \arrow{d}{y \oplus z}
    & D_{1}
        \arrow[dotted]{d}{w}
    \\
    A_{2} 
        \arrow{r}{\begin{psmallmatrix}
         f_{2}\\-c_{2}
        \end{psmallmatrix}}
    & B_{2}\oplus C_{2}
        \arrow{r}{\begin{psmallmatrix}
            b_{2}, \amph g_{2}
        \end{psmallmatrix}}
    & D_{2}
    \end{tikzcd}
\end{equation}
of $\fs$-triangles. 
Note that 
$y \oplus z 
\in\Sn$ 
by \ref{MS0}. 
Applying $Q$, we have that $Qx$ and 
$Q(y \oplus z)$ 
are isomorphisms in $\C/\N$. 
As $\mu \colon \BE \Rightarrow \widetilde{\BE}(Q-,Q-)$ is natural, we have that $(Qx,Q(y\oplus z), Qw)$ is a morphism of $\widetilde{\fs}$-triangles. 
Hence, by \cite[Cor.~3.6]{NP19}, we know $Qw$ is an isomorphism in $\C/\N$, and hence 
$w \in \Sn = \weq$ as $\Sn$ is saturated. 
Therefore, $(\C,\seq,\weq)$ is a wWET category with respect to $(\C,\BE,\fs)$. In fact, the above can be dualised and so we are done.
\end{proof}

Thus, using Lemma~\ref{lem:localization_sequence_wWald} and Remark~\ref{rem:acyclics-inherit-coWH}, there is a sequence 
\begin{equation}
\label{seq:localization_sequence_wWald5}
\begin{tikzcd}
    (\C^{\weq},\seq^{\weq},\veq^{\weq})
        \arrow[hook]{r}{\inc}
    &(\C,\seq,\veq)
        \arrow[hook]{r}{\id_{\C}} 
    &(\C,\seq,\weq)
\end{tikzcd}
\end{equation}
of wbiWET categories (with respect to $(\C,\BE,\fs)$) and the assumptions of Setup~\ref{setup:Quillen_loc1} are met. 
We may now compare Definitions~\ref{def:Sn_from_thick} and \ref{def:mors_assoc_to_wWH_cat}.

\begin{lemma}\label{lem:compare-Lac-and-L}
We have $\L^{\ac} = \L$ and $\R^{\ac} = \R$.
\end{lemma}
\begin{proof}
The containment $\L^{\ac}\sse\L$ follows from Lemma~\ref{lem:L_sec}. Conversely, if 
$
A\overset{f}{\rtail} B \twoheadrightarrow C
$
is an $\fs$-conflation (i.e.\ lies in $\seq$) with $C\in\N$, applying the exact functor $(Q,\mu)$ of extriangulated categories yields 
the $\widetilde{\fs}$-conflation
$
QA\overset{Qf}{\rtail} QB \twoheadrightarrow 0
$.
It follows that $Qf$ is an isomorphism, so $f\in\Sn$ as $\Sn = \weq$ is saturated. Thus, $f\in\cof\cap\weq = \L^{\ac}$.

The equality $\R = \R^{\ac}$ can be shown in a similar way thanks to Lemma~\ref{lem:weak-biWald-implies-R-is-T}.
\end{proof}

\begin{example}
    As a special case of Lemma~\ref{lem:compare-Lac-and-L}, if $\C$ is a triangulated category with a thick subcategory $\N$, then we have $\L^{\ac} = \L = \Sn = \R = \R^{\ac}$ by Example~\ref{ex:list_mult_loc}\ref{item:Verdier}. 
\end{example}

We now show that the exact sequences in $\ET$ coming from Theorem~\ref{thm:NOS} (which include Verdier and Serre quotients, see Example~\ref{ex:list_mult_loc}) give rise to localization sequences (see Definition~\ref{def:exact_sequence_wWald}) in $\wWald$ and induce right exact sequences of abelian groups under $K_{0}$.

\begin{proposition}
\label{prop:ex_seq_wWald_from_NOS}
The sequence \eqref{seq:localization_sequence_wWald5} is a localization sequence in $\wWald$ 
and 
\begin{equation}
\label{seq:homotopy_fiber5}
\begin{tikzcd}[column sep=1.5cm]
K_{0}(\C^{\weq},\seq^{\weq},\veq^{\weq}) 
    \arrow{r}{K_{0}(\inc)}
& K_{0}(\C,\seq,\veq)
    \arrow{r}{K_{0}(\id_{\C})}
& K_{0}(\C,\seq,\weq)
    \arrow{r}
& 0
\end{tikzcd}
\end{equation}
is a right exact sequence of their Grothendieck groups.
\end{proposition}

\begin{proof}
We have already seen above that \eqref{seq:localization_sequence_wWald5} is a sequence of weak biWaldhausen categories. 
From Definition~\ref{def:Sn_from_thick} and the discussion immediately following it, we know a morphism in $\Sn$ is a finite composition of morphisms from $\L \cup \R$. 
By Lemma~\ref{lem:compare-Lac-and-L}, this implies 
$\weq = \Sn$ consists of morphisms that are finite compositions of morphisms from 
$\L^{\ac} \cup \R^{\ac}\sse \L^{\ac} \cup \R^{\ac} \cup \veq$. 
Hence, Corollary~\ref{cor:Quillen_localization2} applies and we are done.
\end{proof}

An example of a localization sequence in $\wWald$ that does not sit in $\ET$ in general is given in Theorem~\ref{thm:ex_seq_wWald_from_abel_loc}.

\begin{lemma}\label{lem:acyclics-is-N}
We have $\C^{\weq} = \N$ and $(\C^{\weq},\seq^{\weq},\veq^{\weq})$
is the weak biWaldhausen category corresponding to the extriangulated category 
$(\N,\BE|_\N,\fs|_\N)$.
\end{lemma}
\begin{proof}
Note that 
\begin{align*}
X\in\C^{\weq} 
    &\iff 0\to X \text{ lies in }\weq = \Sn && \text{by definition of $\C^{\weq}$}\\
    &\iff Q(0\to X) \text{ is an isomorphism} && \text{as $\weq$ is saturated}\\
    &\iff X\in \Ker Q && \text{see \eqref{eqn:KerQ}}\\
    &\iff X\in \N  &&\text{by Theorem~\ref{thm:NOS},}
\end{align*}
and hence $\N = \C^{\weq}$. 
Furthermore, since 
$
\seq^{\weq} 
    = \seq \cap (\C^{\weq})^{\to\to} 
    = \seq \cap \N^{\to\to} 
$
is just the collection of $\fs$-conflations with all terms in $\N$, the second assertion follows.
\end{proof}

\subsection{\texorpdfstring{Theorem~\ref{thm:NOS}}{Theorem~4.14} induces a right exact sequence in \texorpdfstring{$\Ab$}{Ab}}

Recall that we defined the homotopy category of a weak Waldhausen category in Section~\ref{subsec:homotopy-category} 
(see Definition~\ref{def:homotopy_cat}). 
Such homotopy categories arise in extriangulated localizations as we now explain.

\begin{example}
\label{ex:exact_localization_homotopy_cat}
Let \eqref{seq_exact_squence_in_ET} be an extriangulated localization sequence in the sense of Theorem~\ref{thm:NOS}. 
In particular, there is an extriangulated localization functor 
$(Q,\mu)\colon 
(\C,\BE,\fs)\to(\C / \N ,\widetilde{\BE},\widetilde{\fs})$. 
Setting $\seq'$ to be the class of all $\widetilde{\fs}$-conflations and $\weq'$ to be the class of isomorphisms in $\C / \N$, 
we have that the triplet 
$(\C/\N,\seq',\weq')$ is the wbiWET category corresponding to $(\C/\N,\widetilde{\BE},\widetilde{\fs})$ 
(see Examples~\ref{ex:WET1} and \ref{ex:extria_cat_is_wbiWET}). 
Moreover, as discussed earlier in this section (see Proposition~\ref{prop:ex_seq_wWald_from_NOS}), we have the localization sequence 
\eqref{seq:localization_sequence_wWald5} 
in $\wWald$, 
and it follows from \cite[Lem.~3.32]{NOS22} that $(\C/\N,\seq',\weq')$ is the weak Waldhausen homotopy category 
$\Ho(\C,\seq,\weq)$. 
Lastly, we note that the localization functor $Q$ is an exact functor 
$(\C,\seq,\weq)\to \Ho(\C,\seq,\weq)$ 
of wbiWET categories.
\end{example}

We can now combine some of our key results to see that an extriangulated localization sequence induces a right exact sequence of Grothendieck groups, and in particular recover \cite[Cor.~4.32]{ES22} with a new proof.
We remark that the hypotheses we are currently assuming are equivalent to those assumed in \cite[Cor.~4.32]{ES22}, which are \cite[Cond.~4.4]{ES22} and that $\Sn$ is saturated.

\begin{corollary}[Enomoto--Saito's Extriangulated Localization Theorem]
\label{cor:ES_via_Waldhausen}
The extriangulated localization sequence 
\eqref{seq_exact_squence_in_ET} of Theorem~\ref{thm:NOS} 
induces the following right exact sequence in $\Ab$.
\begin{equation}\label{eqn:ESS22-cor-res-in-Ab}
\begin{tikzcd}[column sep=1.5cm]
K_{0}(\N, \BE|_{\N},\fs|_{\N})
    \arrow{r}{K_{0}(\inc)}
& K_{0}(\C,\BE,\fs)
    \arrow{r}{K_{0}(Q)}
& K_{0}(\C/\N,\widetilde{\BE},\widetilde{\fs})
    \arrow{r}
& 0 
\end{tikzcd}
\end{equation}
\end{corollary}

\begin{proof}
Proposition~\ref{prop:ex_seq_wWald_from_NOS} tells us that \eqref{seq:homotopy_fiber5} is exact. 
We have the equalities:
\begin{align*}
K_{0}(\C^{\weq},\seq^{\weq},\veq^{\weq}) 
    &= K_{0}(\N,\BE|_\N,\fs|_\N) 
    && \text{by Lemma~\ref{lem:acyclics-is-N} and Example~\ref{ex:WET1},}\\
K_{0}(\C,\seq,\veq) 
    &= K_{0}(\C,\BE,\fs)
    && \text{by Example~\ref{ex:WET1}, and}\\
K_{0}(\Ho(\C,\seq,\weq)) 
    &= K_{0}(\C/\N,\widetilde{\BE},\widetilde{\fs}) 
    &&\text{by Example~\ref{ex:exact_localization_homotopy_cat} and Example~\ref{ex:WET1}.}
\end{align*}
Note that Setup~\ref{setup:section-3-1} is satisfied, so by Proposition~\ref{prop:K0_of_homotopy_cat} there is the isomorphism 
\[
K_{0}(Q)\colon K_{0}(\C,\seq,\weq) \overset{\cong}{\lra} K_{0}(\Ho(\C,\seq,\weq)).
\]
Therefore, \eqref{eqn:ESS22-cor-res-in-Ab} is exact.
\[
\begin{tikzcd}[column sep=1.3cm, row sep=0.5cm]
K_{0}(\C^{\weq},\seq^{\weq},\veq^{\weq}) 
    \arrow{r}{K_{0}(\inc)}
    \arrow[equals]{d}
& K_{0}(\C,\seq,\veq)
    \arrow{r}{K_{0}(\id_{\C})}
    \arrow[equals]{d}
& K_{0}(\C,\seq,\weq)
    \arrow{r}
    \arrow{d}{K_{0}(Q)}[swap]{\cong}
& 0
\\
K_{0}(\N,\BE|_\N,\fs|_\N)
    \arrow{r}{K_{0}(\inc)}
& K_{0}(\C,\BE,\fs)
    \arrow{r}{K_{0}(Q)}
& K_{0}(\C/\N,\widetilde{\BE},\widetilde{\fs})
    \arrow{r}
& 0 
\end{tikzcd}\]\end{proof}

\section{Application to the index in triangulated categories}
\label{sec:application-to-tri-cat}

In this section, we specialize to the case when $(\C,\BE,\fs)$ is a skeletally small, idempotent complete triangulated category with suspension functor $[1]$. 
The index with respect to an $n$-cluster tilting subcategory $\X$ of $\C$ was introduced in \cite{Jor21}, and it induces an isomorphism between $K_{0}^{\sp}(\X)$ and $K_{0}(\C,\BE^{R}_{\X^{\perp_{0}}},\fs^{R}_{\X^{\perp_{0}}})$, the Grothendieck group of a \emph{relative} extriangulated structure of $(\C,\BE,\fs)$; see Definition~\ref{def:n-CT-triangulated}, Proposition~\ref{prop:relative2}, Corollary~\ref{cor:JS_index}. 
This isomorphism was proved in this generality in \cite{OS23}, where it was called the \emph{JS index isomorphism}, 
but it first manifested as \cite[Prop.~4.11]{PPPP19}, 
which has led to the index being defined with respect to more general subcategories 
\cite{JS23,FJS24,FJS24b}.

\begin{definition}
\label{def:n-CT-triangulated}
Let $n\geq 2$ be an integer. 
A subcategory $\X$ of a triangulated category $\C$ is called \emph{$n$-cluster tilting} if $\X$ is functorially finite and 
\begin{align*}
    \X  &= \Set{C\in\C | \C(\X,C[i])=0\text{ for all } 1 \leq i \leq n-1}\\
        &= \Set{C\in\C | \C(C,\X[i])=0\text{ for all } 1 \leq i \leq n-1}.
\end{align*}  
\end{definition}

If $(\C,\BE,\fs)$ is an extriangulated category and $\mathbb{F}\sse\BE$ is an additive subfunctor (see \cite[Def.~3.7]{HLN21}) such that  $(\C,\mathbb{F},\fs|_\mathbb{F})$ is an extriangulated category, then we say 
$(\C,\mathbb{F},\fs|_\mathbb{F})$ is \emph{relative} to or a \emph{relative theory} of $(\C,\BE,\fs)$. 
The identity $\id_{\C}$ and inclusion $\BF\sse\BE$ constitute an exact functor $(\C,\mathbb{F},\fs|_\mathbb{F})\to(\C,\BE,\fs)$, witnessing that 
$(\C,\mathbb{F},\fs|_\mathbb{F})$ is an extriangulated subcategory of $(\C,\BE,\fs)$ in the sense of \cite[Def.~3.7]{Hau21}
(cf.\ \cite[Thm.~2.12]{JS21}). 
We need the following before we recall the specific relative structure we use.

\begin{definition}
\label{def:Cone-CoCone}
If $\U,\V$ are subcategories of $(\C,\BE,\fs)$, 
then $\cone(\V,\U)$ denotes the full subcategory of $\C$ consisting of objects $X$ for which there is an $\fs$-conflation $V\lra U\lra X$ with $U\in\U$ and $V\in\V$.
The full subcategory $\cocone(\V,\U) \sse \C$ is defined similarly. 
\end{definition}

If $(\C,\BE,\fs)$ is a triangulated category, then it follows that 
$\C = \cone(\V,\U)$ if and only if $\C = \cocone(\V,\U)$ since we can rotate triangles in $\C$.

\begin{proposition}
\label{prop:relative2}
Suppose $(\C,\BE,\fs)$ is a triangulated category with an extension-closed subcategory $\N\sse\C$ that is closed under direct summands and satisfies $\cone(\N,\N)=\C$. 
Then there are extriangulated categories 
\[
\C^{L}_{\N}\deff(\C,\BE^L_\N,\fs^L_\N),
\quad 
\C^{R}_{\N}\deff(\C,\BE^R_\N,\fs^R_\N),
\quad 
\C_{\N}\deff(\C,\BE_\N,\fs_\N),
\]
all relative to $(\C,\BE,\fs)$, 
where $\BE_\N\deff\BE^L_\N\cap\BE^R_\N$, and 
\begin{align}
\BE^L_\N(C,A) 
    &\deff \Set{ h\in\BE(C,A)  | \forall x\colon N\to C \text{ with } N\in\N\text{, we have } hx\in[\,\N[1]\,] }\nonumber\\
    &\,= \Set{ h\colon C\to A[1] | h[-1]\textnormal{\ factors through an object in\ }\N},\nonumber\\
\BE^R_\N(C,A) 
    &\deff \Set{ h\in\BE(C,A) | \forall y\colon A\to N \text{ with } N\in\N\text{, we have } y\circ h[-1]\in[\,\N[-1]\,] }\label{eqn:definition-of-ERN}\\
    &\,= \Set{h\colon C\to A[1] | h\textnormal{\ factors through an object in\ }\N}.\label{eqn:definition2-of-ELN}
\end{align}
\end{proposition}

\begin{proof}
That $\C^{R}_{\N} = (\C,\BE^R_\N,\fs^R_\N)$, as defined in \eqref{eqn:definition-of-ERN}, is an extriangulated category follows from \cite[Prop.~2.1]{Oga22b}. 
The equality \eqref{eqn:definition2-of-ELN} is \cite[Lem.~3.14]{OS23}. 
The assertions concerning $\C^{L}_{\N}$ follow similarly. 
It also follows that $\C_{\N}$ is an extriangulated category. 
\end{proof}

We need one last definition before we can recall Corollary~\ref{cor:JS_index}.

\begin{definition}
\label{def:X_resolution}
\cite[Def.~4.2]{OS23}
Suppose $(\C,\BE,\fs)$ is an extriangulated category with a subcategory $\X\sse\C$. 
The top row of 
\begin{equation}
\label{seq:X_resolution}
\begin{tikzcd}[column sep=0.5cm]
X_{n} 
    \arrow{rr}{f_{n-1}}
    \arrow[equals]{dr}
&{}& X_{n-1}
    \arrow{rr}{f_{n-2}g_{n-1}}
    \arrow{dr}{g_{n-1}}
&{}& \cdots 
    \arrow{rr}{f_1 g_2}
&{}& X_{1}
    \arrow{rr}{f_0 g_1}
    \arrow{dr}{g_{1}}
&{}& X_{0}
    \arrow{rr}{g_0}
    \arrow{dr}{g_{0}}
&{}& C
\\
{}
&C_{n}
    \arrow{ur}{f_{n-1}}
&{} & C_{n-1}
&{} & C_{2}
    \arrow{ur}{f_{1}}
&{} & C_{1}
    \arrow{ur}{f_{0}}
&{}& C_{0}
    \arrow[equals]{ur}
&{}
\end{tikzcd}
\end{equation}
is said to be a \emph{finite $\X$-resolution} of $C\in\C$ if 
$C_{i+1}\overset{f_i}{\lra} X_i\overset{g_i}{\lra} C_i$ 
is an $\fs$-conflation 
for each $0\leq i \leq n-1$
with $X_{i}\in\X$ 
for all $0\leq i \leq n$.
\end{definition}

\begin{example}
\label{example:n-CT-resolution}
\cite[Exam.~4.3]{OS23} 
Suppose $\C$ is an idempotent complete triangulated category and that $\X\sse\C$ is $n$-cluster tilting. Then each object $C\in\C$ admits a finite $\X$-resolution
$
X_{n-1}\to \cdots\to X_1\to X_0\to C
$ 
of length at most $n-1$. 
In fact, this $\X$-resolution is an $\X$-resolution in the relative extriangulated $\C^{R}_{\N}$ where 
$\N \deff \X^{\perp_{0}} \deff \Set{ C\in\C | \C(\X,C) = 0}$. 
\end{example}

The statement of Corollary~\ref{cor:JS_index} below differs slightly from the statement of \cite[Cor.~5.5]{OS23}: an extriangulated structure $\C^{\X}_{R}$ is used in \cite[Cor.~5.5]{OS23}, but we have $\C^{\X}_{R} = \C^{R}_{\N}$ by \cite[Lem.~3.15]{OS23} where $\N \deff \X^{\perp_{0}}$.
Furthermore, special cases of this result have been shown in \cite[Prop.~4.11]{PPPP19} and \cite[Thm.~4.10]{JS23}.

\begin{corollary}[JS index isomorphism]
\label{cor:JS_index}
\cite[Cor.~5.5]{OS23}
Let $(\C,\BE,\fs)$ be a skeletally small, idempotent complete triangulated category, 
$\X\subseteq \C$ an $n$-cluster tilting subcategory, and 
$\N \deff \X^{\perp_{0}}$. 
Then there exists an isomorphism
\begin{align}
\ind_{\X}\colon K_{0}(\C^{R}_{\N}) & \overset{\cong}{\lra}K_{0}^{\sp}(\X) \label{eqn:JS_index}\\
[C] & \longmapsto \ind_{\X}(C) \deff \sum_{i=0}^{n-1}(-1)^i[X_i]  \nonumber\\
[X] & \longmapsfrom [X], \nonumber
\end{align}
where 
$
X_{n-1}\to \cdots\to X_1\to X_0\to C
$ 
is a finite $\X$-resolution of $C$ in $\C^{R}_{\N}$.
\end{corollary}

The element $\ind_{\X}(C) = \sum_{i=0}^{n-1}(-1)^i[X_i]$ in $K_{0}^{\sp}(\X)$ is the \emph{triangulated index of $C$ with respect to $\X$} (see \cite[Def.~3.3]{Jor21}), and is a generalization of the index Palu defined in \cite[Def.~2.1]{Pal08} for $2$-cluster tilting subcategories. 
Our goal in this section is to understand and give a different proof of Corollary~\ref{cor:JS_index} via our Weak Waldhausen Localization Theorem~\ref{thm:Quillen_localization1_new}.

In the remainder of Section~\ref{sec:application-to-tri-cat}, we assume the following.

\begin{setup}
\label{setup:abelian_local}
Suppose $(\C,\BE,\fs)$ corresponds to a skeletally small, idempotent complete, triangulated category with suspension $[1]$. 
Let $\N$ denote an extension-closed 
subcategory of $\C$, which is closed under direct summands and such that  $\cone(\N,\N)=\C$. 
\end{setup}

\subsection{An abelian localization of a triangulated category}
\label{sec:abelian_local_of_tri_cats}

In this subsection we study further the relative extriangulated categories defined in Proposition~\ref{prop:relative2}. In particular, in Theorem~\ref{thm:ex_seq_wWald_from_abel_loc} we will see that the groups 
$K_{0}(\C^{L}_{\N})$, $K_{0}(\C^{R}_{\N})$ and $K_{0}(\N,\BE|_\N,\fs|_\N)$ are all isomorphic.

\begin{remark}
\label{rem:extriangulated_structures_on_N}
We record some facts on the extriangulated structures inherited by $\N$.
\begin{enumerate}[label=\textup{(\arabic*)}]
    \item \label{item:N-extension-closed-in-all-relative-structures}
    Since $\N$ is extension-closed in $(\C,\BE,\fs)$, it is still extension-closed in the relative extriangulated subcategories $\C^{L}_{\N}$, $\C^{R}_{\N}$ and $\C_{\N}$.
    
    \item \label{item:inherited-structures-on-N-all-the-same}
    We have equalities  
    \[
    (\N,\BE|_\N,\fs|_\N)
        = (\N,\BE_\N^R|_\N,\fs_\N^R|_\N) 
        = (\N,\BE_\N^L|_\N,\fs_\N^L|_\N) 
        = (\N,\BE_\N|_\N,\fs_\N|_\N) 
    \]
    as extriangulated categories.
    Indeed, if $N_0\lra N_1\lra N_2$ is an $\fs$-conflation whose terms $N_i$ all lie in $\N$, then it is clearly an $\fs_\N$-conflation given the descriptions of $\BE^{L}_{\N}$ and $\BE^{R}_{\N}$ in Proposition~\ref{prop:relative2}. 
\end{enumerate}
\end{remark}

With respect to the relative structure $\C_{\N}$, the pair $(\C,\N)$ yields a saturated class $\Sn\sse\C^{\to}$ of morphisms in $\C$ so that $\overline{\Sn}$ satisfies the needed conditions (MR1)--(MR4) to obtain an exact sequence in $\ET$.
Although $\N$ may not be thick in $(\C,\BE,\fs)$ in general, passing to $\C_{\N}$ allows us to take an extriangulated localization in the sense of Section~\ref{sec_localization}.
Theorem~\ref{thm:localization_of_tri_cat} summarises some key results from \cite{Oga22b} that we use. We need a definition in preparation.

\begin{definition}
\label{def:left_right_exact_functors}
\cite[Def.~2.7]{Oga21} 
An additive functor $F\colon (\D,\BF,\ft) \to \A$ from an extriangulated category $(\D,\BF,\ft)$ to an abelian category $\A$ is called \emph{right exact} if 
$\begin{tikzcd}[column sep=0.7cm, cramped]
        FA \arrow{r}{Ff}& FB \arrow{r}{Fg}& FC \arrow{r}& 0 
\end{tikzcd}$
is exact in $\A$ whenever
$\begin{tikzcd}[column sep=0.5cm, cramped]
        A \arrow{r}{f}& B \arrow{r}{g}& C
\end{tikzcd}$
is a $\ft$-conflation. 
    \emph{Left exact} functors are defined dually.
\end{definition}

\begin{theorem}
\label{thm:localization_of_tri_cat}
We have the following. 
\begin{enumerate}[label={\textup{(\arabic*)}}]

\item 
\label{item:coresolving_subcat}
$\N$ is closed under taking cones of $\fs_\N^R$-inflations and cocones of $\fs_\N^L$-deflations. That is, if 
$\begin{tikzcd}[column sep=0.5cm, cramped]
        A \arrow{r}{f}& B \arrow{r}{}& \cone(f)
\end{tikzcd}$
is an $\fs_\N^R$-conflation and $A,B\in\N$, then $\cone(f)\in\N$; closure under cocones of $\fs_\N^L$-deflations is defined similarly.

\item
\label{item:N_thick}
$\N$ is thick subcategory of $\C_{\N}$, and 
the corresponding class $\Sn\sse\C^{\to}$ defined in Definition~\ref{def:Sn_from_thick} is saturated and $\overline{\Sn}$ satisfies (MR1)--(MR4) of \cite[p.~343]{NOS22}. 
In addition, we have 
$\R_{\ret}\circ\L= \Sn = \R\circ\L_{\sec}$. 

\item
\label{item:extri_local}
There is an extriangulated localization $(Q,\mu)\colon \C_{\N}\to (\C/\N,\widetilde{\BE_\N},\widetilde{\fs_\N})$ of $\C_{\N}$ with respect to $\N$, 
and it gives rise to the following exact sequence in $\ET$.
\begin{equation}\label{eqn:exact-seq-in-ET-for-index}
\begin{tikzcd}[column sep=1.2cm]
(\N,\BE|_\N,\fs|_\N) \arrow{r}{(\mathsf{inc},\iota)} 
&\C_{\N} \arrow{r}{(Q,\mu)} 
&(\C/\N,\widetilde{\BE_\N},\widetilde{\fs_\N})
\end{tikzcd}
\end{equation}

\item
\label{item:N_serre_local_is_abelian}
Moreover, the extriangulated category $(\C/\N,\widetilde{\BE_\N},\widetilde{\fs_\N})$ corresponds to an abelian category 
and 
$Q\colon \C \to \C / \N$ 
is a cohomological functor. 

\item \label{item:abel_loc}
The functor $Q \colon \C^{R}_{\N} \to \C/\N$ (resp.\ $Q \colon \C^{L}_{\N} \to \C/\N$) is right exact (resp.\ left exact).
\end{enumerate}
\end{theorem}

\begin{proof}
For \ref{item:coresolving_subcat} and \ref{item:N_thick}, see \cite[Thm.~3.12(1),(2)]{OS23}. 
Part \ref{item:extri_local} follows from Theorem~\ref{thm:NOS} given \ref{item:N_thick} (see also \cite[Thm.~3.12(3)]{OS23}).
Part \ref{item:N_serre_local_is_abelian} follows from \cite[Thm.~3.16]{OS23} and \ref{item:abel_loc} is \cite[Cor.~4.4]{Oga22b}. 
\end{proof}

Let us give a useful characterization of the morphisms in $\Sn$ obtained in Theorem~\ref{thm:localization_of_tri_cat}\ref{item:N_thick}, by describing the morphisms in $\C$ that become epic (resp.\ monic) in $\C/\N$ via $Q$.

\begin{lemma}
\label{lem:char_of_Sn}
Let 
$
\begin{tikzcd}[column sep=0.5cm, cramped]
A
	\arrow{r}{f}
& B
	\arrow{r}{g}
& C
	\arrow{r}{h}
& A[1]
\end{tikzcd}
$
be a triangle in $\C$.
Then the statements \ref{Qg-epic}--\ref{h-factors} below are equivalent.
    \begin{enumerate}[label={\textup{(\arabic*)}}]
        \item\label{Qg-epic}
            $Qg$ is epic in $\C/\N$.
        
        \item\label{g-defl}
            $g$ is an $\fs_\N^R$-deflation and 
            $\begin{tikzcd}[column sep=0.5cm, cramped]
                A 
                    \arrow{r}{f}
                & B 
                    \arrow{r}{g}
                & C
                    \arrow[dashed]{r}{h}
                & {}
            \end{tikzcd}$
            is an $\fs_\N^R$-triangle.
        
        \item\label{h-factors}
            $h$ factors through an object in $\N$.
    \end{enumerate}
Dually, the statements \ref{Qf-monic}--\ref{h-1-factors} below equivalent.
    \begin{enumerate}[label={\textup{(\arabic*')}}]
        \item\label{Qf-monic}
            $Qf$ is monic in $\C/\N$.
        \item\label{f-infl}
            $f$ is an $\fs_\N^L$-inflation and 
            $\begin{tikzcd}[column sep=0.5cm, cramped]
                A 
                    \arrow{r}{f}
                & B 
                    \arrow{r}{g}
                & C
                    \arrow[dashed]{r}{h}
                & {}
            \end{tikzcd}$
            is an $\fs_\N^L$-triangle.
        \item\label{h-1-factors}
            $h[-1]$ factors through an object in $\N$.
    \end{enumerate}
In particular, $g\in\Sn$ if and only if both $f$ and $h$ factor through objects in $\N$.
\end{lemma}

\begin{proof}
The equivalence $\ref{Qg-epic} \Leftrightarrow \ref{h-factors}$
follows from \cite[Lem.~2.6]{Oga22b}, 
and 
$\ref{g-defl}\Leftrightarrow \ref{h-factors}$ is just Proposition~\ref{prop:relative2}. 
The proofs of the equivalences  
$\ref{Qf-monic}
    \Leftrightarrow \ref{f-infl}
    \Leftrightarrow \ref{h-1-factors}
$ 
are dual; and the last assertion is an immediate consequence.
\end{proof}

\subsection{A localization sequence in \texorpdfstring{$\wWald$}{wWald} induced by \texorpdfstring{$\C^{R}_{\N}$}{CRN}}
\label{sec:RES_for_triangulated}

Although \eqref{eqn:exact-seq-in-ET-for-index} 
is an exact sequence in $\ET$, 
the sequences 
$\N\into\C_{\N}^R\overset{Q}{\lra}\C/\N$ and $\N\into\C_{\N}^L\overset{Q}{\lra}\C/\N$ 
need not be.
They do, however, induce localization sequences in $\wWald$. We will use Theorem~\ref{thm:Quillen_localization1_new} to show this in Theorem~\ref{thm:ex_seq_wWald_from_abel_loc}, and so 
we setup some notation to help the exposition. 

\begin{notation}\label{notation:section-6}
We let $(\C,\seq^{R}_{\N}, \veq)$ (resp.\ $(\C,\seq_{\N}, \veq)$) denote the canonical wbiWET category corresponding to the extriangulated category $\C^{R}_{\N}$ (resp.\ $\C_{\N}$); see Example~\ref{ex:extria_cat_is_wbiWET}. That is, $\seq^{R}_{\N}$ is the collection of all $\fs^{R}_{\N}$-conflations, 
$\seq_{\N}$ is the collection of all $\fs_{\N}$-conflations, 
and $\veq = \Iso\C$. 
Accordingly, we denote by $\cof^{R}_{\N}$ (resp.\ $\fib^{R}_{\N}$) the cofibrations (resp.\ fibrations) in $(\C,\seq^{R}_{\N}, \veq)$.

By Theorem~\ref{thm:localization_of_tri_cat}\ref{item:N_thick}, we know that $\N$ is a thick subcategory of $\C_{\N}$, so we associate classes of morphisms $\L_{\N}\deff\L$ and $\R_{\N}\deff\R$ as in Definition~\ref{def:Sn_from_thick} to $\N$:
\begin{enumerate}[label=\textup{(\roman*)}]
    \item $\L_\N= \Set{f\in\C^{\to} | f \text{ is an } \fs_{\N}\text{-inflation with } \cone(f)\in\N}$;  

    \item $\R_\N = \Set{g\in\C^{\to} | g \text{ is an } \fs_{\N}\text{-deflation with }\cocone(g)\in\N}$; and 
    
    \item $(\R_{\N})_{\ret} = \Set{ g\in\R_{\N} | g\textnormal{ is a retraction}}$.
    
\end{enumerate}
Put $\weq \deff \Sn$, the smallest subclass of morphisms closed under compositions containing both $\L_\N$ and $\R_\N$. Moreover, by Theorem~\ref{thm:localization_of_tri_cat}\ref{item:N_thick}, we have 
\begin{equation}\label{eqn:W-is-R-composed-L}
\weq = (\R_{\N})_{\ret}\circ\L_{\N}.
\end{equation}

Since $(\C,\seq^{R}_{\N}, \veq)$ is weak Waldhausen, we can consider the subcategory 
$(\C^{\weq},(\seq^{R}_{\N})^{\weq},\veq^{\weq})$ of $\weq$-acyclic objects (see Lemma~\ref{lem:localization_sequence_wWald}). In fact, in the proof of Theorem~\ref{thm:ex_seq_wWald_from_abel_loc} we will see that Setup~\ref{setup:Quillen_loc1} is satisfied, and hence we can consider the classes of morphisms given in Definition~\ref{def:mors_assoc_to_wWH_cat}:
\begin{enumerate}[label=\textup{(\roman*)}]
\setcounter{enumi}{3}

    \item\label{item:LRNac} $(\L^{R}_{\N})^{\ac} \deff \cof^{R}_{\N} \cap \weq$; 

    \item $(\R^{R}_{\N})^{\ac} \deff \fib^{R}_{\N} \cap \weq$; and 

    \item\label{item:TRN} $\T^{R}_{\N} \deff \set{ g\in (\R^{R}_{\N})^{\ac} | \text{there is an } \fs^{R}_{\N}\text{-conflation } A\overset{f}{\rightarrowtail}B\overset{g}{\twoheadrightarrow}C \text{ with } A\in\N}$. 
\end{enumerate}
We use the extra sub/superscripts above only to emphasise the origin of these classes of morphisms. 
\end{notation}

Comparing some of the classes of morphisms above, we observe the following.

\begin{lemma}
\label{lem:comparison_of_morphisms}
We have 
$(\R_{\N})_{\ret} \sse \T^{R}_{\N}$
and 
$\L_\N \sse (\L^R_\N)^{\ac}$. 
\end{lemma}

\begin{proof}
The first claim follows from the fact that $Q$ sends an $\fs_{\N}$-conflation to a short exact sequence and that $\Sn = \weq$ is saturated.
For the second claim, suppose $f\colon A\to B$ belongs to $\L_\N$. We only need to show that $f\in\weq$. 
There is a triangle 
$
\begin{tikzcd}[column sep=1cm,cramped]
    C[-1] \arrow{r}{h[-1]}
    & A \arrow{r}{f}
    & B \arrow{r}{g}
    & C
\end{tikzcd}
$
in $\C$,
where $h[-1]$ (and $h$) factor through $\N$ (see Proposition~\ref{prop:relative2}), 
and with $C\in\N$. 
Moreover, $C\in\N$ implies $g$ factors through $\N$ and hence $f\in \weq$ by Lemma~\ref{lem:char_of_Sn}, so $f\in(\L^R_\N)^{\ac}$ and we are done. 
\end{proof}

In the following main result of this subsection, 
the weak Waldhausen category $(\C,\seq^{R}_{\N},\weq)$
can be thought of as a ``one-sided'' localization of $\C^{R}_{\N} = (\C,\BE^{R}_{\N},\fs^{R}_{\N})$.

\begin{theorem}
\label{thm:ex_seq_wWald_from_abel_loc}
The triplet $(\C,\seq^{R}_{\N},\veq)$ is a wbiWET category and $(\C,\seq^{R}_{\N},\weq)$ is a wWET category both with respect to $\C^{R}_{\N}$, and $\N=\C^{\weq}$. 
Moreover, there is a localization sequence 
\begin{equation}\label{seq:localization_sequence_wWald2}
\begin{tikzcd}
    (\N,(\seq^{R}_{\N})^{\weq},\veq^{\weq})
        \arrow[hook]{r}{\inc}
    &(\C,\seq^{R}_{\N},\veq)
        \arrow[hook]{r}{\id_{\C}} 
    &(\C,\seq^{R}_{\N},\weq)
\end{tikzcd}
\end{equation}
of wWET categories, which induces the following exact sequence of abelian groups.
\begin{equation}
\label{eqn:k0inc-is-isomorphism}
\begin{tikzcd}[column sep=1.3cm]
0
	\arrow{r}
&K_{0}(\N,\BE|_{\N},\fs|_{\N}) 
    \arrow{r}{K_{0}(\inc)}
& K_{0}(\C^{R}_{\N})
    \arrow{r}{}
&0
\end{tikzcd}
\end{equation}
Corresponding statements also hold for $\C^{L}_{\N}$.
\end{theorem}

\begin{proof}
Since $\C^{R}_{\N}$ is an extriangulated category, the triplet $(\C,\seq^{R}_{\N},\veq)$ is just the canonical wbiWET category with respect to $\C^{R}_{\N}$ by Example~\ref{ex:extria_cat_is_wbiWET}.
Now we shall show that $(\C,\seq^{R}_{\N},\weq)$ is weak Waldhausen, from which it will follow that it is wWET with respect to $\C^{R}_{\N}$. 
Conditions \ref{WC0}--\ref{WC2} are satisfied as they already hold for $(\C,\seq^{R}_{\N},\veq)$, and 
\ref{WW0} follows from \ref{MS0} for $\Sn=\weq$. 
To show the Gluing Axiom \ref{WW1}, consider the diagram \eqref{eqn:gluing-axiom}, where $x,y,z$ all lie in $\weq=\Sn$. 
As in the proof of Lemma~\ref{lem:CSeqW-is-wbiWET}, we obtain a morphism \eqref{diag:ex_seq_wWald_from_NOS} of $\fs^{R}_{\N}$-triangles in $\C^{R}_{\N}$. 
Since $Q$ is right exact on $\fs^{R}_{\N}$-triangles by Theorem~\ref{thm:localization_of_tri_cat}\ref{item:abel_loc}, the morphism  
\[
    \begin{tikzcd}[column sep=1.5cm]
     QA_{1} 
        \arrow{r}{Q\begin{psmallmatrix}
         f_{1}\\-c_{1}
        \end{psmallmatrix}}
        \arrow{d}{Qx}[swap]{\cong}
    & QB_{1}\oplus QC_{1}
        \arrow{r}{Q\begin{psmallmatrix}
            b_{1}, \amph g_{1}
        \end{psmallmatrix}}
        \arrow{d}{Q(y \oplus z)}[swap]{\cong}
    & QD_{1}
        \arrow{d}{Qw} 
        \arrow{r}
    & 0
    \\
    QA_{2} 
        \arrow{r}{Q\begin{psmallmatrix}
         f_{2}\\-c_{2}
        \end{psmallmatrix}}
    & QB_{2}\oplus QC_{2}
        \arrow{r}{Q\begin{psmallmatrix}
            b_{2}, \amph g_{2}
        \end{psmallmatrix}}
    & QD_{2} \arrow{r}
    & 0
    \end{tikzcd}
\]
of right exact sequences in $\C/\N$ 
shows that $Qw$ is an isomorphism. 
Thus, $w\in\weq$ as $\weq$ is saturated.

Thus, we have established that \eqref{seq:localization_sequence_wWald2} is a sequence of skeletally small weak Waldhausen categories and, moreover, Setup~\ref{setup:Quillen_loc1} is met. 
We wish to use Theorem~\ref{thm:Quillen_localization1_new} to deduce  \eqref{seq:localization_sequence_wWald2} is a localization sequence in $\wWald$, so we must show that 
$\weq = \Sn$ consists of finite compositions of morphisms from 
$(\L^{R}_{\N})^{\ac} \cup \T^{R}_{\N} \cup \veq$ (see Notation~\ref{notation:section-6}\ref{item:LRNac},\ref{item:TRN}).  
We have 
$\weq = (\R_{\N})_{\ret}\circ\L_{\N}$ (see \eqref{eqn:W-is-R-composed-L}) 
and Lemma~\ref{lem:comparison_of_morphisms} tells us this is enough. 
The localization sequence \eqref{seq:localization_sequence_wWald2} induces the following right exact sequence in $\Ab$.
\begin{equation}
\label{eqn:prop6-14-sequence1}
\begin{tikzcd}[column sep=1.5cm]
K_{0}(\N,(\seq^{R}_{\N})^{\weq},\veq^{\weq})
    \arrow{r}{K_{0}(\inc)}
& K_{0}(\C,\seq^{R}_{\N},\veq)
    \arrow{r}{K_{0}(\id_{\C})}
& K_{0}(\C,\seq^{R}_{\N},\weq)
    \arrow{r}
&[-1cm] 0
\end{tikzcd}
\end{equation}

Recall that $(\C,\seq_{\N},\veq)$ denotes the wbiWET category corresponding to $\C_{\N}$. 
As $\N$ is thick in $\C_{\N}$  by Theorem~\ref{thm:localization_of_tri_cat}\ref{item:N_thick}, we can apply Lemma~\ref{lem:acyclics-is-N} to obtain $\N=\C^{\weq}$ and that $(\N,(\seq_{\N})^{\weq},\veq^{\weq})$
is the weak biWaldhausen category corresponding to the extriangulated category 
$(\N,\BE|_\N,\fs|_\N)$. 
It follows from Remark~\ref{rem:extriangulated_structures_on_N}\ref{item:inherited-structures-on-N-all-the-same} that 
$(\N,(\seq^{R}_{\N})^{\weq},\veq^{\weq}) = (\N,(\seq_{\N})^{\weq},\veq^{\weq})$
is the wbiWET category corresponding to the extriangulated category 
$(\N,\BE_\N^R|_\N,\fs_\N^R|_\N)  = (\N,\BE|_\N,\fs|_\N)$. 
In particular, we have $K_{0}(\N,(\seq^{R}_{\N})^{\weq},\veq^{\weq}) = K_{0}(\N,\BE|_{\N},\fs|_{\N})$ (see Example~\ref{ex:WET1}). 

We also have 
$K_{0}(\C,\seq^{R}_{\N},\veq) = K_{0}(\C^{R}_{\N})$ by Example~\ref{ex:WET1} again, and so \eqref{eqn:prop6-14-sequence1} is 
\begin{equation}
\label{seq:homotopy_fiber6}
\begin{tikzcd}[column sep=1.5cm]
K_{0}(\N,\BE|_{\N},\fs|_{\N}) 
    \arrow{r}{K_{0}(\inc)}
& K_{0}(\C^{R}_{\N})
    \arrow{r}{K_{0}(\id_{\C})}
& K_{0}(\C,\seq^{R}_{\N},\weq)
    \arrow{r}
&[-1cm] 0.
\end{tikzcd}
\end{equation}
To see that \eqref{eqn:k0inc-is-isomorphism} is exact, it is enough to show 
$K_{0}(\inc)$ is an isomorphism. 
The assumption $\cone(\N,\N) = \C$, which is equivalent to $\cocone(\N,\N)=\C$, implies that each object $C\in\C$ admits a triangle 
$\begin{tikzcd}[column sep=0.5cm, cramped]
C \arrow{r}{}& N_0 \arrow{r}{}& N_1 \arrow{r}{}& C[1]
\end{tikzcd}$
in $\C$ with $N_i\in\N$. 
By Proposition~\ref{prop:relative2}, 
this triangle yields the $\fs_{\N}^R$-conflation 
$\begin{tikzcd}[column sep=0.5cm, cramped]
C \arrow{r}{}& N_0 \arrow{r}{}& N_1,
\end{tikzcd}$
which is thus a finite $\N$-coresolution of $C\in\C^{R}_{\N}$ (see \cite[Def.~4.2]{OS23}). 
Second, we have that $\N$ is extension-closed in $\C^{R}_{\N}$, and it closed under taking cones of $\fs^R_\N$-inflations by Theorem~\ref{thm:localization_of_tri_cat}\ref{item:coresolving_subcat}. 
Hence, the dual of \cite[Thm~4.5]{OS23} shows that $K_{0}(\inc)$ is an isomorphism with inverse given by $[C] \mapsto [N_0] - [N_1]$. 
\end{proof}

\begin{remark}
Unlike the case of the extriangulated localization, such as Proposition~\ref{prop:ex_seq_wWald_from_NOS}, 
the wWET category $(\C,\seq^{R}_{\N},\weq)$ appearing in Theorem~\ref{thm:ex_seq_wWald_from_abel_loc} is not wcoWET with respect to $\C^{R}_{\N}$ in general. 
Therefore, \eqref{seq:localization_sequence_wWald2} is not necessarily a sequence of wbiWET categories.
\end{remark}

\subsection{The JS index isomorphism via localization}
\label{sec:JS_index}

Our aim in this last part of Section~\ref{sec:application-to-tri-cat} is to prove Corollary~\ref{cor:JS_index} via some form of reduction technique. Reduction techniques concerning $n$-cluster tilting subcategories have appeared several times before, e.g.\ \cite{IY08,AI12,Jas15,FMP23}. 
In the rest of this section, $\X$ denotes an $n$-cluster tilting subcategory of the triangulated category $\C$ (see Definition~\ref{def:n-CT-triangulated}). 
We also put $\N \deff \X^{\perp_{0}}$, so that we are in a special case of Setup~\ref{setup:abelian_local}.

For each integer $1\leq i\leq n-1$, we put $\N_{i}\deff\X[i]*\cdots*\X[n-1]$.
There is a cotorsion pair 
$(\X*\cdots *\X[i-1],\N_{i}[-1])$ on $\C$ in the sense of \cite[Def.~2.1]{Nak11} by \cite[Prop.~3.2, Thm.~3.4]{Bel15}, and so $\N_{i}$ is extension-closed and closed under direct summands in the triangulated category $(\C,\BE,\fs)$.
Note that $\N_{n-1}=\X[n-1]$, there are inclusions
$\N_{n-1} \sse \N_{n-2} \sse \cdots \sse\N_{1} \sse \C$, and 
$\N_{1} = \X^{\perp_{0}} = \N$. 
Recall that an object $P$ in an extriangulated category $(\D,\BF,\ft)$ is \emph{$\BF$-projective} if $\BF(P,D) = 0$ for all $D\in\D$; see \cite[Def.~3.23, Prop.~3.24]{NP19}.

\begin{lemma}\label{lem:reduction1}
Let $i\in\set{1,\ldots,n-1}$. 
The following assertions hold.
\begin{enumerate}[label=\textup{(\arabic*)}]

\item\label{item:Ni-inherits-extriangulated-structures} The subcategory $\N_{i}\sse\C$ inherits extriangulated structures $(\N_{i},\BE|_{\N_{i}},\fs|_{\N_{i}})$ from the triangulated category $(\C,\BE,\fs)$ and $(\N_{i},\BE^{R}_{\N}|_{\N_{i}},\fs^{R}_{\N}|_{\N_{i}})$ from the extriangulated category $\C^{R}_{\N}$. 

\item\label{item:ET-structures-on-Ni-coincide} We have 
$(\N_{i},\BE|_{\N_{i}},\fs|_{\N_{i}})=(\N_{i},\BE^{R}_{\N}|_{\N_{i}},\fs^{R}_{\N}|_{\N_{i}})$.

\item\label{item:projectives-in-Ni}
The subcategory of $\BE|_{\N_{i}}$-projectives in $(\N_{i},\BE|_{\N_{i}},\fs|_{\N_{i}})$ coincides with $\X[i]$ and $(\N_{i},\BE|_{\N_{i}},\fs|_{\N_{i}})$ has enough $\BE|_{\N_{i}}$-projectives. 
In particular, $\X[i]$ is contravariantly finite, rigid and closed under direct summands in $\N_{i}$.
\end{enumerate}
\end{lemma}

\begin{proof}
\ref{item:Ni-inherits-extriangulated-structures} follows immediately from $\N_{i}$ being extension-closed in $(\C,\BE,\fs)$ and hence also in $\C^{R}_{\N}$.

\ref{item:ET-structures-on-Ni-coincide} is a consequence of  
$(\N,\BE|_{\N},\fs|_{\N}) 
= (\N,\BE^{R}_{\N}|_{\N},\fs^{R}_{\N}|_{\N})$ 
(see Remark~\ref{rem:extriangulated_structures_on_N}\ref{item:inherited-structures-on-N-all-the-same}) 
and $\N_{i}\sse\N$.

For \ref{item:projectives-in-Ni}, we see that $\X[i]$ consists of $\BE|_{\N_{i}}$-projective objects as $(\X*\cdots *\X[i-1],\N_{i}[-1])$ is a cotorsion pair, hence 
$\BE|_{\N_{i}}(\X[i],\N_{i}) = \BE(\X[i],\N_{i}) = 0$. 
If $i=n-1$, then there is nothing to show, so suppose $i\leq n-2$. 
Since $\N_{i}=\X[i]*(\X[i+1]*\cdots *\X[n-1])$, we have 
$\N_{i}= \cone(\X[i]*\cdots *\X[n-2],\X[i])$, which shows that $(\N_{i},\BE|_{\N_{i}},\fs|_{\N_{i}})$ has enough $\BE|_{\N_{i}}$-projectives. 
If $C$ is $\BE|_{\N_{i}}$-projective, then this also shows that it must be a direct summand of an object in $\X[i]$ and so $C\in\X[i]$. Thus, $\X[i]$ is precisely the collection of $\BE|_{\N_{i}}$-projectives. 
The remaining claims follow from this.
\end{proof}

To simplify notation, we put $K_{0}(\N_{i}) \deff K_{0}(\N_{i},\BE|_{\N_{i}},\fs|_{\N_{i}})$. 
Note that the extriangulated structure 
$(\N_{n-1},\BE|_{\N_{n-1}},\fs|_{\N_{n-1}})$
is just the split exact structure on $\X[n-1]$ since $\BE(\X[n-2],\X[n-2])\cong \BE(\X,\X) = 0$, and so $K_{0}(\N_{n-1}) = K^{\sp}_{0}(\X[n-1])$.

For any additive category $\A$, we denote the category of finitely presented functors from $\A^{\op}$ to $\Ab$ by $\mod\A$. Recall that a functor $\mathsf{M}\colon \A^{\op} \to \Ab$ is \emph{finitely presented} if there is an exact sequence
$
\begin{tikzcd}[cramped, column sep=0.5cm]
\A(-,A)
    \arrow{r}
& \A(-,B)
    \arrow{r}
& M 
    \arrow{r}
& 0
\end{tikzcd}
$
for some objects $A,B\in\A$. 
For each $0 \leq i \leq n-1$, consider the restricted Yoneda embedding
$\BY|_{\X[i]} \colon \N_{i}\to\mod(\X[i])$ that sends an object $C\in \N_{i}$ to the contravariant additive functor 
$\BY|_{\X[i]}(C) \deff \N_{i}(-,C)|_{\X[i]}$. 
Note that 
$\BY|_{\X[i]}$ is essentially surjective and 
$\Ker (\BY|_{\X[i]}) = \N_{i+1}$ (viewed as subcategories of $\N_{i}$) for $1\leq i \leq n-2$. 
Furthermore, since the functor $C\mapsto \C(-,C)$ is a cohomological functor $\BY\colon \C \to \mod\C$, we know that $\BY|_{\X[i]}$ sends 
an $\fs|_{\N_{i}}$-conflation $A\to B\to C$ to an exact sequence
$\BY|_{\X[i]}(A)\to \BY|_{\X[i]}(B)\to \BY|_{\X[i]}(C)$ in $\mod(\X[i])$. 
The next result says this is actually a right exact sequence.

\begin{lemma}\label{lem:reduction_right_exact}
For $1 \leq i \leq n-1$, the functor
$
\BY|_{\X[i]} \colon (\N_{i},\BE|_{\N_{i}},\fs|_{\N_{i}})\to\mod(\X[i])
$
is right exact in the sense of Definition~\ref{def:left_right_exact_functors}.
\end{lemma}
\begin{proof}
Since $\X[i]$ is the subcategory of $\BE|_{\N_{i}}$-projectives in $(\N_{i},\BE|_{\N_{i}},\fs|_{\N_{i}})$ by Lemma~\ref{lem:reduction1}\ref{item:projectives-in-Ni}, we have 
$
\BY|_{\X[i]}(A[1]) = 
\N_{i}(\X[i],A[1])
    = \BE|_{\N_{i}}(\X[i],A) 
    = 0 
$ 
for any $A\in \N_{i}$.
The result follows.
\end{proof}

Although the functor $\BY|_{\X}\colon \C \to \mod\X$ fails to be right exact in general, its restriction to $\C^{R}_{\N}$ is so. 
Indeed, by \cite[Exam.~3.20(2)]{OS23} we know that there is an exact equivalence 
$G \colon \C/\N \overset{\sim}{\lra} \mod\X$ of abelian categories with $\BY|_{\X} = G Q$. Thus, as $Q \colon \C^{R}_{\N} \to \C/\N$ is right exact (see Theorem~\ref{thm:localization_of_tri_cat}\ref{item:abel_loc}), we have that 
$\BY|_{\X}\colon\C^{R}_{\N}\to \mod\X$ is right exact.

In particular, we have the diagram 
\begin{equation}
\label{diag:summary_reduction}
\begin{tikzcd}[row sep=0.6cm]
    \N_{n-1} \arrow[hook]{r}& \N_{n-2} \arrow[hook]{r}\arrow[two heads]{d}{\BY|_{\X[n-2]}}& \cdots \arrow[hook]{r}& \N_1 \arrow[hook]{r}\arrow[two heads]{d}{\BY|_{\X[1]}}& \C^{R}_{\N}\arrow[two heads]{d}{\BY|_{\X}} \\
    {}& \mod(\X[n-2]) & {} & \mod(\X[1]) & \mod\X
\end{tikzcd}
\end{equation}
in which the horizontal arrows are the canonical inclusions and the 
vertical arrows are the right exact restricted Yoneda functors.
To establish the JS index isomorphism \eqref{eqn:JS_index}, we will show there is a chain of isomorphisms as follows, induced by the top row of \eqref{diag:summary_reduction}.
\begin{equation}
\label{diag:summary_reduction_Grothendieck_groups}
\begin{tikzcd}
K^{\sp}_{0}(\X[n-1]) = K_{0}(\N_{n-1}) \arrow{r}{\cong}& K_{0}(\N_{n-2}) \arrow{r}{\cong}& \cdots \arrow{r}{\cong}& K_{0}(\N_1) \arrow{r}{\cong}& K_{0}(\C^{R}_{\N})
\end{tikzcd}
\end{equation} 
Indeed, the JS index isomorphism then follows from noting that the equivalence \mbox{$\X\overset{\sim}{\lra}\X[n-1]$} yields an isomorphism 
$\zeta\colon K^{\sp}_{0}(\X) \overset{\cong}{\lra} K^{\sp}_{0}(\X[n-1])$ given by 
\begin{equation}\label{eqn:iso-X-Xn-1}
    [X] \overset{\zeta}{\longmapsto} (-1)^{n-1}[X[n-1]].
\end{equation}
The sign is chosen because of the equality 
$[X] = (-1)^{n-1}[X[n-1]]$ in $K_{0}(\C) = K_{0}(\C,\BE,\fs)$, the Grothendieck group of the triangulated category $\C$.

As a first step in the proof of Corollary~\ref{cor:JS_index}, we already have the rightmost isomorphism in \eqref{diag:summary_reduction_Grothendieck_groups} by Theorem~\ref{thm:ex_seq_wWald_from_abel_loc}. 
The remaining steps of our reduction process from $\N_{i}$ to $\N_{i+1}$ will be shown by mimicking the proof of Theorem~\ref{thm:ex_seq_wWald_from_abel_loc}. 
To this end, we will impose natural wWET structures on $\N_{i}$ for any $1\leq i\leq n-1$. We put (cf.\ Notation~\ref{notation:section-6}):
\begin{enumerate}[label=\textup{(\roman*)}]
    \item $\seq^{R}_{\N}|_{\N_{i}}$ to be the collection of all $\fs^{R}_{\N}|_{\N_{i}}$-conflations;
    \item $\veq_i\deff \Iso\N_{i}$; and
    \item $\weq_i\deff\Set{f\in\N^{\to}_i|\BY|_{\X[i]}(f)\text{ is an isomorphism}}$.
\end{enumerate}
It follows that $(\N_{i},\seq^{R}_{\N}|_{\N_{i}},\veq_i)$ is the canonical wbiWET category associated to $(\N_{i},\BE|_{\N_{i}},\fs|_{\N_{i}})$. 
We remark that making the analogous declarations for $i=0$ would yield the wWET categories $(\C,\seq^{R}_{\N},\veq)$ and $(\C,\seq^{R}_{\N},\weq)$, using the equivalence $G \colon \C/\N \to \mod\X$ to see $\weq_{0} = \weq$.

\begin{lemma}\label{lem:Ni-with-Wi-is-wWET}
The triplet 
$(\N_{i},\seq^{R}_{\N}|_{\N_{i}},\weq_i)$ 
is a wWET category with respect to $(\N_{i},\BE|_{\N_{i}},\fs|_{\N_{i}})$. 
\end{lemma}

\begin{proof}
Since $(\N_{i},\seq^{R}_{\N}|_{\N_{i}},\veq_i)$ is a weak Waldhausen category, 
we need only show \ref{WW0} and \ref{WW1}.
Since \ref{WW0} is clear, 
suppose we are in the setup \eqref{eqn:gluing-axiom} of the Gluing Axiom \ref{WW1}. 
Then we obtain a morphism 
\eqref{diag:ex_seq_wWald_from_NOS}
of $\fs|_{\N_{i}}$-conflations. 
We apply the right exact functor 
$\BY|_{\X[i]}$ 
to \eqref{diag:ex_seq_wWald_from_NOS}.
Since $\BY|_{\X[i]}(x)$ and 
$\BY|_{\X[i]}(y\oplus z)$
are isomorphisms in $\mod(\X[i])$, 
we have that $\BY|_{\X[i]}(w)$ must also be an isomorphism 
using the right exactness of $\BY|_{\X[i]}$ (see Lemma~\ref{lem:reduction_right_exact}), so $w\in\weq_i$. 
\end{proof}

The last preparatory lemma we need is the following.

\begin{lemma}\label{lem:right_ex_loc_via_Waldhausen_ver2}
For each $1\leq i\leq n-2$, there is a sequence of wWET categories as follows
\begin{equation}
\label{eqn:lem_loc_seq_reduction}
\begin{tikzcd}
    (\N_{i+1},\seq^{R}_{\N}|_{\N_{i+1}},\veq_{i+1})
        \arrow[hook]{r}{\inc_{i}} 
    & (\N_{i},\seq^{R}_{\N}|_{\N_{i}},\veq_i)
        \arrow{r}{\id_{\N_{i}}} 
    & (\N_{i},\seq^{R}_{\N}|_{\N_{i}},\weq_i),
\end{tikzcd}
\end{equation}
where 
$(\N_{i+1},\seq^{R}_{\N}|_{\N_{i+1}},\veq_{i+1})$ is the subcategory of $\weq_{i}$-acyclic objects in $(\N_{i},\seq^{R}_{\N}|_{\N_{i}},\veq_i)$. 
Moreover, the induced homomorphism 
$K_{0}(\inc_{i})
    \colon 
    K_{0}(\N_{i+1})
        \overset{\cong}{\lra} 
    K_{0}(\N_{i})$ 
is an isomorphism. 
\end{lemma}

\begin{proof}
We know all the triplets in \eqref{eqn:lem_loc_seq_reduction} 
are wWET by Lemma~\ref{lem:Ni-with-Wi-is-wWET}. 
The equality $\N_{i}^{\weq_{i}} = \N_{i+1}$ follows from $\Ker (\BY|_{\X[i]}) = \N_{i+1}$ and the definition of $\weq_{i}$. 

The last assertion will follow from the dual of \cite[Thm.~4.5]{OS23} once we note the following.
First, any object in $\N_{i}$ has a  finite $\N_{i+1}$-coresolution because $\N_{i}=\cocone(\N_{i+1},\X[i+1])$ and $\X[i+1]\sse\N_{i+1}$. 
Second, suppose $A\to B\to C$ is an $\fs|_{\N_{i}}$-conflation with $A,B\in\N_{i+1}$. Using Lemma~\ref{lem:reduction_right_exact} and $\BY|_{\X[i]}(B)=0$, we have $\BY|_{\X[i]}(C)=0$ and so $C\in\N_{i+1}$. 
\end{proof}

Due to Lemma~\ref{lem:right_ex_loc_via_Waldhausen_ver2}, we have confirmed that  
\eqref{diag:summary_reduction} induces the following diagram of wWET categories.

\begin{equation}
\label{diag:summary_reduction2}
\begin{tikzcd}[column sep=0.7cm, row sep=0.5cm]
    (\N_{n-1},\seq^{R}_{\N}|_{\N_{n-1}},\veq_{n-1}) \arrow[hook]{r}&
    \cdots \arrow[hook]{r}&
    (\N_{1},\seq^{R}_{\N}|_{\N_{1}},\veq_{1}) \arrow[hook]{r}\arrow[two heads]{d}&
    (\N_{0},\seq^{R}_{\N}|_{\N_{0}},\veq_{0})\arrow[two heads]{d} \\
    {}& {}& (\N_{1},\seq^{R}_{\N}|_{\N_{1}},\weq_{1})& (\N_{0},\seq^{R}_{\N}|_{\N_{0}},\weq_{0})
\end{tikzcd}
\end{equation}
In \eqref{diag:summary_reduction} 
each segment of shape $\begin{tikzcd}[column sep=0.4cm, row sep=0.4cm, scale cd=0.4]
    \bullet\arrow[hook]{r}&\bullet\arrow[two heads]{d}\\
    {}&\bullet
\end{tikzcd}$ 
does not necessarily belong to $\ET$, 
whereas the entirety of \eqref{diag:summary_reduction2} sits in $\wWald$.
We can now prove the JS index isomorphism \eqref{eqn:JS_index}.

\begin{proof}[Proof of Corollary~\ref{cor:JS_index}]
Since we have obtained the isomorphisms in \eqref{diag:summary_reduction_Grothendieck_groups}, we have established that the inclusion $\X\into \C^{R}_{\N}$ induces an isomorphism
$K_{0}^{\sp}(\X) \overset{\cong}{\lra} K_{0}(\C^{R}_{\N})$. We need only show that its inverse is as described in Corollary~\ref{cor:JS_index}.

Let $\N_{0}\deff \C^{R}_{\N}$, so that we have an isomorphism 
$K_{0}(\inc_{i})\colon K_{0}(\N_{i+1})\overset{\cong}{\lra}K_{0}(\N_{i})$ induced by the inclusion 
$\inc_{i}\colon \N_{i+1} \into \N_{i}$ 
for each $i\in\set{0,\ldots,n-2}$; see Theorem~\ref{thm:ex_seq_wWald_from_abel_loc} and Lemma~\ref{lem:right_ex_loc_via_Waldhausen_ver2}.
Let $\eta_{i}\colon K_{0}(\N_{i}) \overset{\cong}{\lra} K_{0}(\N_{i+1})$ denote the inverse of $K_{0}(\inc_{i})$. 
Each isomorphism $\eta_{i}$ is given by coresolving a given object via the equality $\N_{i} = \cocone(\N_{i+1},\X[i+1])$, namely, for any generator $[N_{i}]\in K_{0}(\N_{i})$, we have
$
\eta_{i}([N_{i}]) = [N_{i+1}]-[X_{i}[i+1]]
$
where 
$N_{i}\to N_{i+1}\to  X_{i}[i+1]$ 
is an $\fs^{R}_{\N}$-conflation
in $\N_{i}$ with $N_{i+1}\in\N_{i+1}$ and $X_{i}\in\X$. 
Note that if $X\in\X$, then for all $0\leq i \leq n-2$ we have an $\fs^{R}_{\N}$-triangle of the form
$\begin{tikzcd}[column sep=1.3cm, cramped]
X[i]
    \arrow{r}
&[-0.8cm] 0 
    \arrow{r}
&[-0.8cm] X[i+1]
    \arrow[dashed]{r}{\id_{X[i+1]}}
& {}
\end{tikzcd}$
in $\N_{i}$, and hence 
$\eta_{i}[X[i]]=-[X[i+1]]$. 
Thus, for an arbitrary object $N_{0}\deff C \in\C^{R}_{\N}$, 
we have:
\begin{align*}
    [C] = [N_0] &\overset{\eta_{0}}{\longmapsto} [N_1]-[X_0[1]]\\
    &\overset{\eta_{1}}{\longmapsto} [N_2]-[X_1[2]]+[X_0[2]]\\
    &\overset{\eta_{2}}{\longmapsto} [N_3]-[X_2[3]]+[X_1[3]]-[X_0[3]]\\
    &{} \hspace{10pt}\vdots \\
    &\overset{\eta_{n-2}}{\longmapsto} \sum_{i=0}^{n-1}(-1)^{n-1-i}[X_i[n-1]]\\
    &\overset{\zeta}{\longmapsto}  \sum_{i=0}^{n-1}(-1)^{i}[X_i] = \ind_{\X}(C) & (\text{see } \eqref{eqn:iso-X-Xn-1})
\end{align*}
This is indeed the same as the JS index isomorphism \eqref{eqn:JS_index}, because the above coresolving procedure yields the diagram 
\begin{equation*}\label{eqn:resolution_from_reduction_process}
\begin{tikzcd}[column sep=0.2cm]
X_{n-1} 
    \arrow{rr}
    \arrow[equals]{dr}
&[-0.25cm]&[-0.25cm] X_{n-2}
    \arrow{rr}
    \arrow{dr} 
&[-0.25cm]&[-0.25cm] X_{n-3}
    \arrow{rr}
&&\cdots
    \arrow{rr}
&& X_1 
    \arrow{rr}
    \arrow{dr}
&& X_0 
    \arrow{dr}
    \arrow{rr}
&& C\\
& N_{n-1}[-(n-1)] 
    \arrow{ur}
&& N_{n-2}[-(n-2)]
    \arrow{ur}
&&&&
&& N_1[-1] 
    \arrow{ur}
&& N_0
    \arrow[equals]{ur}
\end{tikzcd}
\end{equation*}
and hence a finite $\X$-resolution of $C$. This is an $\X$-resolution in $\C^{R}_{\N}$ 
as each triangle
\[
\begin{tikzcd}
N_{i+1}[-(i+1)] 
	\arrow{r}
&X_{i}
	\arrow{r}
&N_{i}[-i] 
	\arrow{r}
& N_{i+1}[-i]
\end{tikzcd}
\]
satisfies $N_{i+1}[-i]\in\X[1]* \cdots * \X[n-1-i]\sse\X^{\perp_{0}} = \N$.
\end{proof}

\section{Sarazola's \texorpdfstring{$K$}{K}-theory localization}
\label{sec:Sarazola_K-theory}

In \cite{Sar20}, Sarazola showed that there exists a natural Waldhausen structure on an exact category arising from a single cotorsion pair as an attempt to generalize Hovey's abelian model structure corresponding to a certain twin cotorsion pair, see \cite[Thm.~2.2]{Hov02} and \cite[Sec.~4]{NP19}.
A localization theorem of $K$-theory is also established there (see \cite[Thm.~3]{Sar20}). 
In this section we will investigate Sarazola's construction in the extriangulated setting.

\subsection{Weak Waldhausen categories following Sarazola}
\label{sec:Sarazola_Wald_cat}
Motivated by \cite[Sec.~3]{Yu22}, we will work under Setup~\ref{setup:Sarazola1} in this subsection. 
An \emph{$\BE$-injective} object in an extriangulated category $(\C,\BE,\fs)$ is defined dually to an $\BE$-projective. Moreover, $(\C,\BE,\fs)$ \emph{has enough $\BE$-injectives} if each object $A\in\C$ admits an $\fs$-conflation
$A\lra I_{A}\lra \Sigma A$ with $I_{A}$ an $\BE$-injective. 
In this case, $\Sigma A$ is called a \emph{cosyzygy} of $A$.

\begin{setup}\label{setup:Sarazola1}
Let $(\C,\BE,\fs)$ be a skeletally small,  weakly idempotent complete, extriangulated category with enough $\BE$-injectives.
We denote by $\I \deff \BE\Inj(\C) \sse \C$ the full subcategory of $\BE$-injective objects.
In addition, we fix a full subcategory $\N\sse\C$ containing $\I$ that is extension-closed, and is closed under direct summands and cones of $\fs$-inflations. 
\end{setup}

Although $\N$ is not a thick subcategory of $(\C,\BE,\fs)$ in general, following Definition~\ref{def:Sn_from_thick} we consider classes of morphisms: 
\begin{itemize}
\item
$\K_{\N} \deff \Set{f\in\C^{\to} | f \ \textnormal{is an $\fs$-inflation with}\ \cone(f)\in\N}$; 
\item
$\Q_{\I}\deff\Set{ g\in\C^{\to} | g\textnormal{\ is an $\fs$-deflation with } \cocone(g)\in\I }$; and 
\item
$\weq\deff\Q_{\I}\circ\K_{\N}$.
\end{itemize}
Note that $\Q_{\I}$ is the class of retractions in $\C$ with the kernels sitting in $\I$, thanks to the weakly idempotent completeness of $\C$ (see \cite[Prop.~2.5]{BHST2022}) and the $\BE$-injectivity of objects in $\I$.

In the sequel, we will construct the following localization sequence of weak Waldhausen categories, which is an analogue of Sarazola's construction \cite[Thm.~6.1]{Sar20}.
Indeed, if we assume $\C$ is an exact category closed under kernels of epimorphisms and $\N$ is thick in addition to Setup~\ref{setup:Sarazola1}, then we have the setup of \cite[Thm.~6.1]{Sar20}. 
Denote by $(\C,\seq,\veq)$ the canonical wbiWET category corresponding to the extriangulated category $(\C,\BE,\fs)$ (see Example~\ref{ex:extria_cat_is_wbiWET}).

\begin{proposition}
\label{prop:Sarazola_homotopy_fiber}
There is a sequence 
\begin{equation}
\label{seq:Sarazola_homotopy_fiber}
\begin{tikzcd}
(\C^{\weq},\seq^{\weq},\veq^{\weq}) 
    \arrow[hook]{r}{\mathsf{inc}}
&(\C,\seq,\veq) 
    \ar{r}{\id_{\C}} 
&(\C,\seq,\weq)
\end{tikzcd}
\end{equation}
in $\wWald$, 
where $\C^{\weq}=\N$ and which induces the following exact sequence in $\Ab$. 
\begin{equation}\label{eqn:7.2-K0-groups}
\begin{tikzcd}[column sep=1.3cm]
K_{0}(\N,\BE|_{\N},\fs|_{\N}) 
    \arrow{r}{K_{0}(\mathsf{inc})}
&K_{0}(\C,\BE,\fs) 
    \arrow{r}{K_{0}(\id_{\C})}
&K_{0}(\C,\seq,\weq)
    \arrow{r}
&0
\end{tikzcd}
\end{equation}
\end{proposition}

Before proving the above, we show that an example of Setup~\ref{setup:Sarazola1} arises via hereditary cotorsion pairs of extriangulated categories (see \cite{Gil06} for the exact case).

\begin{example}
For this example, we assume that $(\C,\BE,\fs)$ also has enough $\BE$-projectives, which is defined dually to having enough $\BE$-injectives. 
A \emph{syzygy} $\Omega A$ of $A\in\C$ is defined dually to a cosyzygy. 
Then one can define the \emph{higher extensions} $\BE^i(-,-)$ of the bifunctor $\BE(-,-)$, namely, one puts $\BE^i(C,A)\deff \BE(C,\Sigma^{i+1}A)\cong\BE(\Omega^{i+1}C,A)$ for $i\geq 1$ and $A,C\in\C$ (see \cite[Sec.~5]{LN19}).

Let $(\U,\V)$ be a cotorsion pair in $(\C,\BE,\fs)$ (see \cite[Def.~4.1]{NP19}), and assume it is 
\emph{hereditary} (see \cite[Def.~2.7]{Yu22}, also \cite[Def.~2.7]{LZ23}), i.e.\ $\BE^{i}(\U,\V)=0$ for all $i\geq 1$. 
Note that $\V = \U^{\perp_{1}} = \Set{ A\in\C | \BE(\U,A) = 0 }$, and so $\V$ is extension-closed, closed under direct summands and contains $\I$.
It follows that $\V$ is closed under cones of $\fs$-inflations. Indeed, like the exact case, it is known from \cite[Prop.~5.2]{LN19} that, 
for any $\fs$-conflation $A\to B\to C,$
the sequence 
\[
\BE(U,A)\to \BE(U,B)\to \BE(U,C)\to \BE^2(U,A)\to \BE^2(U,B)\to \BE^2(U,C)\to \cdots
\]
is exact for any $U\in\U$. 
Thus if $A,B\in\V$, then $C\in\V$.
In particular, $\V$ fulfills all properties that we demand of $\N$ in Setup~\ref{setup:Sarazola1}.

Since $\V$ is extension-closed, it inherits an extriangulated structure from $(\C,\BE,\fs)$. 
Thus, applying Proposition~\ref{prop:Sarazola_homotopy_fiber} to $\N = \V$ measures the difference of these Grothendieck groups of $\C$ and $\V$ in some sense, namely, there is a right exact sequence $K_{0}(\V)\to K_{0}(\C)\to K_{0}(\C,\seq,\weq)\to 0$.
\end{example}

To show Proposition~\ref{prop:Sarazola_homotopy_fiber}, we first prepare a more detailed description of the factorization $\weq$.

\begin{lemma}
\label{lem:factorization_via_injective_hull}
Let $w\colon A\to B$ be a weak equivalence in $\weq=\Q_{\I}\circ \K_{\N}$. 
Then there is a factorization 
$
\begin{tikzcd}[column sep=1.2cm, cramped]
A
    \arrow{r}{\begin{psmallmatrix}
 \gamma_A \\
 w
\end{psmallmatrix}}
& I_A \oplus B
    \arrow{r}{\begin{psmallmatrix}0, \amph \id_B
 \end{psmallmatrix}}
& B
\end{tikzcd}
$
of $w$, where
$A\overset{\gamma_A}{\lra}I_A$ is an $\fs$-inflation  with $I_A\in\I$, 
${\begin{psmallmatrix}\gamma_A \\
w
\end{psmallmatrix}}\in\K_{\N}$ 
and
${\begin{psmallmatrix}0, & \id_B
\end{psmallmatrix}}\in\Q_{\I}$. 
\end{lemma}

\begin{proof}
By the definition of $\weq$, the weak equivalence $w$ admits a factorization 
\[
\begin{tikzcd}[row sep=0.6cm]
A \arrow{dr}{l}\arrow{rr}{w}&& B \\
&C\arrow{ur}{r}&
\end{tikzcd}
\]
with $l\in\K_{\N}$ and $r\in\Q_{\I}$. 
It follows that we may assume 
$C = I \oplus B$ for some $I\in\I$, 
$r = \begin{psmallmatrix}0, & \id_B
 \end{psmallmatrix}$, 
and 
$
l = \begin{psmallmatrix}
\gamma \\
 w
 \end{psmallmatrix}
\colon A \to I\oplus B$.

There is an $\fs$-conflation 
$A\overset{\gamma'}{\lra} I'\overset{}{\lra} \Sigma A$ 
with $I'\in\I$ as $(\C,\BE,\fs)$ has enough $\BE$-injectives. 
By Lemma~\ref{lem:wPO_wPB}, there is a morphism 
\begin{equation}
\label{eqn:wPO-gamma-prime-s}
\begin{tikzcd}
    A\arrow{r}{\begin{psmallmatrix}\gamma \\
 w
 \end{psmallmatrix}}\arrow{d}[swap]{\gamma'}\wPO{dr}
 &I\oplus B\arrow{r}{}\arrow{d}{}
 &N\arrow[dashed]{r}{}\arrow[equals]{d}{}&{}\\
    I'
    	\arrow{r}{}&N'\arrow{r}{}&N\arrow[dashed]{r}{}&{}
\end{tikzcd}
\end{equation}
of $\fs$-triangles, where $N\in\N$, 
as well as an $\fs$-conflation
\begin{equation}
\label{eqn:conflation-for-gamma-gamma-prime-and-s}
\begin{tikzcd}
A
	\arrow{r}{\begin{psmallmatrix}
 \gamma'\\
 \gamma \\
 w
\end{psmallmatrix}}
& I'\oplus I\oplus B
	\arrow{r}
& N'. 
\end{tikzcd}
\end{equation}
In addition, the bottom row of \eqref{eqn:wPO-gamma-prime-s} splits as $I'\in\I$, and so $N' \cong I'\oplus N \in\N$. 

Lastly, by \cite[Cor.~3.16]{NP19}, we have that 
$
\begin{psmallmatrix}
\gamma' \\
\gamma
\end{psmallmatrix}
\colon A \to I' \oplus I
$
is an $\fs$-inflation. So we are done by setting 
$I_A \deff I' \oplus I$ and 
$\gamma_A
\deff \begin{psmallmatrix}
\gamma' \\
\gamma
\end{psmallmatrix}$.
\end{proof}

\begin{corollary}
\label{cor:inflation-to-I-and-weq-is-again-weq}
If $w\colon A\to B$ is in $\weq$ and $\alpha\colon A \to J$ is an $\fs$-inflation with $J\in\I$, then $
\begin{psmallmatrix}
    \alpha \\
    w
\end{psmallmatrix}
\in \K_{\N}$. 
\end{corollary}

\begin{proof} 
Suppose we have the factorization of $w$ as given in the statement of Lemma~\ref{lem:factorization_via_injective_hull}, 
and by assumption we can consider an $\fs$-conflation  $A\overset{\alpha}{\lra} J\overset{}{\lra} \Sigma A.$
We can construct $\fs$-conflations witnessing that 
$
\begin{psmallmatrix}
 \gamma_{A}\\
 \alpha \\
 w
\end{psmallmatrix}
$
is an $\fs$-inflation in the following two ways. 
First, start with an $\fs$-conflation 
$
\begin{tikzcd}[cramped]
A
	\arrow{r}{\begin{psmallmatrix}
 \gamma_{A}\\
 w
\end{psmallmatrix}}
& I_{A}\oplus B
	\arrow{r}
& N'
\end{tikzcd}
$
with $N'\in\N$
(see \eqref{eqn:conflation-for-gamma-gamma-prime-and-s}), and 
apply Lemma~\ref{lem:wPO_wPB} to obtain
an $\fs$-conflation
$
\begin{tikzcd}[cramped]
A
	\arrow{r}{\begin{psmallmatrix}
 \gamma_{A}\\
 \alpha \\
 w
\end{psmallmatrix}}
& I_{A}\oplus J\oplus B
    \arrow{r}
& J \oplus N'.
\end{tikzcd}
$
Second, start with $A\overset{\alpha}{\lra} J\overset{}{\lra} \Sigma A$
and apply Lemma~\ref{lem:wPO_wPB} twice to obtain 
diagrams
\begin{equation}
\begin{tikzcd}
A
    \arrow{r}{\alpha}
 \arrow{d}[swap]{w}
 \wPO{dr}
&J
    \arrow{r}{}
    \arrow{d}{}
&\Sigma A
    \arrow[dashed]{r}{}
    \arrow[equals]{d}{}&{}\\
B
    \arrow{r}{}
&M
    \arrow{r}{}
&\Sigma A
    \arrow[dashed]{r}{}&{}
\end{tikzcd}
\hspace{0.5cm}
\text{and}
\hspace{0.5cm}
\begin{tikzcd}
A
    \arrow{r}{\begin{psmallmatrix}\alpha \\
 w
 \end{psmallmatrix}}
    \arrow{d}[swap]{\gamma_{A}}
    \wPO{dr}
&J\oplus B
    \arrow{r}{}
    \arrow{d}{}
&M
    \arrow[dashed]{r}{}
    \arrow[equals]{d}{}
&{}\\
I_{A}
    \arrow{r}{}
&I_{A}\oplus M
    \arrow{r}{}
&M
    \arrow[dashed]{r}{0}
&{,}
\end{tikzcd}
\end{equation}
the second of which gives rise to the $\fs$-conflation
$
\begin{tikzcd}[cramped]
A
    \arrow{r}{\begin{psmallmatrix}
 \gamma_{A}\\
 \alpha \\
 w
\end{psmallmatrix}}
& I_{A}\oplus J\oplus B
    \arrow{r}
& I_{A}\oplus M. 
\end{tikzcd}
$
It follows that $I_{A}\oplus M \cong J\oplus N' \in\N$, so $\cone \begin{psmallmatrix}
\alpha \\
w
\end{psmallmatrix} = M\in\N$ as $\N$ is closed under direct summands.
\end{proof}

Lemma~\ref{lem:factorization_via_injective_hull} also yields a characterization of the $\fs$-inflations in $\weq$.

\begin{corollary}
\label{cor:char_Ln}
Let $A\overset{f}{\lra} B\overset{g}{\lra} C\overset{\delta}{\dra}$ be any $\fs$-triangle.
Then $f\in\K_{\N}$ if and only if $f\in\weq$.
\end{corollary}
\begin{proof}
The forwards direction is clear. Thus suppose $f\in\weq$. 
By Lemma~\ref{lem:factorization_via_injective_hull},
there exists an $\fs$-inflation $\gamma_A\colon A\to I_A$ with $I_A\in\I$ such that the morphism $\begin{psmallmatrix}
\gamma_A \\
f
\end{psmallmatrix}\colon A\to I_A\oplus B$ belongs to $\K_{\N}$.
By Lemma~\ref{lem:wPO_wPB}, there is a morphism 
\[
\begin{tikzcd}[row sep=0.6cm]
    A\arrow{r}{f}\arrow{d}[swap]{\gamma_A}\wPO{dr}
 &B\arrow{r}{g}\arrow{d}{}
 &C\arrow[dashed]{r}{}\arrow[equals]{d}{}&{}\\
    I_{A}
    	\arrow{r}{}&N\arrow{r}{}&C\arrow[dashed]{r}{}&{}
\end{tikzcd}
\]
of $\fs$-triangles, in which the bottom row splits as $I_{A}\in\I$. 
In particular, $C$ is a direct summand of $N = \cone\begin{psmallmatrix}
\gamma_A \\
f
\end{psmallmatrix} \in \N$, so $C\in\N$ and we are done.
\end{proof}

The next lemma is just an easy application of the Horseshoe Lemma in the exact category case.
However, our category $\C$ is extriangulated and we need a careful discussion.

\begin{lemma}
\label{lem:horseshoe}
(cf. \cite[Lem.~3.4]{Yu22})
Let $A_1\overset{f}{\lra} A_2\overset{g}{\lra} A_3\overset{\delta}{\dra}$ be an $\fs$-triangle.
Then there exist $\fs$-triangles $A_i\overset{f_i}\lra I_i\overset{g_i}{\lra} \Sigma A_i\overset{\xi_i}{\dra}$ with $I_i\in\I$, and also a commutative diagram
\begin{equation}
\label{diag:horseshoe1}
\begin{tikzcd}[column sep=1.2cm, ampersand replacement=\&, row sep=0.6cm]
A_1 \arrow{r}{f}\arrow{d}{f_1} \& A_2 \arrow{r}{g}\arrow{d}{f_2}\& A_3\arrow{d}{f_3}\arrow[dashed]{r}{\delta} \&{}\\
I_1 \arrow{r}{p}\arrow{d}{g_1}\& I_2 \arrow{r}{q}\arrow{d}{g_2}\& I_3\arrow[dashed]{r}{0} \arrow{d}{g_3}\& {} \\
\Sigma A_1\arrow{r}{\Sigma f}\arrow[dashed]{d}{\xi_1} \& \Sigma A_2 \arrow{r}{\Sigma g}\arrow[dashed]{d}{\xi_2}\& \Sigma A_3 \arrow[dashed]{r}{\Sigma \delta}\arrow[dashed]{d}{\xi_3}\& {} \\
{} \& {} \& {} \& {} 
\end{tikzcd}
\end{equation}
in which 
$(f_1,f_2,f_3), (g_1,g_2,g_3), (f,p,\Sigma f), (g,q,\Sigma g)$ are morphisms of $\fs$-triangles.
\end{lemma}

\begin{proof}
First, for $i=1,3$ consider $\fs$-triangles $A_i\overset{f_i}{\lra}I_i\overset{g_i}{\lra}\Sigma A_i\overset{\xi_i}{\dra}$ with $I_i\in\I$. 
By \cite[Prop.~3.15]{NP19} and using that $I_{1}$ is an $\BE$-injective, we get the following commutative diagram. 
\[
\begin{tikzcd}[column sep=1.5cm, row sep=0.6cm]
A_{1}
	\arrow{r}{f}
	\arrow{d}[swap]{f_{1}}
& A_{2}
	\arrow{r}{g}
	\arrow{d}{\begin{psmallmatrix}
	a \\ g
	\end{psmallmatrix}}
& A_{3}
	\arrow[dashed]{r}{\delta}
	\arrow[equals]{d}{}
& {}
\\
I_{1} 
	\arrow{r}{\begin{psmallmatrix}
	\id_{I_{1}} \\ 0
	\end{psmallmatrix}}
	\arrow{d}{}
& I_{1} \oplus A_{3}
	\arrow{r}{\begin{psmallmatrix}
	0, \amph \id_{A_{3}}
	\end{psmallmatrix}}
	\arrow{d}{}
& A_{3}
	\arrow[dashed]{r}{0}
& {} 
\\
\Sigma A_{1}
	\arrow[equals]{r}{}
	\arrow[dashed]{d}{\xi_{1}}
& \Sigma A_{1}
	\arrow[dashed]{d}{}
& {}
& {}
\\
{}&{}&{}&{}
\end{tikzcd}
\]
In particular, the morphism 
$
\begin{psmallmatrix}
	a \\ g
\end{psmallmatrix}
	\colon A_{2} \to I_{1}\oplus A_{3}
$
is an $\fs$-inflation. 
The morphism 
$
\begin{psmallmatrix}
\id_{I_{1}} \amph 0 \\
0 \amph f_{3}
\end{psmallmatrix}
$
is a direct sum of $\fs$-inflations, so also an $\fs$-inflation. 
Thus, the composite
$
f_{2} \deff 
\begin{psmallmatrix}
\id_{I_{1}} \amph 0 \\
0 \amph f_{3}
\end{psmallmatrix} 
\circ
\begin{psmallmatrix}
	a \\ g
\end{psmallmatrix}
= \begin{psmallmatrix}
	a \\ f_{3}g
\end{psmallmatrix}
\colon A_{2} \to I_{1} \oplus I_{3} \eqqcolon I_{2}
$
is an $\fs$-inflation by \ref{ET4}, and there is an $\fs$-triangle
$
\begin{tikzcd}[column sep=0.6cm, cramped]
A_{2} 
	\arrow{r}{f_{2}}
& I_{2}
	\arrow{r}
& \Sigma A_{2}
	\arrow[dashed]{r}{\xi_{2}}
& {}
\end{tikzcd}
$
for some object $\Sigma A_{2}$. 

Since $\begin{psmallmatrix}
0, & \id_{I_{3}}
\end{psmallmatrix}f_{2} = f_{3}g$, 
the dual of \cite[Lem.~5.9]{NP19} yields 
a commutative diagram
\begin{equation}
\label{diag:horseshoe4}
\begin{tikzcd}[column sep=1.5cm, ampersand replacement=\&]
A_1 
	\arrow{r}{f} 
	\arrow[dotted]{d}{f'_{1}}
\& A_2 
	\arrow{r}{g}
	\arrow{d}{f_2}
\& A_3
	\arrow{d}{f_3}
	\arrow[dashed]{r}{\delta} 
\&{}\\
I_1
	\arrow{r}{\begin{psmallmatrix} \id_{I_{1}}\\
	 0
	 \end{psmallmatrix}}
	 \arrow[dotted]{d}{g'_{1}}
 \& I_{2}
	 \arrow{r}{\begin{psmallmatrix}0, & \id_{I_{3}}
	 \end{psmallmatrix}}
	 \arrow{d}{g_2}
\arrow[mysymbol]{ur}[description]{\circlearrowleft}
 \& I_3
	 \arrow[dashed]{r}{0} 
	 \arrow{d}{g_3}
 \& {} \\
B_{1}
	\arrow[dotted]{r}
	\arrow[dashed]{d}{\xi'_{1}}
\& \Sigma A_2 
	\arrow[dotted]{r}
	\arrow[dashed]{d}{\xi_2}
\& \Sigma A_3 
	\arrow[dashed]{r}{}
	\arrow[dashed]{d}{\xi_3}
\& {} \\
{} \& {} \& {} \& {} 
\end{tikzcd}
\end{equation}
of $\fs$-triangles and morphisms of $\fs$-triangles.
On inspection, we see that $f'_{1} = f_{1}$ so $\xi'_{1} \cong \xi_{1}$, and hence we may replace $\xi'_{1}$ by $\xi_{1}$ to obtain \eqref{diag:horseshoe1}.
\end{proof}

Now we are ready to prove Proposition~\ref{prop:Sarazola_homotopy_fiber}.

\begin{proof}[Proof of Proposition~\ref{prop:Sarazola_homotopy_fiber}]
We already know $(\C,\seq,\veq)$ is a weak (bi)Waldhausen category. 
To show $(\C,\seq,\weq)$ is a weak Waldhausen category, we need only check \ref{WW0} and \ref{WW1}.

For \ref{WW0}, it is clear that $\weq\deff\Q_{\I}\circ\K_{\N}$ contains $\Iso\C$, so we just need to check $\weq$ is closed under composition. First we note that $\K_{\N}$ and $\Q_{\I}$ are each closed under composition as $\N$ and $\I$, respectively, are extension-closed. 
Each composite
$
\begin{tikzcd}[column sep=1.3cm]
I\oplus B
    \arrow[two heads]{r}{\begin{psmallmatrix}
    0, \amph \id_{B}
    \end{psmallmatrix}}
& B
    \arrow{r}{l}
&[-0.5cm] C
\end{tikzcd}
$
in $\K_\N\circ \Q_\I$, 
where $I\in\I$ and $l\in\K_\N$, 
is also the composite 
$
\begin{tikzcd}[column sep=1.3cm, cramped]
I\oplus B
    \arrow{r}{\id_{I} \oplus l}
& I \oplus C
    \arrow[two heads]{r}{\begin{psmallmatrix}
    0, \amph \id_{C}
    \end{psmallmatrix}}
& C ,
\end{tikzcd}
$
where 
$
\id_{I} \oplus l
$
belongs to $\K_\N$ as it is a direct sum of morphisms of $\K_\N$, 
and 
$
\begin{psmallmatrix}
    0, \amph \id_{C}
    \end{psmallmatrix}
    \in \Q_{\I}
$.
Therefore, $\K_\N\circ \Q_\I\sse \Q_\I\circ\K_\N$, hence  
$\weq\circ\weq = \Q_{\I}\circ\K_{\N}\circ\Q_{\I}\circ\K_{\N}\sse \Q_{\I}\circ\Q_{\I}\circ\K_{\N}\circ\K_{\N}=\Q_\I\circ\K_\N = \weq$ and 
$\weq$ is closed under composition.

Now we will check \ref{WW1'} (see Lemma~\ref{lem:WW1'}). Consider the following solid diagram of $\fs$-triangles in which $f'a=bf$ and where $a,b\in\weq$.
\begin{equation}
\label{diag:Sarazola_homotopy_fiber_proof}
\begin{tikzcd}[row sep=0.6cm]
    A\arrow{r}{f}\arrow{d}{\sim}[swap]{a}\arrow[mysymbol]{dr}[description]{\circlearrowleft}& B\arrow{r}{g}\arrow{d}{\sim}[swap]{b} & C\arrow[dashed]{r}{\delta}\arrow[dotted]{d}{}[swap]{c'}&{}\\
    A'\arrow{r}{f'}& B'\arrow{r}{g'}& C'\arrow[dashed]{r}{\delta'}&{}
\end{tikzcd}
\end{equation}
We shall find a weak equivalence $c'\colon C \to C'$ that makes the whole diagram commute. 
Note that \ref{WW1'} does not require the triplet $(a,b,c')$ to be a morphism of $\fs$-triangles.

First, by axiom (ET3) for an extriangulated category, we may complete the pair $(a,b)$ to a morphism $(a,b,c)$ of $\fs$-triangles.
Applying Lemma~\ref{lem:horseshoe} to 
$\begin{tikzcd}[column sep=0.5cm, cramped, row sep=0.6cm]
    A 
        \arrow{r}{f}
    & B 
        \arrow{r}{g}
    & C
        \arrow[dashed]{r}{\delta}
    &{,}
\end{tikzcd}$
we have a morphism
\begin{equation}
\label{diag:Sarazola_homotopy_fiber_proof2}
\begin{tikzcd}[row sep=0.8cm, ampersand replacement=\&]
    A\arrow{r}{f}\arrow[tail]{d}{}[swap]{\gamma_A} \& B\arrow{r}{g}\arrow[tail]{d}{}[swap]{\gamma_B} \& C\arrow[dashed]{r}{\delta}\arrow[tail]{d}{}[swap]{\gamma_C} \& {}\\
    I_A\arrow{r}{}  \& I_B\arrow{r}{} \& I_C\arrow[dashed]{r}{0} \& {}
\end{tikzcd}
\end{equation}
of $\fs$-triangles 
from the upper half of the corresponding diagram \eqref{diag:horseshoe1}. 
In particular, the morphisms $\gamma_{A}, \gamma_{B},\gamma_{C}$ are $\fs$-inflations, and hence ${\begin{psmallmatrix}\gamma_A \\
a
\end{psmallmatrix}},
{\begin{psmallmatrix}\gamma_B \\
b
\end{psmallmatrix}},
{\begin{psmallmatrix}\gamma_C \\
c
\end{psmallmatrix}}$
are also $\fs$-inflations. 
Moreover, as $a,b\in\weq$, we have 
${\begin{psmallmatrix}\gamma_A \\
a
\end{psmallmatrix}},
{\begin{psmallmatrix}\gamma_B \\
b
\end{psmallmatrix}}\in\K_{\N}$
by Corollary~\ref{cor:inflation-to-I-and-weq-is-again-weq}.

Thus we have a factorization of $(a,b,c)$ as in the following commutative diagram.
\begin{equation}
\label{diag:Sarazola_homotopy_fiber_proof2-1}
\begin{tikzcd}[column sep=1.5cm, row sep=0.6cm, ampersand replacement=\&]
    A\arrow{r}{f}\arrow[tail]{d}{\sim}[swap]{\begin{psmallmatrix}\gamma_A \\
 a
 \end{psmallmatrix}} \& B\arrow{r}{g}\arrow[tail]{d}{\sim}[swap]{\begin{psmallmatrix}\gamma_B \\
 b
 \end{psmallmatrix}} \& C\arrow[dashed]{r}{\delta}\arrow[tail]{d}{}[swap]{\begin{psmallmatrix}\gamma_C \\
 c
 \end{psmallmatrix}} \& {}\\
    I_A\oplus A'\arrow{r}{}\arrow{d}{\sim}[swap]{\begin{psmallmatrix}0, & \id_{A'}\end{psmallmatrix}}  \& I_B\oplus B'\arrow{r}{}\arrow{d}{\sim}[swap]{\begin{psmallmatrix}0, & \id_{B'}\end{psmallmatrix}} \& I_C\oplus C'\arrow[dashed]{r}{0\oplus\delta'}\arrow{d}{\sim}[swap]{\begin{psmallmatrix}0, & \id_{C'}\end{psmallmatrix}} \& {}\\
    A'\arrow{r}{f'} \& B'\arrow{r}{g'} \& C'\arrow[dashed]{r}{\delta'} \& {}
\end{tikzcd}
\end{equation}
The vertical arrows in the lower half of \eqref{diag:Sarazola_homotopy_fiber_proof2-1} lie in $\Q_{\I}\sse\weq$. 
Therefore, we are reduced to showing 
$\begin{psmallmatrix}\gamma_C \\
c
\end{psmallmatrix}\in\K_{\N}$. 
In particular, 
we may assume $a,b\in\K_{\N}$ 
in \eqref{diag:Sarazola_homotopy_fiber_proof} from the very beginning in order to find an appropriate $c'\colon C\to C'$ in $\weq$ making \eqref{diag:Sarazola_homotopy_fiber_proof} commute.

Let $\delta_{1} \deff a_{*}\delta = c^{*}\delta'$. 
By \cite[the dual of Prop.~3.15]{NP19}, there is a commutative diagram
\[
\begin{tikzcd}[column sep=1.5cm, row sep=0.6cm]
A
    \arrow{r}{f}
    \arrow[tail]{d}{a}
& B
    \arrow{r}{g}
    \arrow[tail]{d}{b_{1}}
& C
    \arrow[dashed]{r}{\delta}
    \arrow[equals]{d}{}
& {}
\\
A' 
    \arrow{r}{f_{1}}
    \arrow{d}{}
    \wPO{ur}
& B_{1}
    \arrow{r}{g_{1}}
    \arrow{d}{}
& C 
    \arrow[dashed]{r}{\delta_{1}}
&[1cm] {}
\\
N_{a}
    \arrow[equals]{r}{}
    \arrow[dashed]{d}{}
& N_{a}
    \arrow[dashed]{d}{}
& {}
& {}
\\
{}&{}&{}&{}
\end{tikzcd}
\]
of $\fs$-triangles and morphisms of $\fs$-triangles. Note that since $\delta_{1} = a_{*}\delta$, the square marked (wPO) is indeed a weak pushout of $f$ along $a$. 
We know $N_{a}\in\N$ as $a\in\K_{\N}$, hence $b_{1}\in\K_{\N}$ too.
Since $f'a = bf$, there exists $p\colon B_{1} \to B'$ such that 
$b = pb_{1}$ and $f' = pf_{1}$. 

By Lemma~\ref{lem:wPO_wPB}\ref{item:wPB}, we have the following morphism of $\fs$-triangles.
\[
\begin{tikzcd}[column sep=1cm, row sep=0.6cm]
    A'
        \arrow{r}{f_{1}}
        \arrow[equals]{d}{}
    & B_{1}
        \arrow{r}{g_{1}}
        \arrow{d}[swap]{b_{2}}
        \wPB{dr}
    & C
        \arrow[dashed]{r}{\delta_{1}}
        \arrow{d}{c}
    & {}\\
    A'
        \arrow{r}{f'}
    & B'
        \arrow{r}{g'}
    & C'
        \arrow[dashed]{r}{\delta'}
    & {}
\end{tikzcd}
\]
Consider the morphism $b-b_{2}b_{1} \colon B \to B'$. 
We have 
$
(b-b_{2}b_{1})f = 0
$, 
so there exists $v\colon C \to B'$ such that
$
b - b_{2} b_{1} = vg
$, 
or 
$
b 
    = b_{2} b_{1} + vg
    = b_{2} b_{1} + v g_{1} b_{1}
    = ( b_{2} + v g_{1} ) b_{1}
$.
Set $b_{3} \deff b_{2} + v g_{1} \colon B_{1} \to B'$, so that 
$b = b_{3} b_{1}$. 

Let 
$
\begin{tikzcd}[column sep=0.7cm, cramped]
B_{1}
    \arrow[tail]{r}{\gamma_{B_{1}}}
& I_{B_{1}}
\end{tikzcd}
$
be an $\fs$-inflation with $I_{B_{1}}\in\I$ coming from $(\C,\BE,\fs)$ having enough $\BE$-injectives.
This yields another $\fs$-inflation
$
\begin{tikzcd}[column sep=1cm, cramped]
B_{1}
    \arrow[tail]{r}{
    \begin{psmallmatrix}
    \gamma_{B_{1}} \\
    b_{3}
    \end{psmallmatrix}
    }
& I_{B_{1}} \oplus B'
\end{tikzcd}
$
by \cite[Cor.~3.16]{NP19}.
Applying \ref{ET4} to the composition 
$
\begin{psmallmatrix}
\gamma_{B_{1}} \\
 b_{3}
\end{psmallmatrix}\circ b_{1}
= \begin{psmallmatrix}
\gamma_{B_{1}}b_{1} \\
 b_{3}b_{1}
\end{psmallmatrix}
= \begin{psmallmatrix}
\gamma_{B_{1}}b_{1} \\
 b
\end{psmallmatrix}
$, 
we have a commutative diagram
\begin{equation}
\label{eqn:ET4-to-composition}
\begin{tikzcd}[column sep=1.5cm, row sep=0.6cm]
B 
    \arrow{r}{b_{1}}
    \arrow[equals]{d}{}
& B_{1}
    \arrow{r}{}
    \arrow[tail]{d}{\begin{psmallmatrix}
            \gamma_{B_{1}} \\
            b_{3}
            \end{psmallmatrix}}
& N_{a}
    \arrow[dashed]{r}{}
    \arrow[tail]{d}{}
& {}
\\
B 
    \arrow{r}[swap]{\begin{psmallmatrix}
                \gamma_{B_{1}}b_{1} \\
                b
                \end{psmallmatrix}}
& I_{B''}\oplus B'
    \arrow{r}[swap]{}
        \arrow{d}{}
& N
    \arrow[dashed]{r}{}
    \arrow{d}{}
& {}
\\
& M 
    \arrow[equals]{r}{}
    \arrow[dashed]{d}{}
& M 
    \arrow[dashed]{d}{} 
& {}
\\
&{} &{} &{}
\end{tikzcd}
\end{equation}
of $\fs$-triangles, in which 
$N\in\N$ by Corollary~\ref{cor:inflation-to-I-and-weq-is-again-weq}. Hence, $M\in\N$ as $\N$ is closed under cones of $\fs$-inflations by assumption (see Setup~\ref{setup:Sarazola1}).

Next we inspect the cone of the composite $\fs$-inflation
$
\begin{psmallmatrix}
\gamma_{B_{1}} \\
b_{3}
\end{psmallmatrix} \circ f_{1}
\colon A' \to I_{B_{1}}\oplus B'$.
Notice that 
$
b_{3}f_{1} 
    = (b_{2}+vg_{1})f_{1}
    = b_{2}f_{1}
    = f'
$
since $g_{1}f_{1}$ vanishes, and thus
$
\begin{psmallmatrix}
\gamma_{B_{1}} \\
b_{3}
\end{psmallmatrix} f_{1}
    = \begin{psmallmatrix}
\gamma_{B_{1}}f_{1} \\
f'
\end{psmallmatrix}
$.
A weak pushout of $f'$ along $\gamma_{B_{1}}f_{1}$ 
yields a morphism of $\fs$-triangles as follows
\[
\begin{tikzcd}[column sep=2cm, row sep=0.6cm]
    A'\arrow{r}{f'}\arrow{d}{}[swap]{\gamma_{B_{1}}f_{1}}
    & B'\arrow{r}{g'}
    \arrow{d}
    & C'\arrow[dashed]{r}{\delta'}\arrow[equal]{d}{}
    &[-1cm]{}\\
    I_{B_{1}}\arrow{r}{\begin{psmallmatrix}
        \id_{I_{B_{1}}} \\ 0
    \end{psmallmatrix}}
    & I_{B_{1}}\oplus C'\arrow{r}{\begin{psmallmatrix}
        0, \amph \id_{C'}
    \end{psmallmatrix}}
    & C'\arrow[dashed]{r}{0}
    &[-1cm]{,}
\end{tikzcd}
\] 
where the bottom row splits as $I_{B_{1}}\in\I$.
Moreover, by \cite[Prop.~2.22]{HS20} we have an $\fs$-triangle
\begin{equation}
\label{eqn:PB-of-delta-prime-along-01}
\begin{tikzcd}[column sep=2cm]
A' 
    \arrow{r}{\begin{psmallmatrix}
    \gamma_{B_{1}}f_{1} \\
f'
\end{psmallmatrix}}
& I_{B_{1}} \oplus B'
    \arrow{r}{\begin{psmallmatrix}-\id_{I_{B_{1}}}\amph \alpha \\
 0\amph h
 \end{psmallmatrix}}
& I_{B_{1}} \oplus C'
    \arrow[dashed]{r}{\begin{psmallmatrix}
        0, \amph \id_{C'}
    \end{psmallmatrix}^{*}\delta'}
& {,}
\end{tikzcd}
\end{equation}
where 
$
\begin{psmallmatrix}
\alpha \\
h
\end{psmallmatrix}
\colon B' \to I_{B_{1}} \oplus C'
$
satisfies 
$ h = 
\begin{psmallmatrix}
0, \amph \id_{C'}
\end{psmallmatrix} 
\begin{psmallmatrix}
\alpha \\
h
\end{psmallmatrix}
= g'
$
and 
$
\begin{psmallmatrix}
\alpha \\
h
\end{psmallmatrix} f' 
    = \begin{psmallmatrix}
\id_{I_{B_{1}}} \\
0
\end{psmallmatrix} \gamma_{B_{1}}f_{1}
$.
Applying \ref{ET4} to the composite $\begin{psmallmatrix}\gamma_{B_{1}} \\
b_{3}
\end{psmallmatrix}\circ f_{1}$, we have a commutative diagram
\begin{equation}
\label{diag:Sarazola_homotopy_fiber_proof3}
\begin{tikzcd}[column sep=1.5cm, ampersand replacement=\&, row sep=0.6cm]
    A' \arrow{r}{f_{1}}\arrow[equals]{d}{}
    \& B_{1} \arrow{r}{g_{1}}\arrow{d}{\begin{psmallmatrix}\gamma_{B_{1}} \\
 b_{3}
 \end{psmallmatrix}}
    \& C \arrow[dashed]{r}{\delta_{1}}\arrow{d}{\begin{psmallmatrix}\gamma' \\
 c'
 \end{psmallmatrix}}
    \& {}
\\
    A' \arrow{r}[swap]{\begin{psmallmatrix}\gamma_{B_{1}}f_{1} \\
 f'
 \end{psmallmatrix}}
\& I_{B_{1}}\oplus B'
\arrow{r}
 \arrow{d}{}
\& I_{B_{1}}\oplus C' 
    \arrow[dashed]{r}{\delta_{2}}
    \arrow{d}{}
    \& {}
    \\
\& M \arrow[equals]{r}{}\arrow[dashed]{d}{}
    \& M \arrow[dashed]{d}{} 
    \& {}
    \\
\& {}  \&{}   \&
\end{tikzcd}
\end{equation}
of $\fs$-triangles. 
Note that
$\begin{psmallmatrix}\gamma' \\
 c'
 \end{psmallmatrix}\in\K_{\N}$ 
 as $M$ belongs to $\N$. 
 
Although it is not necessarily true that $\delta_{2} = \begin{psmallmatrix}
        0, \amph \id_{C'}
    \end{psmallmatrix}^{*}\delta'
$, 
these $\BE$-extensions are isomorphic thanks to \cite[Cor.~3.6]{NP19}. In particular, from \eqref{eqn:PB-of-delta-prime-along-01} and \eqref{diag:Sarazola_homotopy_fiber_proof3} we have a morphism of $\fs$-triangles
\[
\begin{tikzcd}[column sep=2cm, row sep=0.6cm]
A' 
    \arrow{r}{f_{1}}
    \arrow[equals]{d}{}
& B_{1} 
    \arrow{r}{g_{1}}
        \arrow{d}{\begin{psmallmatrix}\gamma_{B_{1}} \\
                b_{3}
                \end{psmallmatrix}}
& C 
    \arrow[dashed]{r}{\delta_{1}}
    \arrow{d}{\beta\begin{psmallmatrix}
                \gamma' \\
                c_{1}
                \end{psmallmatrix}}
& {}
\\
A' 
    \arrow{r}[swap]{\begin{psmallmatrix}
                    \gamma_{B_{1}}f_{1} \\
                    f'
                    \end{psmallmatrix}}
& I_{B_{1}}\oplus B'
    \arrow{r}[swap]{\begin{psmallmatrix}-\id_{I_{B_{1}}}\amph \alpha \\
 0\amph g'
 \end{psmallmatrix}}
& I_{B_{1}}\oplus C' 
    \arrow[dashed]{r}{\begin{psmallmatrix}
        0, \amph \id_{C'}
    \end{psmallmatrix}^{*}\delta'}
&{}
\end{tikzcd}
\]
for some automorphism $\beta$ of $ I_{B_{1}}\oplus C'$. 
Lastly, set 
$c'
\deff 
    \begin{psmallmatrix}
    0, \amph \id_{C'}
    \end{psmallmatrix}
    \beta
    \begin{psmallmatrix}
    \gamma'\\
    c_{1}
    \end{psmallmatrix}
    \colon C \to C'
$, which lies in $\weq=\Q_{\I}\circ\K_{\N}$ as $I_{B_{1}}\in\I$. 
One can then confirm the commutativity of \eqref{diag:Sarazola_homotopy_fiber_proof} as required, and so \ref{WW1'} holds for $(\C,\seq,\weq)$.

We have thus checked that $(\C,\seq,\weq)$ is a weak Waldhausen category 
and hence established the sequence \eqref{seq:Sarazola_homotopy_fiber} lies in $\wWald$.
Now we check $\N=\C^{\weq}$. Recall that $\C^{\weq}$ consists of the objects $C\in\C$ for which $0\to C$ lies in $\weq$. 
As $\C$ is weakly idempotent complete, the section 
$0\to N$ belongs to $\K_{\N}$ for all $N\in\N$ by e.g.\ \cite[Prop.~2.5]{BHST2022}, 
so $\N\sse\C^{\weq}$. 
Conversely, suppose we are given an object $X\in\C^{\weq}$. Then, by the definition of $\weq$, we get morphisms $0\overset{l}{\to} N\overset{r}{\to} X$ with $l\in\K_{\N}$ and $r\in\Q_{\I}$. 
In particular, $N\in\N$ and $X$ is a direct summand of $N$, which shows $X\in\N$.

To produce the right exact sequence \eqref{eqn:7.2-K0-groups}, we will use Theorem~\ref{thm:Quillen_localization1_new}. First, note that Setup~\ref{setup:Quillen_loc1} is satisfied, so we may consider the classes defined in Definition~\ref{def:mors_assoc_to_wWH_cat}. 
Since $\weq = \Q_{\I}\circ \K_{\N}$, 
it suffices to show $\K_{\N}\sse\L^{\ac} = \cof \cap \weq$ 
and 
\[
\Q_{\I} \sse 
    \T 
        = \set{ g\in\fib \cap \weq | \text{there is an } \fs\text{-conflation } A\overset{f}{\to}B\overset{g}{\to}C \text{ with } A\in\C^{\weq}=\N}.
\]
Since $\K_{\N}$ is precisely the class of $\fs$-inflations in $\weq$ by Corollary~\ref{cor:char_Ln}, 
the containment $\K_{\N}\sse\L^{\ac}$ holds. 
Since a morphism in $\Q_{\I}\sse\weq$ is a retraction with cocone in $\I\sse\N$, we also see $\Q_{\I}\sse\T$. 
Therefore, by Theorem~\ref{thm:Quillen_localization1_new} we know \eqref{seq:Sarazola_homotopy_fiber} is a localization sequence in $\wWald$ and 
\eqref{seq:homotopy_fiber1} 
is right exact. 
Finally, the weak Waldhausen categories $(\C,\seq,\veq)$ and $(\N,\seq^{\weq},\veq^{\weq})$ are the canonical wbiWET structures on $\C$ and $\N$, respectively, arising from the extriangulated structure $(\C,\BE,\fs)$. 
Hence, the isomorphisms $K_{0}(\N,\BE|_{\N},\fs|_{\N})\cong K_{0}(\N,\seq^{\weq},\veq^{\weq})$ and $K_{0}(\C,\BE,\fs)\cong K_{0}(\C,\seq,\veq)$  (see Example~\ref{ex:WET1}) imply 
the sequence \eqref{eqn:7.2-K0-groups} of Grothendieck groups is right exact, as claimed.
\end{proof}

\subsection{A connection to extriangulated localization}
\label{sec:connection_to_extri_loc}

To clarify how the localization sequence \eqref{seq:Sarazola_homotopy_fiber} relates to extriangulated localization, we investigate the situation when $\N$ is thick in $(\C,\BE,\fs)$. 
We will see that, under Setup~\ref{setup:Sarazola2} below, our way of considering an extriangulated localization as a localization sequence of wWET categories as in Section~\ref{sec:wWET-cats} is equivalent to following the construction motivated by \cite{Sar20} in Section~\ref{sec:Sarazola_Wald_cat} above. 
Assume the following in the rest.

\begin{setup}
\label{setup:Sarazola2}
In addition to Setup~\ref{setup:Sarazola1}, we assume 
$\N$ is thick in $(\C,\BE,\fs)$, 
$\Sn$ is saturated and $\overline{\Sn}$ satisfies conditions 
\textup{(MR1)}--\textup{(MR4)} of \cite[p.~343]{NOS22}. 
Thus, there is the following extriangulated localization sequence (i.e.\ exact sequence) in $\ET$ (see Theorem~\ref{thm:NOS}).
\begin{equation}
\label{seq:ET_loc_from_Sarazola}
\begin{tikzcd}[column sep=1.2cm]
(\N,\BE|_\N,\fs|_\N)
    \arrow{r}{(\inc,\iota)}
& (\C,\BE,\fs)
\arrow{r}{(Q,\mu)} 
& (\C/\N,\widetilde{\BE},\widetilde{\fs})
\end{tikzcd}
\end{equation}
\end{setup}

Thanks to Proposition~\ref{prop:ex_seq_wWald_from_NOS},
the exact sequence \eqref{seq:ET_loc_from_Sarazola} in $\ET$ induces the following localization sequence in $\wWald$ where $\weq' \deff \Sn$.
\begin{equation}
\label{seq:exact_sequence_from_Sarazola}
\begin{tikzcd}
    (\N,\seq^{\weq'},\veq^{\weq'})
        \arrow[hook]{r}{\inc}
    &(\C,\seq,\veq)
        \arrow{r}{\id_{\C}} 
    &(\C,\seq,\weq')
\end{tikzcd}
\end{equation}

\begin{proposition}\label{lem:toward_higher_K}
The localization sequences \eqref{seq:Sarazola_homotopy_fiber} and \eqref{seq:exact_sequence_from_Sarazola} coincide.
In particular, we have a group isomorphism 
$K_{0}(\C,\seq,\weq)
= K_{0}(\C,\seq,\weq')
\overset{\cong}{\lra} K_{0}(\C/\N,\widetilde{\BE},\widetilde{\fs})$.
\end{proposition}

\begin{proof}
We have to only show that $\weq' = \Sn$ coincides with $\weq = \Q_{\I}\circ \K_{\N}$. 
We immediately see that $\K_{\N}\sse\L$ and $\Q_{\I}\sse\R$, so 
$\weq'$ being the closure of $\L \cup \R$ with respect to the composition (see Definition~\ref{def:Sn_from_thick}) implies $\weq\sse\weq'$. 
To show the converse, take a morphism $A\overset{f}{\to}B$ in $\weq'$.
Since $\C$ has enough $\BE$-injectives, there is an $\fs$-inflation $A\overset{\gamma_A}{\lra}I_A$ with $I_A\in\I\sse\N$ which induces an $\fs$-conflation 
$\begin{tikzcd}[cramped]
A
	\arrow{r}{\begin{psmallmatrix}\gamma_A \\
				f
				\end{psmallmatrix}}
& I_A\oplus B
	\arrow{r}
& N.
\end{tikzcd}
$
Since $Q\begin{psmallmatrix}\gamma_A \\
f
\end{psmallmatrix} \cong Qf$ is an isomorphism in $\C/\N$, 
the morphism $\begin{psmallmatrix}\gamma_A \\
f
\end{psmallmatrix}$ belongs to $\weq' = \Sn$ as $\Sn$ is saturated.
Thus, we have checked $N\in\Ker Q=\N$ and $\begin{psmallmatrix}\gamma_A \\
f
\end{psmallmatrix}\in\K_{\N}$.
As 
$f=\begin{psmallmatrix}0, & \id_{B}
\end{psmallmatrix}\begin{psmallmatrix}\gamma_A \\
f
\end{psmallmatrix}$, we have thus verified that $\weq'\sse\Q_{\I}\circ \K_{\N}=\weq$.

Therefore, we have 
$(\C,\seq,\weq)
= (\C,\seq,\weq')$. 
Finally, from the proof of Corollary~\ref{cor:ES_via_Waldhausen} we know that there is the group isomorphism 
$K_{0}(Q) \colon K_{0}(\C,\seq,\weq')\overset{\cong}{\lra}  K_{0}(\C/\N,\widetilde{\BE},\widetilde{\fs})$. 
\end{proof}

\medskip
\noindent
{\bf Acknowledgements.}
The authors would like to thank Marcel B\"{o}kstedt, Hiroyuki Nakaoka and Katsunori Takahashi for their time, valuable discussions and comments.

Y.\ O.\ is supported by JSPS KAKENHI (grant JP22K13893).
A.\ S.\ is supported by: the Danish National Research Foundation (grant DNRF156); the Independent Research Fund Denmark (grant 1026-00050B); and the Aarhus University Research Foundation (grant AUFF-F-2020-7-16).


%
\end{document}